\documentclass[reqno]{amsart}

\pdfoutput=1

\usepackage{amsfonts,amsmath,amsthm,amsxtra,amssymb,mathrsfs,dsfont}

\usepackage{tikz-cd}

\usepackage[square,numbers]{natbib}

\usepackage{graphicx}  

\usepackage{subfigure}

\usepackage{fancybox}

\usepackage[colorlinks = true, citecolor = blue, urlcolor = blue]{hyperref}

\newtheorem{theorem}{Theorem}[section]
\newtheorem{lemma}[theorem]{Lemma}
\newtheorem{proposition}[theorem]{Proposition}
\newtheorem{corollary}[theorem]{Corollary}
\theoremstyle{definition}
\newtheorem{definition}[theorem]{Definition}
\newtheorem{remark}[theorem]{Remark}

\def\gor#1{\widetilde{#1}}
\def\tech#1{\widehat{#1}}
\def\bar#1{\overline{#1}}

\def\h(#1,#2){\mbox{Hom}\left(#1,#2\right)}

\def\t(#1,#2){\mbox{Tor}\left(#1,#2\right)}
\def\e(#1,#2){\mbox{Ext}\left(#1,#2\right)}

\def\CP{\mathbb{C}\mathbf{P}}
\def\RP{\mathbb{R}\mathbf{P}}
\def\FP{\mathbb{F}\mathbf{P}}

\def\C{\mathbb{C}}
\def\F{\mathbb{F}}

\def\N{\mathbb{N}}

\def\R{\mathbb{R}}

\def\X{\mathbb{X}}

\def\Z{\mathbb{Z}}

\def\uu{\mathbf{u}}
\def\vv{\mathbf{v}}
\def\ww{\mathbf{w}}
\def\xx{\mathbf{x}}
\def\yy{\mathbf{y}}

\def\XX{\mathbf{X}}
\def\YY{\mathbf{Y}}

\def\<{\langle}
\def\>{\rangle}

\newcommand{\Mat}[1]{{\begin{bmatrix}#1\end{bmatrix}}}
\newcommand{\ppart}[1]{\left|#1 \right|_{+}}

\DeclareMathOperator*{\argmin}{\arg\!\min\;}
\DeclareMathOperator*{\argmax}{\arg\!\max\;}

\title[Projective Coordinates]{Multiscale Projective Coordinates
via Persistent Cohomology of Sparse  Filtrations\thanks{This work was partially supported by the NSF under grant DMS-1622301 and  DARPA under grant HR0011-16-2-003}
}

\author[Jose Perea]{Jose A. Perea}

\address{Department of Computational Mathematics, Science \& Engineering, Department of Mathematics,
         Michigan State University, East Lansing,
         MI, USA.}
\email{joperea@math.msu.edu}



\begin{document}

\bibliographystyle{abbrvnat}
\begin{abstract}
We present in this paper a framework
which leverages the underlying topology of a data set,
in order to produce appropriate coordinate representations.
In particular, we show how to construct maps to
real and complex projective spaces, given appropriate
persistent cohomology classes.
An initial map is obtained in two steps: First, the persistent cohomology of a sparse filtration is used to compute systems of
 transition functions for  (real and complex) line
bundles over neighborhoods of the data.
Next, the transition functions are used to produce explicit classifying maps  for the induced  bundles.
A framework for dimensionality reduction in projective space
(Principal Projective Components) is also developed,
aimed at decreasing the target dimension of the original map.
Several examples are provided as well as
theorems addressing  choices in the construction.
\end{abstract}

\maketitle

\section{Introduction}

Algebraic topology has emerged in the last decade as a powerful framework for analyzing complex high-dimensional data \cite{carlsson2009topology, carlsson2014topological}.
In this setting, a data set is interpreted as a finite subset $X$   of an ambient metric space $(\mathbb{M}, \mathbf{d})$;
in practice  $\mathbb{M}$ is usually
   Euclidean space, a manifold or  a simplicial complex.
If $X$ has been sampled from/around a ``continuous'' object $\mathbb{X} \subset \mathbb{M}$,
the topology inference problem asks whether topological features of $\mathbb{X}$ can be inferred from $X$.
On the one hand, this
is relevant because several data science questions are reinterpretations of  topological tests.
To name a few: clustering is akin to finding the  connected components of a space
\cite{carlsson2010characterization, carlsson2013classifying};
coverage in a sensor network relates to the existence of holes \cite{ghrist2005coverage, de2007coverage}; periodicity and quasiperiodicity are linked to nontrivial 1-cycles in time delay reconstructions \cite{perea2015sliding, perea2016persistent}.

On the other hand, concrete descriptions of the underlying space $\mathbb{X}$ --- e.g. via equations or as a quotient space --- yield  (geometric) models for the data $X$, which can then
be used for simulation,  prediction and hypothesis testing
\cite{martin2010topology,perea2014klein}.
When determining low-dimensional representations
for an abstract data set,
the presence of nontrivial  topology heavily
constraints the dimension and type
of appropriate model spaces.
The simplest example of this phenomenon is perhaps the circle, which is
intrinsically 1-dimensional, but cannot be recovered on the real line without
considerable distortion.
Viral evolution provides another example;
evolutionary trees are often inadequate models, as
horizontal recombination introduces cycles
\cite{chan2013topology}.
The main goal of this paper is to show
that one can leverage knowledge of the topology
underlying a data set, e.g. determined via topological inference, to produce appropriate low-dimensional coordinates.
In particular, we show how the persistent cohomology of a data set can be used to compute economic representations in
real (or complex) projective spaces.
These coordinates capture the topological obstructions
which prevent the data from being recovered in
low-dimensional Euclidean spaces and, correspondingly,
yield appropriate representations.

Persistent (co)homology is one avenue to address the topology inference problem
\cite{weinberger2011persistent, de2011dualities}; it provides a multiscale description of topological features,  and satisfies several inference theorems
\cite{niyogi2008finding, chazal2009sampling, cohen2007stability}.
Going from a persistent homology computation to actionable knowledge about the initial data set,  is in general highly nontrivial.
One would like to determine: how are the persistent homology features reflected on the data?
if the homology suggests a candidate underlying space $\mathbb{X}$, what is a concrete realization (e.g., via equations or as a quotient)? how does the data fit on/around the realization?
At least three approaches have been proposed in the literature
to address these questions: localization, homological coordinatization and circular coordinates.

For a relevant homology class, the idea behind localization is to find an appropriate representative cycle.
This strategy is successful in low dimensions \cite{dey2010approximating,dey2011optimal}, but in  general  even reasonable heuristics lead to NP-hard problems which are NP-hard to approximate \cite{chen2011hardness}.
Homological coordinatization \cite{tausz2011homological}
attempts to map a simplicial complex on the data to
a  simplicial complex with prescribed homology
--- a model suggested by a persistent homology computation.
The  map is selected
by examining the set of chain homotopy classes of chain maps between
the complex on the data and the model.
Selecting an appropriate representative chain map, however, involves combinatorial
optimizations which are difficult to solve in practice.
Circular coordinates  \cite{de2011persistent}
leverages the following fact:
there is a bijection $H^1(B;\Z) \cong \big[B,S^1\big]$
between the 1-dimensional $\Z$-cohomology of a topological space $B$,
and the set of homotopy classes of maps from $B$ to the  circle $S^1$.
When $B$ is a simplicial complex with the data as vertex set,
this observation is used to turn 1-dimensional (persistent) $\Z/p$-cohomology classes for  appropriate primes $p$, into circle-valued functions on the data.

The circular coordinates approach has been used successfully ---
for instance to parameterize periodic dynamics \cite{de2012topological} ---
and  is  part of a bigger picture:
if $G$ is an abelian group, $n\in \Z_{\geq 0}$ and
$K(G,n)$  is an Eilenberg-MacLane space
(i.e. its $n$-th homotopy group is $G$ and  the others are trivial),
then \cite[Chapter 22, sec. 2]{may1999concise}
\begin{equation}\label{eq:BrownRep}
  H^n(B; G) \cong \big[B, K(G,n)\big]
\end{equation}
In particular $K(\Z,1) $ is homotopy equivalent to $S^1$; that is,
 $K(\Z,1) \simeq S^1$.
Since the bijection in (\ref{eq:BrownRep}) can be chosen so that it
commutes with morphisms induced by
maps of spaces (i.e.  natural),
using maps to other Eilenberg-MacLane spaces emerges as an avenue for
interpreting persistent cohomology computations,
and more importantly, to produce coordinates which leverage knowledge
of the underlying topology.
One quickly runs into difficulties as the next candidate spaces
$K(\Z/2,1) \simeq \RP^\infty$,
$K(\Z/p,1) \simeq S^\infty / (\Z/p)$ (action via scalar multiplication
on $S^\infty \subset \C^\infty$ by the $p$-th roots of unity)  and
$K(\Z,2) \simeq \CP^\infty$
are  infinite.
The purpose of this paper is to address the cases $\RP^\infty$ and $\CP^\infty$.

\subsection{Approach}
We use the fact that if $\F$ is $\R$ or $\C$,
then $\FP^\infty$ is the
Grassmannian  of 1-planes in $\F^\infty$ and --- for a topological space $B$ --- $\big[B,\FP^\infty\big]$ can be naturally identified with
the set of isomorphism classes of $\F$-line bundles over $B$ (Theorem \ref{thm:BundleClassification}).
The isomorphism type of an $\R$-line bundle over $B$
is  uniquely determined by its first Stiefel-Whitney
class $w_1 \in H^1(B;\Z/2)$,
and  the isomorphism type of a $\C$-line bundle over $B$
is  classified by its first Chern class
$c_1 \in H^2(B;\Z)$.
These classes can be identified
as elements of  appropriate sheaf cohomology groups,
and  a classifying
map $f: B \longrightarrow \FP^\infty$ can be
described explicitly from a \v{C}ech cocycle representative (Theorem \ref{thm:ClassifyingMapFormula1}).
This  links  cohomology to $\FP^\infty$-coordinates.
For the case of point cloud data  $X \subset (\mathbb{M}, \mathbf{d})$, $B$ will be an open neighborhood of $X$ in $\mathbb{M}$,
and we use the persistent cohomology
of a sparse filtration on $X$ to generate appropriate \v{C}ech cocycle representatives  (Theorem \ref{thm:ClassifyingFormulaSparse}).
Evaluating the resulting classifying map $f$ on the data $X$
produces another point cloud $f(X)\subset \FP^n$.
We develop in
Section \ref{sec:ProjectiveCoordinates}
a dimensionality reduction procedure for subsets of $\FP^n$,
referred to as Principal Projective Component Analysis.
When this methodology is
applied to the point cloud $f(X)\subset \FP^n$, it produces a sequence of
 projections
$P_k : f(X)\longrightarrow \FP^k $, $k = 1,\ldots, n$,
each of which attempts to minimize an appropriate notion
of distortion.
The sets $P_k \circ f(X)$ are the $\FP^k$-coordinates of
$X$, for the \v{C}ech cocycle giving rise to $f$.

\subsection{Organization}
We end this introduction in Subsection \ref{subsec:motivation}
with a motivating data example.
Section \ref{sec:preliminaries} is devoted to  the theoretical preliminaries needed in later sections of the paper;
in particular, we provide a terse introduction to
vector bundles, the \v{C}ech cohomology of (pre)sheaves
and the persistent cohomology of filtered complexes.
In Sections \ref{sec:ClassifyingMaps} and
\ref{sec:TransitionFuncFromSimplicialCohomology}
we prove the main theoretical results underlying our method.
In particular, we show how singular cohomology classes
 yield explicit and computable maps to real (and complex) projective space.
Sections \ref{sec:ProjectiveCoordinates}, \ref{sec:ChoosingCocycles} and \ref{sec:TransitionFuncFromPersistentCohomology}
deal with   computational aspects.
Specifically, section \ref{sec:ProjectiveCoordinates}
develops Principal Projectives Components,
a dimensionality reduction method in projective space,
used to lower the target dimension of the
maps developed in Sections \ref{sec:ClassifyingMaps} and \ref{sec:TransitionFuncFromSimplicialCohomology}.
Section \ref{sec:ChoosingCocycles} deals with the problem
of choosing cocycle representatives; we present
theorems and examples which guide these choices.
We end the paper in Section \ref{sec:TransitionFuncFromPersistentCohomology} showing
how persistent cohomology can be used to make the approach
practical for data sets with low intrinsic dimension.

\subsection{Motivation and  road map}
\label{subsec:motivation}
Let us illustrate some of the  ideas we will
develop in this paper via an example.
To this end,
let $\mathbb{X}$ be the collection of intensity-centered $7\times 7$
grey-scale images depicting a   line segment of fixed width,
as show in
Figure \ref{fig:exemplaryPatches}.

\begin{figure}[!ht]
\centering
\includegraphics[scale = 0.35]{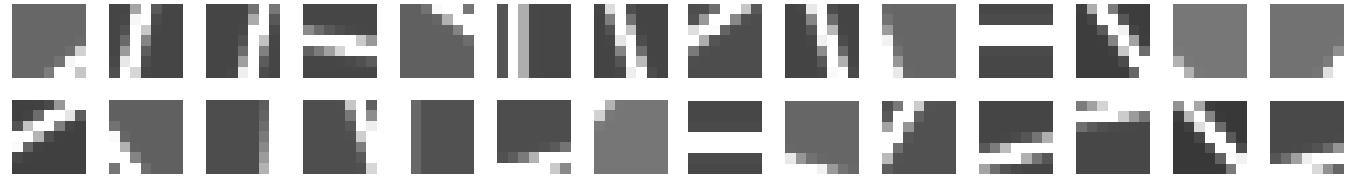}
\caption{Typical elements in  $\mathbb{X}$}
\label{fig:exemplaryPatches}
\end{figure}

By intensity-centered we mean that if the pixel values of
$x\in \mathbb{X}$ are encoded as real numbers between
-1 (black) and 1 (white), then the mean pixel intensity of $x$ is zero.
We  regard $\mathbb{X}$ as a subset of $\R^{49}$
by representing each  image
as a vector of pixel intensities, and
endow it  with the distance inherited from $\R^{49}$.
A data set $X\subset \mathbb{X}$ is generated by sampling $ 1,682$ points.
The thing to notice is that even when the ambient space for $\mathbb{X}$ is $\R^{49}$, the
intrinsic dimensionality is low. Indeed, each image
can be generated from two numbers: the angle of  the line
segment with the horizontal, and the signed distance from the segment
to the center of the patch.
This suggests that $\mathbb{X}$ is locally  2-dimensional.

Principal Component Analysis (PCA)
\cite{jolliffe2002principal}
and ISOMAP \cite{tenenbaum2000global}
are standard tools to produce low-dimensional representations for data;
let us see if we can use them to recover an appropriate 2-dimensional representation
for $X$.
Given the first $k$ principal components of $X$, calculated with PCA,
their linear span $V_k$ is interpreted as the $k$-dimensional linear
space which best approximates $X$.
One can calculate the \emph{residual variance} of this approximation,
by computing the mean-squared distance from $X$ to $V_k$.
Similarly, the fraction of variance from $X$ explained by $V_k$ is equal to difference between the variance of $X$ and the residual variance, divided
by the variance of $X$.
A similar notion can be defined for ISOMAP.
We show in Figure \ref{fig:PCA and ISOMAP dim_VS_variance} the fraction of variance, from $X$, recovered by PCA and ISOMAP\footnote{ using a $7$-th nearest neighbor graph.}

\begin{figure}[!ht]
\centering
\subfigure[PCA]{
\includegraphics[scale = 0.339]{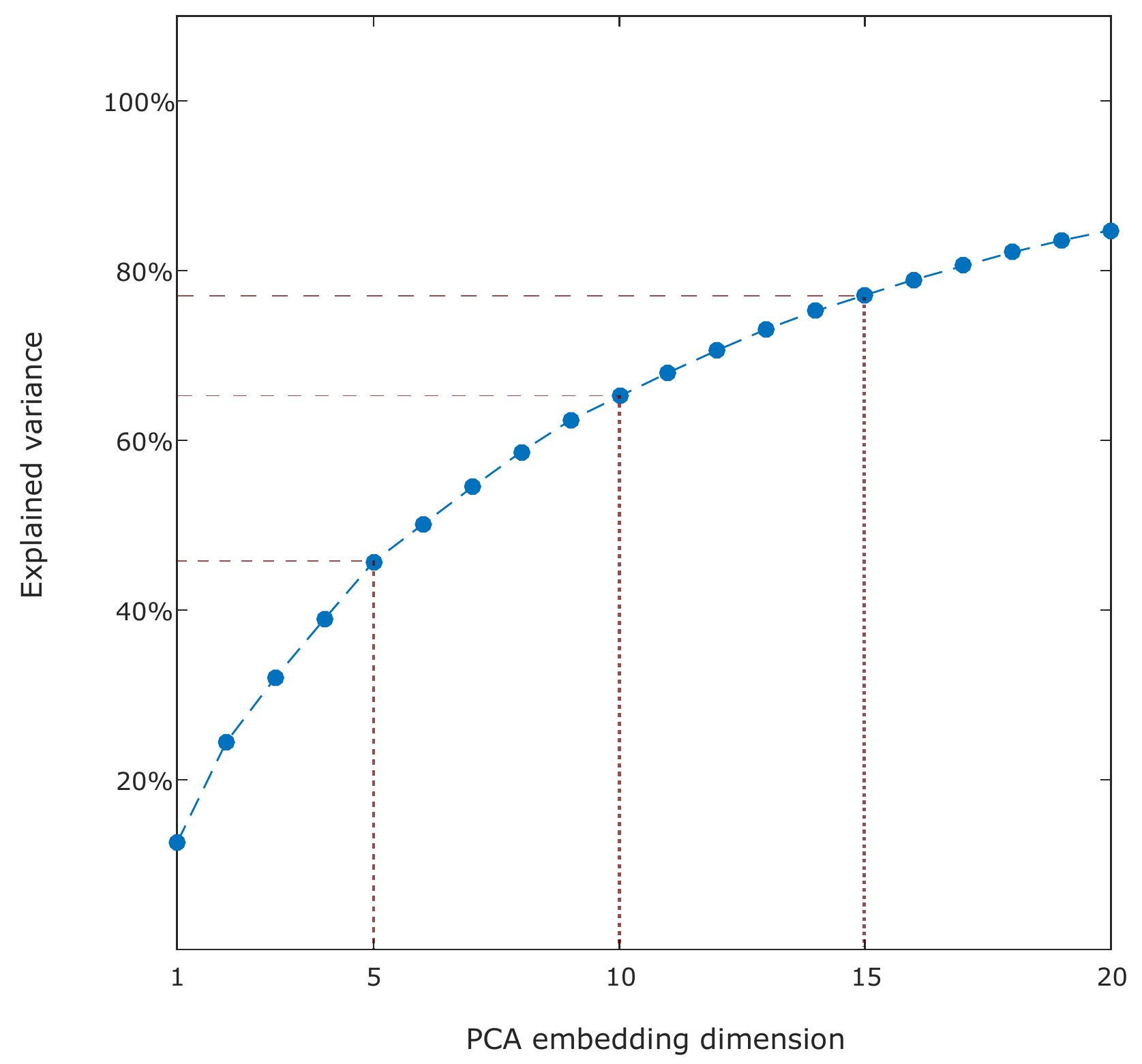}
}
\subfigure[ISOMAP]{
\includegraphics[scale = 0.34]{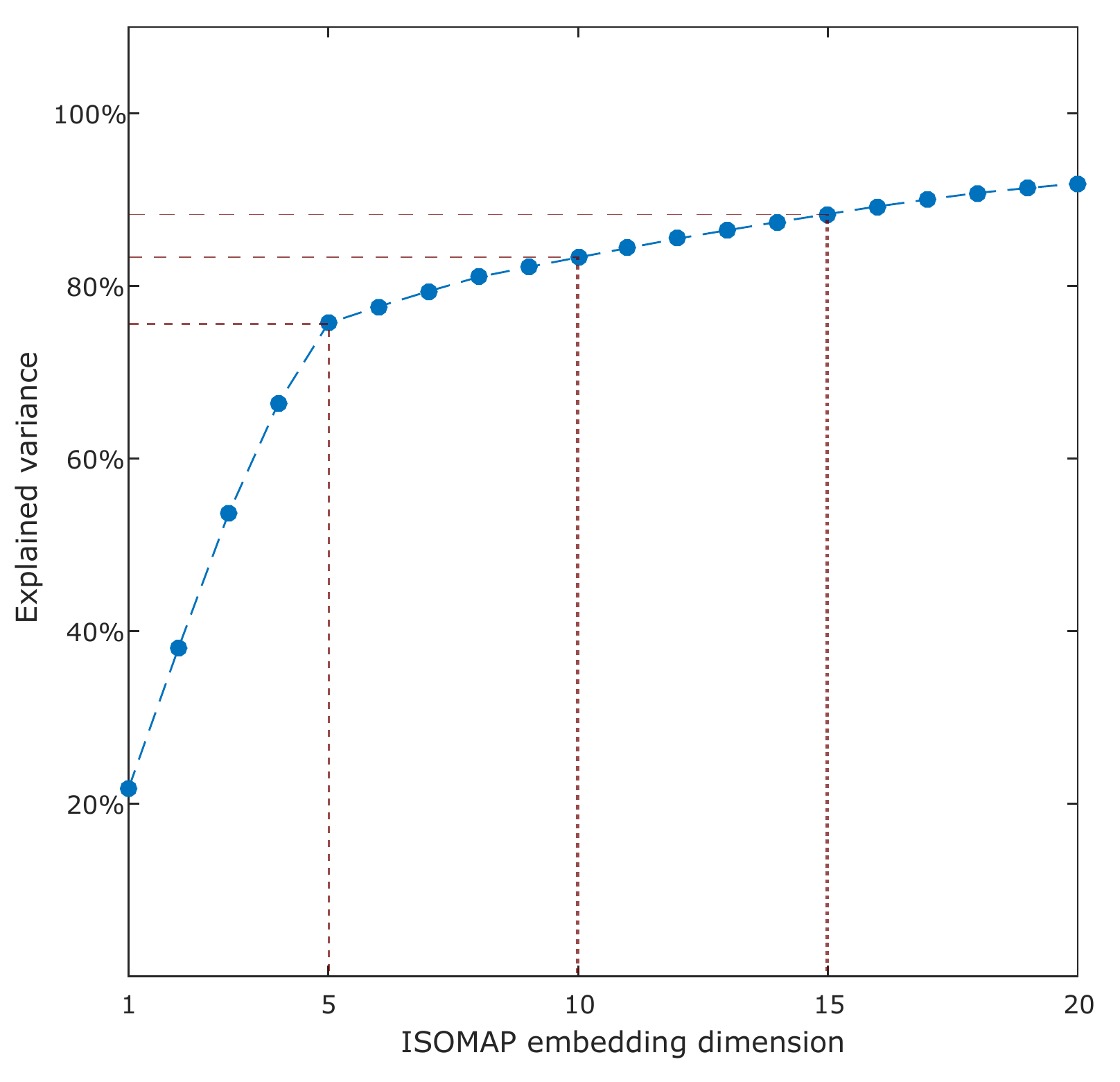}
}
\label{fig:PCA and ISOMAP dim_VS_variance}
\caption{Explained variance, from $X$, versus embedding dimension.}
\end{figure}

Given the low intrinsic dimensionality of $\X$,
it follows from the PCA plot that the original embedding
$\mathbb{X} \hookrightarrow \R^{49}$ is highly non-linear.
Moreover, the ISOMAP plot  implies that  even after accounting for
the way in which $\mathbb{X}$ sits in $\R^{49}$, the data
has intrinsic complexity that prevents it from being recovered faithfully
in $\R^2$ or $\R^3$. That is, the data is locally simple
(e.g. each $x\in X$ is described by angle  and displacement)
but
globally complex.
One possible source of said complexity is whether or not $\mathbb{X}$ is orientable; this is a topological obstruction to low-dimensional Euclidean embeddings.
Using the observation that $\mathbb{X}$ is a manifold,
its  orientability  can be determined from $X$ as we describe next.
First, we construct a covering $\{B_r\}_{r= 0}^n$ of $X$.
From now on we will simplify the set notation
$\{a_\alpha\}_{\alpha \in \Lambda}$
to $\{a_\alpha\}$ in the cases where the indexing set $\Lambda$
can be inferred from the context.
Next,
we apply Multi-Dimensional Scaling (MDS) \cite{kruskal1964multidimensional} on each $B_r$
to get local Euclidean coordinates
and, finally, we  compute the determinant $\omega_{rt} = \pm 1$
associated to the
 change of local  Euclidean coordinates
on  each $B_r \cap B_t$.
If there is global agreement of local orientations (e.g., $\omega_{rt}$ =1 always), or local orientations can be reversed in the appropriate $B_r$'s
so that the result is globally consistent (i.e.
$\{\omega_{rt}\}$ is a coboundary: there exist $\nu_r$'s,  with $\nu_r = \pm 1$, so that $\omega_{rt} = \nu_t/\nu_r$ for all
$r,t=1,\ldots, n$), then $\mathbb{X}$
would be deemed  orientable.

The cover for $X$ will be a collection of open balls centered at landmark data points
selected  through  \verb"maxmin" (also known as farthest point) sampling. That is, first one chooses an arbitrary landmark $\ell_0 \in X$, and if
$\ell_0,\ldots, \ell_r$ have been determined,
then $\ell_{r+1} \in X$ is given by
\[
\ell_{r+1} = \argmax_{x\in X} \big(\min\{\mathbf{d}(x, \ell_0),\ldots, \mathbf{d}(x , \ell_r)\}\big).
\]
Here $\mathbf{d}$ is the geodesic distance estimate from the ISOMAP calculation.
A finite number of steps of \verb"maxmin" sampling results in a landmark set $\{\ell_0,\ldots, \ell_n\}$ which tends to be
well-distributed and well-separated across the data. For the current  example we used $n=14$.
Let\footnote{Determined experimentally using a persistent cohomology computation} $\epsilon_0 = \cdots = \epsilon_n = 9.3$ and
\[
\mathcal{B} =\{B_r\}\;\; \mbox{ where } \;\; B_r = \{x\in X : \mathbf{d}(x,\ell_r) < \epsilon_r\}
\]
To put the radii $\epsilon_r$ in perspective, the distance between  distinct $\ell_r$'s ranges from 4.5 to 14.6, and the mean pairwise distance is
9.5.

Let us now show how to calculate the determinant of the change of
local coordinates. If $B_r \cap B_t \neq \emptyset$,  let
\[
f_r : B_r \longrightarrow \R^2\;\; \mbox{ and } \;\; f_t : B_t \longrightarrow \R^2
\]
be the functions obtained    from applying MDS on
$\left(B_r , \mathbf{d}\big|_{B_r}\right)$ and $\left(B_t , \mathbf{d}\big|_{B_t}\right)$, respectively.
If $\mathsf{O}(2)$ denotes the set of orthogonal $2\times 2$ real matrices, then the solution to the orthogonal Procrustes problem
\[
\left(\Omega_{rt}, \vv_{rt}\right) \;\;
=
\argmin_{\Omega \in \mathsf{O}(2), \; \vv \in \R^2} \; \sum_{x\in B_r \cap B_t}
\big\|
f_t(x) -\big( \Omega\cdot f_r(x) + \vv \big)
\big\|^2
\]
computed  following \cite{schonemann1966generalized},
yields the best linear approximation to an isometric change of local
ISOMAP coordinates.
We let $\omega_{rt}:= \det(\Omega_{rt})$.
In summary, we have constructed a finite covering $\mathcal{B} = \{B_r\}$ for  $X$
and a collection of numbers $\omega_{rt}= \pm 1$
 which, at least for this example,  satisfy the \textbf{cocycle condition}:
for all $ 0 \leq r \leq n$  we have
$\omega_{rr} =1$
and  if $B_r \cap B_s \cap B_t \neq \emptyset$, then
$\omega_{rs}\cdot \omega_{st} = \omega_{rt}$.
In particular, a cohomology computation shows that $\{\omega_{rt}\}$
is not a coboundary and hence
$\mathbb{X}$ is estimated to be non-orientable.

What we will see now, and throughout the paper,
is that this type of cohomological feature
can be further leveraged to produce useful coordinates
for the data.
Indeed (Theorem \ref{thm:ClassifyingMapFormula1} and Corollary \ref{coro:ClassifyingMapFormula1c}):
\\[.2cm]
\noindent \textbf{Theorem.} For $\lambda\in \R$ let $\ppart{\lambda} = \max\{\lambda,0\} $,
let $\F$ be either $\R$ or $\C$, and let $\F^\times = \F\smallsetminus\{\mathbf{0}\}$.
For a metric space $(\mathbb{M},\mathbf{d})$  let
$\{\ell_0,\ldots, \ell_n\}\subset \mathbb{M}$
and fix positive real numbers
$\epsilon_0,\ldots, \epsilon_n$.
If we let $\mathcal{B} = \{B_r\}$ with
$B_r =\{b \in \mathbb{M} \,: \, \mathbf{d}(b,\ell_r) < \epsilon_r\}$,
$B = \bigcup \mathcal{B}$, and there is a collection
$\omega = \{\omega_{rt} : B_r \cap B_t \longrightarrow \F^\times\}$
 of continuous maps
satisfying
the cocycle condition (\ref{eq:CocycleCondition}),
then
$f_\omega: B \longrightarrow\mathbb{F} \mathbf{P}^n$
given in homogeneous coordinates by
\begin{equation}
f_\omega(b) =
\bigg[\omega_{0j}(b)\cdot
\ppart{\epsilon_0 - \mathbf{d}(b,\ell_0)}
:\cdots:
\omega_{nj}(b)\cdot
\ppart{\epsilon_n - \mathbf{d}(b, \ell_n)}
\bigg] \;\;\; ,  \; \; \; b\in B_j
\label{eq:classMapFormula}
\end{equation}
is well-defined (i.e. the value $f_\omega(b)$ is independent of the $j$ for which $b\in B_j$)
 and
classifies the  $\F$-line bundle on $B$
induced by $\big(\mathcal{B}, \omega\big)$.
\\[.2cm]
The preliminaries needed to understand this theorem will be covered in Section \ref{sec:preliminaries}, and
we will devote
Section \ref{sec:ClassifyingMaps} to proving it.
The map $f_\omega$  encodes  in a global manner the local interactions
captured by $\omega$, and the explicit formula allows us to map our  data
set  $X\subset \R^{49}$ into $\RP^{14}$
 using the computed landmarks and determinants of local changes of coordinates.
If we compute the principal projective components for $f_\omega(X) \subset \RP^{14}$ (this will be developed in Section \ref{sec:ProjectiveCoordinates} as a natural extension to PCA in Euclidean space and of  Principal Nested Spheres Analysis
\cite{jung2012analysis}),
the profile of recovered variance shown in Figure \ref{fig:projdDim_vs_variance} emerges:

\begin{figure}[!ht]
\centering
\includegraphics[scale = 0.35]{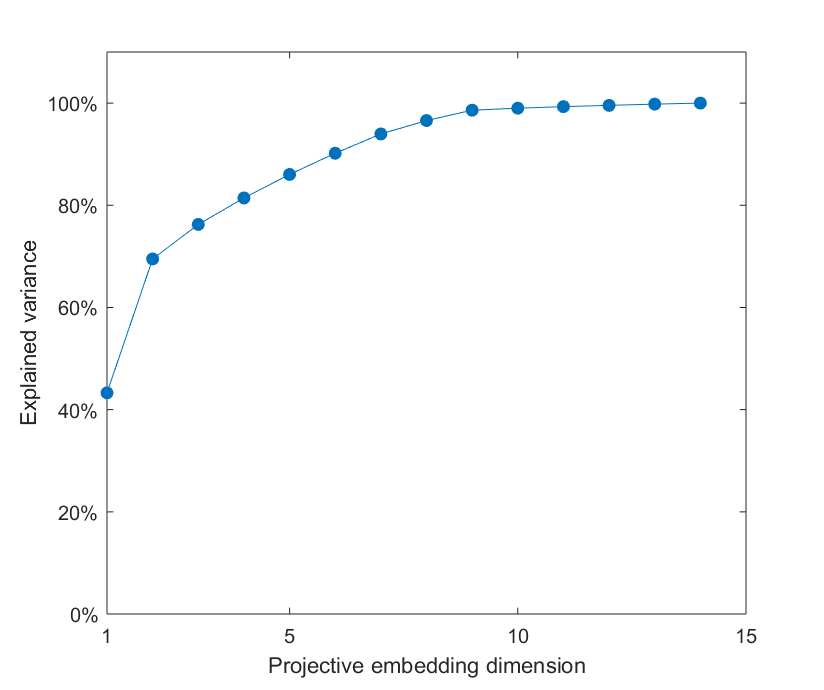}
\caption{Recovered variance, from $f_\omega(X) \subset \RP^{14}$, when projected onto
principal projective subspace of given dimension.}
\label{fig:projdDim_vs_variance}
\end{figure}

From this plot we conclude that 2-dimensional projective space  provides an appropriate reduction for $f_\omega(X)$.
We show in Figure \ref{fig:projCoordPatches} said representation; that is,
each image is placed in the $\RP^2$ coordinate computed via
principal projective  component analysis on $f_\omega(X)$.

\begin{figure}[!ht]
\centering
\includegraphics[scale = 0.6]{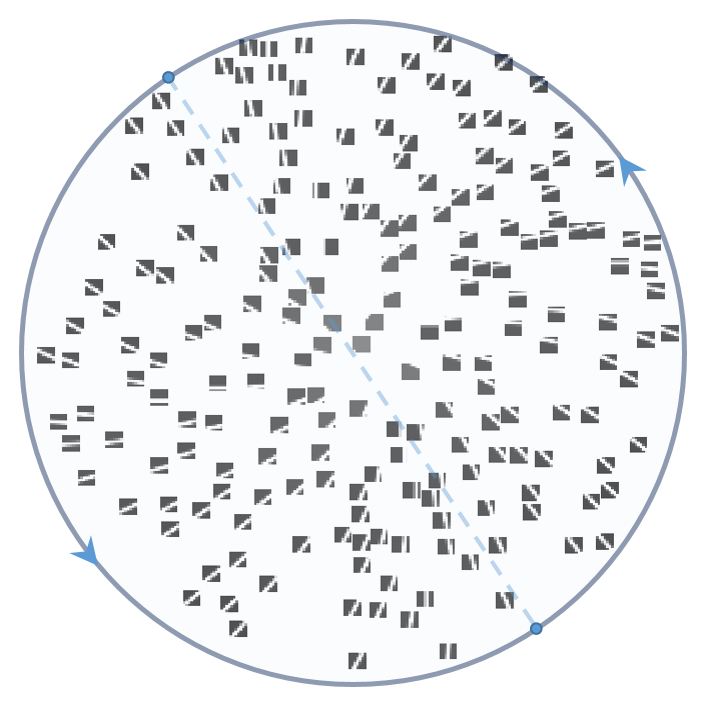}
\caption{Some elements from $X$ placed on their computed $\RP^2$-coordinates.}
\label{fig:projCoordPatches}
\end{figure}

As the figure shows, the resulting coordinates recover
the variables which we identified as describing points in $\mathbb{X}$: the
radial coordinate in $\RP^2$ corresponds to distance from the line segment to the
center of the patch, and the angular coordinate captures
 orientation.
Also, it indicates how $\RP^2 \cong  \mathbb{X}$ parameterizes the original data set $X$.

Though the strategy employed in this example (i.e. local MDS + determinant of local change of  coordinates) was successful,
 one cannot assume in general that the data under analysis has been sampled from/around a manifold.
That said, the result from formula (\ref{eq:classMapFormula}) only requires a covering
$\mathcal{B}$ via open  balls, and a collection of $\F^\times$-valued
continuous functions $\omega = \{\omega_{rt}\}$ satisfying the cocycle condition. Given a finite  subset $X$ of an ambient metric space $(\mathbb{M},\mathbf{d})$, one can always use {\tt maxmin} sampling to produce a covering.
We will show that any
1-dimensional (resp. 2-dimensional) $\Z/2$-cocyle (resp. $\Z$-cocycle)
of the nerve complex $\mathcal{N}(\mathcal{B})$, for $\F = \R$ (resp. $\F = \C$), yields one such $\omega$ (Proposition \ref{prop:Simplicial2CechReal} and Corollary \ref{coro:Simplicial2CechComplex}).
We also show that in dimension 1 (i.e., $\F = \R$) cohomologous cocycles
yield equivalent projective coordinates,
while in dimension 2 (i.e., $\F = \C$) the harmonic cocycle is needed (see Section \ref{sec:ChoosingCocycles}).

It is entirely possible that a cohomology class reflecting sampling artifacts is chosen, as opposed to one associated to robust topological features of a continuous space $\mathbb{X} \subset \mathbb{M}$ underlying $X$. Here is where persistent cohomology comes in.
Indeed,  under mild connectivity conditions of $\mathcal{B}$, distinct cohomology classes yield maps to projective space with distinct homotopy types.
Moreover, the  maps resulting from a persistent  class across its lifetime
are compatible up to homotopy (Theorem \ref{thm:ClassifyingFormulaSparse} and Proposition \ref{prop:Naturality}).
Hence, the result is a multiscale family of compatible maps  which, for classes
with long persistence, are more likely to reflect robust features
of (neighborhoods around) $\mathbb{X}$.

The strategy outlined here is in fact a two-way street.
One can use persistent cohomology to compute multiscale compatible projective coordinates, but the reserve is also useful:
The resulting coordinates can be used to interpret the distinct persistent cohomology (= persistent homology) features of neighborhoods of the data, at least in cohomological dimensions 1 (with $\Z/2$ coefficients) and 2 (with $\Z/p$ coefficients for appropriate primes $p$).

\section{Preliminaries}\label{sec:preliminaries}
\subsection*{Vector Bundles}
For a more thorough review please refer to \cite{milnor1974characteristic}.
Let $E$ and $B$ be topological spaces,
and let $p: E \longrightarrow B$ be a surjective continuous  map.
The triple $\zeta = (E,B,p)$ is said to be a rank $k\in \N$ vector bundle over a field $\F$ (i.e. an $\F$-vector bundle)
if each fiber $p^{-1}(b)$ is an $\F$-vector space of dimension $k$,
and $\zeta$ is locally trivial. That is, for every $b_0\in B$
there exist an open neighborhood $U \subset B$ and a homeomorphism
$\rho_U : U\times \F^k  \longrightarrow p^{-1}(U) $,
called a local trivialization around $b_0$, satisfying:
\begin{enumerate}
\item $p\left(\rho_U(b,\vv)\right) = b$ for every $(b,\vv) \in U\times \F^k$
\item $\rho_U(b,\cdot) : \F^k \longrightarrow p^{-1}(b)$
is an isomorphism of $\F$-vector spaces for each $b\in U$
\end{enumerate}
$E$ and $B$ are  referred to as the total and base space of the bundle,
and the function
$p: E \longrightarrow B$  is  called the projection map.
Two vector bundles $\zeta =(E,B,p)$ and $\zeta' = (E',B,p')$
are said to be isomorphic, $\zeta \cong \zeta'$, if
there exists a homeomorphism $T: E \longrightarrow E'$  so that
$p'\circ T = p $ and for which each  restriction $T|_{p^{-1}(b)} $,
 $b\in B$, is a linear isomorphism.

The collection of isomorphism classes of $\F$-vector bundles of rank $k$ over $B$ is denoted
$\mathsf{Vect}_\F^k(B)$.
An $\F$-vector bundle of rank $1$ is called an $\F$-line bundle, and
 the set $\mathsf{Vect}^1_\F(B)$ is an abelian
group with respect to fiberwise tensor product of $\F$-vector spaces.
\\[.2cm]
\noindent\underline{Examples:}
\begin{itemize}
\item \emph{The trivial bundle $B\times \F^k$}: Fix $k\in \N$ and let
\[
\begin{array}{rccl}
p: &B\times \F^k &\longrightarrow &B \\
&(b,\vv)&\mapsto & b
\end{array}
\]
It follows that $\varepsilon_k = (B\times \F^k, B, p)$ is an $\F$-vector bundle over $B$ of rank $k$.
$\varepsilon_k$ is referred to as the trivial bundle.
\\[-.1cm]
\item \emph{The Moebius band}: Let $\sim$ be the relation on $\R\times \R$ given by
$(x,\uu) \sim (y,\vv)$  if and only if $x-y \in \Z$ and $\uu = (-1)^{x-y}\vv$.
It follows that $\sim$ is an equivalence relation, and if $E = \R\times \R/\sim$,
then $\gor{p} : \R \times \R \longrightarrow \R$ given by $\gor{p}(x,\uu) = x$ descends to a continuous surjective map
$p : E \longrightarrow \R/\Z$. Hence
$\gamma^1 = (E, \R/\Z, p)$ is an $\R$-line bundle over the circle $\R/\Z$, whose
total space is a model for the Moebius band.
Since $E$ is nonorientable, it follows that  $\gamma^1 $ is not isomorphic to the trivial line bundle $ \varepsilon_1$.
\\[-.1cm]
\item \emph{Grassmann manifolds and their tautological bundles}: Let $\F$ be either $\R$ or $\C$.
Given   $k\in \Z_{\geq 0}$ and $m\in \Z_{\geq k}\cup\{\infty\}$,
let $\mathsf{Gr}_k(\F^m)$ be the collection of $k$-dimensional linear subspaces of $\F^m$.
This set is in fact a manifold, referred to as the Grassmannian  of $k$-planes in $\F^m$.
The tautological bundle over $\mathsf{Gr}_k(\F^m)$, denoted $\gamma^k_{m}$, has total space
\[
E(\gamma_{m}^k) = \big\{(V,\uu) \in \mathsf{Gr}_k(\F^m) \times \F^m \; : \; \uu \in V\big\}
\]
and projection $p: E(\gamma_{m}^k) \longrightarrow \mathsf{Gr}_k(\F^m)$ given by $p(V,\uu) = V$.
In particular one has that $\mathsf{Gr}_1(\F^{m+1}) = \FP^{m}$, which shows that
each projective space $\FP^m$ can be endowed with a tautological line bundle $\gamma^1_{m}$.
\\[-.1cm]
\item \emph{Pullbacks}: Let $B$ and $B'$ be topological spaces, let $\zeta = (E,B,p)$
be a vector bundle and let $f: B' \longrightarrow B$ be a continuous map.
The pullback of $\zeta$ through $f$, denoted $f^*\zeta$, is the vector bundle over $B'$ with total space
\[
E(f^*\zeta) = \big\{(b,e) \in B'\times E \;: \; f(b) = p(e)\big\}
\]
and projection $p':E(f^*\zeta) \longrightarrow B'$ given by
$p'(b,e) = b$.

\end{itemize}

\begin{theorem}[{\cite[5.6 and 5.7]{milnor1974characteristic}} ]\label{thm:BundleClassification}
If $B$ is a paracompact topological space and  $\zeta$ is an $\F$-vector bundle
of rank $k$ over $B$, then there exists  a continuous map
\[f_\zeta: B \longrightarrow \mathsf{Gr}_k(\F^\infty)\]
satisfying $f_\zeta^*\gamma_{\infty}^k \cong \zeta$.
Moreover, if $g : B \longrightarrow \mathsf{Gr}_k(\F^\infty)$
is continuous and also satisfies $g^*\gamma_{\infty}^k \cong \zeta$,
then $f_\zeta \simeq  g$ and hence $f_\zeta$ is unique
up to homotopy.
\end{theorem}

The previous theorem can  be rephrased as follows:
For $B$ paracompact, the function

\begin{equation}\label{eq:VectToMaps}
\begin{array}{ccc}
\mathsf{Vect}_\F^k(B) & \longrightarrow & \big[B\,,\, \mathsf{Gr}_k(\F^\infty)\big]  \\[.1cm]
[\zeta] & \mapsto & [f_\zeta]
\end{array}
\end{equation}
is a bijection. Any such $f_\zeta$ is  referred to as a \textbf{classifying map} for $\zeta$.
\\[.2cm]
\noindent\underline{Transition Functions.}
If  $\rho_U : U\times \F^k \longrightarrow p^{-1}(U)$ and
$\rho_V : V \times \F^k \longrightarrow \rho^{-1}(V)$ are local trivializations
around a point $b_0\in U\cap V$,
then given $b\in U\cap V$ the composition
\[
\begin{tikzcd}[column sep = scriptsize]
\F^k \arrow{rr}{\rho_V(b,\,\cdot\,)}
& &
p^{-1}(b) \arrow{rr}{\rho_U(b,\,\cdot\,)^{-1}}
& &
\F^k
\end{tikzcd}
\]
defines an element $\rho_{UV}(b) $ in the general linear group $ \mathsf{GL}_k(\F)$.
The resulting function
$\rho_{UV} : U \cap V \longrightarrow \mathsf{GL}_k(\F)$
is a continuous map,
 uniquely determined by
\[
\rho_U^{-1}\circ \rho_V(b,\vv) = \big(b, \rho_{UV}(b)(\vv)\big)\;\;\;\;\; \mbox{for every} \;\;\;\;\;
(b,\vv) \in (U\cap V)\times \F^k.
\]
This characterization of $\rho_{UV}$
readily implies that the set $\{\rho_{UV}\}$ satisfies:
\begin{equation}\label{eq:CocycleCondition}
\fbox{
\begin{Bitemize}
\item[] \\[-.4cm]
\item[] \hfill  \underline{\textbf{The Cocycle Condition}} \hfill \mbox{ } \\[-.1cm]
\item  $\rho_{UU}(b)$ is the identity linear transformation
  for every $b\in U$. \; \\
\item
$
\rho_{UW}(b) = \rho_{UV}(b) \circ \rho_{VW}(b)
$
for every $b\in U\cap V\cap W$.\\
\end{Bitemize}
}
\end{equation}

Each $\rho_{UV} : U\cap V\longrightarrow \mathsf{GL}_k(\F)$ is called a
\emph{transition function} for the bundle $\zeta$,
and the collection $\{\rho_{UV}\}$
is the system of transition functions
associated to the system of local trivializations $\{\rho_U\}$.
More importantly, this construction can be reversed:
If $\mathcal{U} = \{U_r\}$ is a covering of
$B$ and
\[
\omega = \{\omega_{rt}: U_r \cap U_t \longrightarrow \mathsf{GL}_k(\F)\}
\]
is a collection of continuous functions satisfying the cocycle condition,
then one can form the quotient space
\[
E(\omega) = \left(\bigcup_r \left(U_r\times\{r\} \times \F^k\right)\right)\Big/\sim
\]
where   $(b,r,\vv) \sim \left(b,t,\omega_{rt}(b)^{-1}(\vv)\right)$
for  $b\in U_r \cap U_t$.
Moreover, if
\[p_\omega: E(\omega) \longrightarrow B\] is projection onto the first coordinate,
then $\zeta_\omega = \big(E(\omega), B, p_\omega\big)$ is
an  $\F$-vector bundle of rank $k$ over $B$.
It follows that each composition
\[ \rho_r :  \hspace{-.2cm}
\begin{tikzcd}[column sep = scriptsize]
U_r\times\ \F^k \arrow[hook]{r} & U_r \times \{r\}\times \F^k \arrow{r} &
p_\omega^{-1}(U_r)
\end{tikzcd}
\]
is a local trivialization for $\zeta_\omega$,
and that $\omega$ is the associated  system of transition functions.
We say that $\zeta_\omega$ is the vector bundle induced by
$(\mathcal{U}, \omega)$.

\subsection*{(Pre)Sheaves and their \v{C}ech Cohomology}
For a more detailed introduction please refer to \cite{miranda1995algebraic}.
A presheaf $\mathcal{F}$ of abelian groups over a topological space $B$
is a collection of abelian groups $\mathcal{F}(U)$, one for each open set $U\subset B$,
and group homomorphisms $\eta^{U}_{V}: \mathcal{F}(U) \longrightarrow \mathcal{F}(V)$
for each pair $V\subset U$ of open subsets of $B$, called restrictions,  so that:
\begin{enumerate}
\item $\mathcal{F}(\emptyset)$ is the group with one element
\item $\eta_U^U$ is the identity homomorphism
\item $\eta_W^U = \eta^V_W \circ \eta^U_V$ for every triple $W\subset V \subset U$
\end{enumerate}
\noindent Furthermore, a presheaf $\mathcal{F}$ is said to be a sheaf if it satisfies the gluing axiom:
\begin{enumerate}
\item[4.] If $U\subset B$ is open, $\{U_j\}_{ j\in J}$ is an open covering of $U$ and there are elements
$\{s_j \in \mathcal{F}(U_j)\, : \, j\in J\}$ so that
\[\eta^{U_j}_{U_j \cap U_\ell}(s_j) = \eta_{U_j\cap U_\ell}^{U_\ell}(s_\ell)\]
for every non-empty intersection $U_j \cap U_\ell\neq \emptyset$,
with $j,\ell \in J$, then there exists a unique $s\in \mathcal{F}(U)$
 so that $\eta^U_{U_j}(s) = s_j$ for every $j\in J$.
\end{enumerate}

\noindent\underline{Examples:}
\begin{itemize}
\item \emph{Presheaves of constant functions}: Let $G$ be an abelian group and for each open set $\emptyset \neq U \subset B$,
    let $\mathcal{P}_G(U)$ be the set of constant functions from $U$ to $G$. Let $\mathcal{P}_G(\emptyset) = \{0\} \subset G$.
    If for $V\subset U \subset B$ we let $\eta_V^U : \mathcal{P}_G(U) \longrightarrow \mathcal{P}_G(V)$
    be the restriction map $f\mapsto f|_V$, then  $\mathcal{P}_G$ is a presheaf over $B$.
    It is not in general a sheaf since it does not always satisfy the gluing axiom:
    for if $U,V \subset B$ are disjoint nonempty open sets and $|G| \geq 2$, then $f \in \mathcal{P}_G(U)$ and
    $g \in  \mathcal{P}_G(V)$
    taking distinct values cannot be realized as  restrictions of a constant function $h: U\cup V \longrightarrow G$.

\item \emph{Sheaves of locally constant functions}: Let $G$ be an abelian group and for each open set $\emptyset \neq U \subset B$, let $\underline{G}(U)$ be the set of functions $f: U \longrightarrow G$ for which there exists an open set
    $\emptyset\neq V \subset U$ so that the restriction $f|_V : V \longrightarrow G$ is a constant function.
    Define $\underline{G}(\emptyset)$ and $\eta_V^U$ as in the presheaf of constant functions.
    One can check that $\underline{G}$ is   a sheaf over $B$.

\item \emph{Sheaves of continuous functions}: Let $G$ be a topological abelian group and for each open set $\emptyset \neq U \subset B$,
    let $\mathscr{C}_G(U)$ be the set of continuous functions from $U$ to $G$. If  $\mathscr{C}_G(\emptyset)$
    and $\eta_V^U$ are as above, then $\mathscr{C}_G$ is a sheaf over $B$. Moreover,
    $\underline{G}$ is a subsheaf of $\mathscr{C}_G$ in that  $\underline{G}(U) \subset \mathscr{C}_G(U)$
    for every open set $U\subset B$.
    Similarly, if $R$ is a commutative topological ring with unity, and $R^\times$ denotes its (multiplicative) group of units, then
    $\mathscr{C}^\times_R := \mathscr{C}_{R^\times}$ is also a sheaf over $B$.
\end{itemize}

Let $n\geq 0$ be an integer, $\mathcal{U} = \{U_j\}$ an open cover of $B$
and let $\mathcal{F}$ be a presheaf over $B$.
The group of \v{C}ech $n$-cochains is defined as
\[
\check{C}^n(\mathcal{U}; \mathcal{F}) = \prod_{(j_0,\ldots, j_n)} \mathcal{F}(U_{j_0} \cap \cdots \cap U_ {j_n})
\]
Elements of $\check{C}^n(\mathcal{U}; \mathcal{F})$ are denoted  $ \{f_{j_0,\ldots, j_n}\}$, for $f_{j_0,\ldots,j_n} \in \mathcal{F}(U_{j_0} \cap \cdots \cap U_{j_n})$.
If $0 \leq r \leq n$  let $(j_0,\ldots, \tech{j}_r, \ldots, j_n)$
denote the $n$-tuple obtained by removing $j_r$ from the $(n+1)$-tuple $(j_0,\ldots, j_n)$,
let $U_{j_0, \ldots, j_n} = U_{j_0} \cap \cdots \cap U_{j_n}$ and let
\[
\eta_{j_r} : \mathcal{F}(U_{j_0,\ldots, \tech{j}_r , \ldots, j_n})
\longrightarrow
\mathcal{F}(U_{j_0,\ldots, j_n})
\]
be the associated restriction homomorphism.
The coboundary homomorphism
\[
\delta^n : \check{C}^n(\mathcal{U};\mathcal{F}) \longrightarrow
\check{C}^{n+1}(\mathcal{U}; \mathcal{F})
\]
is given by  $\delta^{n}(\{f_{j_0,\ldots, j_n}\}) = \{g_{k_0,\ldots, k_{n+1}}\}$
where
\[
g_{k_0,\ldots, k_{n+1}} = \sum_{r=0}^{n+1} (-1)^r \eta_{k_r}\left(f_{k_0,\ldots, \tech{k}_r,\ldots, k_{n+1}}\right)
\]
One can check that  $\delta^{n+1} \circ \delta^n = 0$.
The group of
\v{C}ech $n$-cocycles $\check{Z}^n(\mathcal{U}; \mathcal{F})$ is the kernel of $\delta^n$,
the group of \v{C}ech $n$-boundaries $\check{B}^n(\mathcal{U};\mathcal{F})
\subset \check{Z}^n(\mathcal{U};\mathcal{F})$
is the image of $\delta^{n-1}$, and the $n$-th \v{C}ech
cohomology group of $\mathcal{F}$ with respect to the covering $\mathcal{U}$ is given by the
quotient of abelian groups
\[
\check{H}^n(\mathcal{U};\mathcal{F}) = \check{Z}^n(\mathcal{U};\mathcal{F}) / \check{B}^n(\mathcal{U};\mathcal{F}).
\]

\subsection*{Persistent Cohomology of Filtered Complexes}
Given a nonempty set $S$, an abstract simplicial complex $K$ with vertices in  $S$
is a set
\[
K \subset \{ \sigma \subset S : \sigma \mbox{ is finite and }\sigma \neq \emptyset \}
\]
for which
$\emptyset\neq \tau \subset \sigma \in K$ always implies $\tau \in K$.
An element $\sigma \in K$ with cardinality $|\sigma| = n+1$ is called an $n$-simplex of $K$,
and a 0-simplex  is referred to as a vertex.
\\[.2cm]
\noindent\underline{Examples:}
\begin{itemize}
\item \emph{The Rips Complex:} Let $(\mathbb{M}, \mathbf{d})$ be a metric space,
let $X\subset \mathbb{M}$ and $\epsilon \geq 0$.
The Rips complex at scale $\epsilon$ and vertex set $X$, denoted  $R_\epsilon(X)$,
is the collection of finite nonempty subsets of $X$ with  diameter  less than $2\epsilon$.

\item \emph{The \v{C}ech Complex:} With $\mathbb{M}, \mathbf{d}, X, \epsilon$ as above, the
(ambient) \v{C}ech complex at scale $\epsilon$ and vertices in $X$ is the set
\[
\check{C}_\epsilon(X) =
\big\{
\{s_0,\ldots, s_n\} \subset  X : B_{\epsilon}(s_0) \cap \cdots \cap B_{\epsilon}(s_n)
\neq \emptyset \; , \;\; n \in \Z_{\geq 0} \big\}
\]
where $B_\epsilon(s)$ denotes the open ball in $\mathbb{M}$ of radius $\epsilon$ centered at $s\in X$.
It can be readily checked that $\check{C}_{\epsilon}(X) \subset R_{\epsilon}(X) \subset \check{C}_{2\epsilon}(X)$
for all $\epsilon > 0$.
\end{itemize}

For each $n \in \Z_{\geq 0}$ let $K^{(n)}$ be the set of $n$-simplices of $K$.
If $G$ is an abelian group,  the set of functions $\varphi : K^{(n)} \longrightarrow G$
which evaluate to zero in all but finitely many $n$-simplices form an abelian group denoted
$
C^n(K;G)
$,
and referred to as the group of $n$-cochains of $K$ with coefficients in $G$.
The coboundary of an $n$-cochain $\varphi \in C^n(K;G)$
is the element $\delta^n(\varphi) \in C^{n+1}(K;G)$ which operates on each  $n+1$ simplex
$\sigma = \{s_0,\ldots, s_{n+1}\}$ as
\[
\delta^n(\varphi)(\sigma) = \sum_{j=0}^{n+1}
(-1)^j\varphi \big(\sigma \smallsetminus \{s_j\}\big)
\]
This defines a homomorphism $\delta^n : C^n(K;G) \longrightarrow C^{n+1}(K;G)$ that,
as can be checked, satisfies $\delta^{n+1}\circ \delta^n = 0$ for all $n\in \Z_{\geq 0}$.
The group of
$n$-coboundaries  $B^n(K;G) = Img (\delta^{n-1})$ is therefore a subgroup of
$Z^n(K;G)= Ker(\delta^n)$, the group of $n$-cocyles,
and the $n$-th cohomology
group of $K$ with coefficients in $G$ is defined as the quotient
\[
H^n(K;G) = Z^n(K;G)/ B^n(K;G).
\]
Notice that if $\mathbb{F}$ is a field, then $H^n(K;\mathbb{F})$ is in fact a vector space over $\mathbb{F}$.

A filtered simplicial complex is a collection $\mathcal{K}= \{K_\epsilon\}_{\epsilon \geq 0}$ of abstract simplicial complexes
so that $K_0 = \emptyset$ and $K_\epsilon \subset K_{\epsilon'}$ whenever $\epsilon \leq \epsilon'$.
If $0 = \epsilon_0 < \epsilon_1 < \cdots$ is a discretization of $\R_{\geq0}$,
then for each field $\mathbb{K}$ and  $n \in \Z_{\geq 0}$ one obtains the diagram of $\mathbb{F}$-vector spaces and
linear transformations
\begin{equation}\label{eq:persistenceModule}
\begin{tikzcd}[column sep = 2em]
H^n(K_{\epsilon_0} ; \mathbb{F})
&
H^n(K_{\epsilon_1} ; \mathbb{F})\arrow[swap]{l}{T_1}
&
\cdots\arrow[swap]{l}{T_2}
&
H^n(K_{\epsilon_r} ; \mathbb{F})\arrow[swap]{l}{T_r}
&
\cdots\arrow{l}
\end{tikzcd}
\end{equation}
where $T_{r} : H^n(K_{\epsilon_{r}} ; \mathbb{F}) \longrightarrow H^n(K_{\epsilon_{r-1}} ; \mathbb{F})$
is  given by
\[
T_{r}\left(\big[\varphi\big]\right) =
\left[
\varphi\big|_{K^{(n)}_{\epsilon_{r-1}}}
\right]
\]
If each $H^n(K_{\epsilon_r}; \mathbb{F})$ is finite dimensional
and for all $r$ large enough $T_r$ is an isomorphism,
we say that  (\ref{eq:persistenceModule}) is of finite type.

The Basis Lemma \cite[Section 3.4]{edelsbrunner2015persistent}
implies that when (\ref{eq:persistenceModule}) is of finite type
one can choose a basis
$V^r = \{\vv^r_1,\ldots, \vv^r_{d_r}\}$ for each $H^n(K_{\epsilon_r};\mathbb{F})$
so that the following compatibility condition holds:
$T_{r}\big(V^r\big) \subset V^{r-1} \cup \{\mathbf{0}\}$
for all $r\in \Z_{\geq 0}$,
and if $T_r(\vv_\ell^r) = T_r(\vv_m^r) \neq \mathbf{0}$, then $\ell = m $.
The set
\[
V = \bigcup\limits_{r\in \Z_{\geq 0 }} V^r
\]
can  be endowed with  a partial order
$\preceq$  where $\vv^j_m \preceq  \vv^r_\ell$ if and only if
$r\geq j$ and  $\vv_m^j=T_{j+1}\circ \cdots \circ T_r(\vv^r_\ell) $.
The maximal chains in $(V,\preceq)$ are pairwise disjoint,
and hence represent independent  cohomological features
of the complexes $K_\epsilon$ which are stable with respect to changes in  $\epsilon$. These are called persistent cohomology classes.
A maximal chain of finite length
\[
\vv_m^j \preceq \vv_k^{j+1} \preceq \cdots \preceq \vv_\ell^r
\]
yields the interval $[\epsilon_{j}, \epsilon_{r}]$,
while an infinite maximal chain
$\vv_m^j \preceq \vv_k^{j+1} \preceq   \cdots$
yields the interval $[\epsilon_{j} ,\infty)$.
This is meant to signify that there is a class
which starts (is born) at the cohomology group corresponding
to the right end-point of the interval, here $\epsilon_r$ or $\infty$. This class, in turn,  is mapped to zero
(it dies) leaving the cohomology group for the left end-point,
here $\epsilon_j$,
but not before.
The  multi-set of all such intervals  (as several chains might yield the same interval) is independent of the choice of compatible bases $V^r$, and can be visualized as a barcode:

\begin{figure}[!ht]
  \centering
  \includegraphics[width=.7\textwidth]{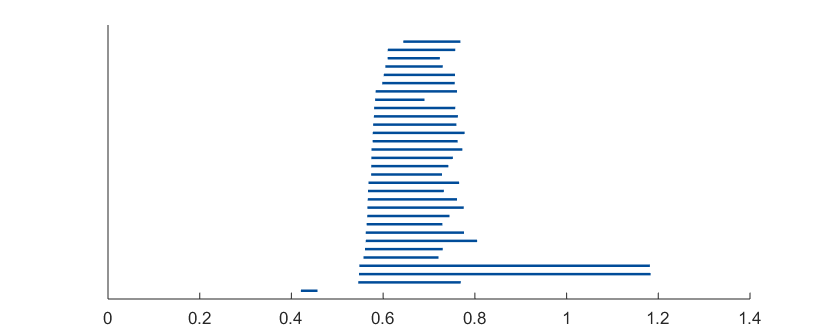}
  \caption{Example of a barcode}\label{fig:exaple_barcode}
\end{figure}

Details on the computation of persistent cohomology, and its advantages over persistent homology,
can be found in  \cite{de2011dualities}.

\section{Explicit Classifying Maps}
\label{sec:ClassifyingMaps}
The goal of this section is to derive equation (\ref{eq:classMapFormula}),
which is in fact a  specialization
of the proof
of Theorem \ref{thm:BundleClassification}
to the case of line bundles over  metric spaces with finite trivializing covers.
When a metric is given,   the partition of unity
involved in the argument can be described explicitly in terms of bump functions
supported on metric balls.
Moreover, the local trivializations used in the proof
can be replaced by transition functions which ---
as we will see in sections \ref{sec:TransitionFuncFromSimplicialCohomology}
and \ref{sec:TransitionFuncFromPersistentCohomology} ---
can be calculated in a robust multiscale manner from the
persistent cohomology of an appropriate sparse filtration.
From this point on, all topological spaces are assumed to be
paracompact and Hausdorff.

\subsection*{Classifying maps in terms of local trivializations}
Let us sketch the proof of existence in Theorem \ref{thm:BundleClassification}
when $B$ has a finite trivializing cover.
Starting with  local trivializations
\begin{equation}\label{eq:TrivializingCover}
\begin{tikzcd}
\rho_r: U_r \times \F^k \arrow{r}{\cong} &p^{-1}(U_r)&  r=0,\ldots, n
\end{tikzcd}
\end{equation}
for the vector bundle
$\zeta = (E,B,p)$,
let $\mu_r : p^{-1}(U_r) \longrightarrow \F^k$ be
$\mu_r\big(\rho_r(b,\vv)\big) = \vv$ for all $(b,\vv) \in U_r \times \F^k$.

\begin{definition}
A collection of continuous maps   $\varphi_r : U_r \longrightarrow \R_{\geq 0}$   is called a
\emph{partition of unity dominated} by $\mathcal{U} = \{U_r\}$
if
\[
\sum_{r} \varphi_r = 1 \;\;\;\;\;\; \mbox{ and }\;\;\;\;\;\;
\mathsf{support}(\varphi_r) \subset \mathsf{closure}(U_r)
\]
\end{definition}
Notice that
this notion differs from the usual partition of unity subordinated to a cover in that supports need not be contained in the open sets.
However, this is enough for our purposes.
Notice that since for paracompact  spaces there is always a partition of unity
subordinated to a given cover, the same is true in the dominated case.

Let $\psi: [0,1] \longrightarrow [0,1]$
be any homeomorphism  so that  $\psi(0) = 0$ and $\psi(1)=1$.
If $\{\varphi_r\}$ is a partition of unity dominated by the
trivializing cover from equation (\ref{eq:TrivializingCover})
and we let $\psi_r = \psi\circ \varphi_r$,
then each $\mu_r : p^{-1}(U_r) \longrightarrow \F^k$  yields a fiberwise linear map
$\tech{\mu}_r : E \longrightarrow \F^k$
given by
\[\tech{\mu}_r(e) =
\left\{
  \begin{array}{cc}
    \psi_r\big(p(e)\big)\cdot \mu_r(e) & \hbox{ if }\;\;\;p(e)\in U_r \\ \\
    \mathbf{0} & \hbox{ if }\;\;\;p(e)\notin U_r
  \end{array}
\right.
\]

Thus one has a continuous map
\[
\begin{array}{rccc}
\tech{\mu} : &E&\longrightarrow & \F^k \oplus \cdots \oplus \F^k \\
&e& \mapsto & \left[\tech{\mu}_0(e),\ldots, \tech{\mu}_n(e)\right]
\end{array}
\]
which, as  can be checked, is linear and injective on each fiber.
It follows from \cite[Lemma 3.1]{milnor1974characteristic} that
the induced continuous map
\[
\begin{array}{rccc}
f_\zeta : &B& \longrightarrow & \mathsf{Gr}_k\left(\F^{k(n+1)}\right) \\
&b& \mapsto & \tech{\mu}\left(p^{-1}(b)\right)
\end{array}
\]
 satisfies $f_\zeta^*\left(\gamma^k_{m}\right) \cong \zeta$,
 if $m = k(n+1)$.
This completes the sketch of the proof;
 let us now describe $f_\zeta$ more explicitly  in terms
of transition functions.
The main reason for this change is that transition functions
can be determined via a cohomology computation.

\subsection*{Classifying maps from transition functions}
Fix $b\in B$ and let $0 \leq j\leq n$ be   so that $b \in U_j$.
If $\{\vv_1,\ldots, \vv_k\}$ is a basis for $\F^k$,  then
$
\{ \rho_j ( b, \vv_s ) \; | \; s=1,\ldots, k  \}
$
is a basis for $p^{-1}(b)$ and therefore
\begin{equation}
f_\zeta(b) = \mathsf{Span}_\F\, \left\{ \tech{\mu}\big(\rho_j(b,\vv_s)\big)\;:\; s=1,\ldots, k \right\}
\label{eq:classMapFormula2}
\end{equation}
where
\[\tech{\mu}\big(\rho_j(b,\vv_s)\big) = \left[ \tech{\mu}_0\big(\rho_j(b,\vv_s)\big),\ldots, \tech{\mu}_n\big(\rho_j(b,\vv_s)\big)\right]
\in \F^{k(n+1)}\]
If $\{\omega_{rt}: U_r \cap U_t \longrightarrow \mathsf{GL}_k(\F)\}$
is the collection of transition functions
for $\zeta$ associated to the system of local trivializations $\{\rho_r\}$,
then whenever $U_r \cap U_t \neq \emptyset$ we have the commutative diagram
\[
\begin{tikzcd}[column sep = scriptsize]
\mbox{ }&p^{-1}(U_r \cap U_t) \\
(U_r\cap U_t) \times \F^k \arrow{rr}[below]{(b,\vv)\mapsto \big(b\;,\;\omega_{rt}(b)^{-1}(\vv)\big)}
\arrow{ur}{\rho_r} & &
(U_r\cap U_t) \times \F^k \arrow{lu}[above]{\rho_t}
\end{tikzcd}
\]
Let $0\leq l \leq n$. If $b\in U_j \smallsetminus U_l$, then  $\tech{\mu}_l\big(\rho_j(b,\vv_s)\big) =0$;
else $b \in U_j\cap U_l$ and
\begin{eqnarray*}
\tech{\mu}_l\big(\rho_j(b,\vv_s)\big) &=&
\tech{\mu}_l\left(\rho_l\big(b, \omega_{jl}(b)^{-1}(\vv_s)\big)\right) \\
&=& \psi_l(b)\cdot \omega_{lj}(b)(\vv_s)
\end{eqnarray*}
\noindent Putting this calculation together with equation (\ref{eq:classMapFormula2}), it follows that for $b\in U_j$
\begin{equation}
f_\zeta(b) = \mathsf{Span}_\F\Big\{ \big[\psi_0(b)\cdot \omega_{0j}(b)(\vv_s)\,, \ldots,\,\psi_n(b)\cdot \omega_{nj}(b)(\vv_s)  \big]\; : \; s =1,\ldots, k\Big\}
\label{eq:classMapFormula3}
\end{equation}

\subsection*{The case of line bundles}
If  $k=1$ we can take $\vv_1 = \textbf{1} \in \F$,
and abuse notation by writing $\omega_{rt}(b) \in \F^\times$ instead of $\omega_{rt}(b)(\mathbf{1})$.
Moreover, in this case we have $\mathsf{Gr}_k(\F^{k(n+1)}) = \FP^n$,
and if we use homogeneous coordinates and $\psi(x) = \sqrt{x}$,
then $f_\zeta: B \longrightarrow \FP^n$ can be expressed locally (i.e. on each $U_j$) as
\[
f_\zeta(b) =
\left[
\omega_{0j}(b)\cdot\sqrt{\varphi_0(b)}
: \cdots :
\omega_{nj}(b)\cdot \sqrt{\varphi_n(b)} \right]\;\;\;\; , \;\;\;\; b\in U_j
\]
The choice $\psi(x) = \sqrt{x}$ is so that when the transition functions
$\omega_{rt}$ are unitary, i.e. $|\omega_{rt}(b)| =1 $, then
the formula above without homogeneous coordinates
produces a representative of $f_\zeta(b)$ on the unit sphere of $\F^{n+1}$.
We summarize the results thus far in the following theorem:
\begin{theorem}\label{thm:ClassifyingMapFormula1}
Let $B$ be a  topological space and  let
$\mathcal{U} = \{U_r\}_{r=0}^n$ be
an open  cover.
If $\{\varphi_r\}$ is a partition of unity dominated by $\mathcal{U}$,
$\F $ is $\R$ or $\C$,
 and
\[\omega =
\big\{
\omega_{rt} : U_r \cap U_t \longrightarrow \F^\times
\big\}
\]
is a collection
of continuous maps  satisfying the cocycle condition (\ref{eq:CocycleCondition}),
then the map
$
f_\omega: B \longrightarrow \FP^n
$
given  in homogenous coordinates
by
\begin{equation}\label{eq:classMapFormula_final}
f_\omega(b) =
\left[
\omega_{0j}(b)\cdot \sqrt{\varphi_0(b)}:
\cdots:
\omega_{nj}(b)\cdot \sqrt{\varphi_n(b)}
\right] \;\;\;\;\; , \;\;\;\;\; b\in U_j
\end{equation}
is well-defined and classifies  the $\F$-line bundle $\zeta_\omega$
induced by $(\mathcal{U}, \omega)$.
\end{theorem}

\subsection*{Line bundles over metric spaces}
If $B$ comes equipped with a  metric $\mathbf{d}$, then (\ref{eq:classMapFormula_final}) can be further specialized to
a  covering  via
open balls, and a dominated partition of unity  constructed
from bump functions supported on the closure of each ball.
Indeed, let
\[ B_{\epsilon_r}(\ell_r) = \{b\in B\, :\, \mathbf{d}(b,\ell_r) < \epsilon_r\}\;,\;\;\;
r = 0,\ldots, n\] for some
collection $\{\ell_0,\ldots, \ell_n\} \subset B$ and
radii $\epsilon_r > 0$.
\begin{proposition}
Let $(B,\mathbf{d})$ be a metric space
and let $\mathcal{B} = \{ B_{\epsilon_r}(\ell_r)\}_{r= 0 }^n$
be an open cover.
If $\phi: \R \longrightarrow \R_{\geq 0}$ is a continuous map so that $\phi^{-1}(0) = \R_{\leq 0}$,
and $\lambda_0, \ldots, \lambda_n \in \R_{> 0}$ is a set of  weights, then
\[
\varphi_r (b)
=
\frac{
\lambda_r \cdot \phi \left( 1 - \frac{\mathbf{d}(b,\ell_r)}{\epsilon_r}\right)
}
{
\sum\limits_{t=0}^n
\lambda_t\cdot\phi\left(1  - \frac{\mathbf{d}(b,\ell_t)}{\epsilon_t}\right)
}
\;\; , \;\;\; r = 0,\ldots, n
\]
is a partition of unity for $B$ dominated  by  $\mathcal{B}$.
\end{proposition}

Due to the shape of its graph, the map $ b \mapsto \lambda_r \cdot \phi \left(1 - \frac{\mathbf{d}(b,\ell_r)}{\epsilon_r}\right)$
is often referred to as a bump function supported on   $\overline{B_r}$.
The height of the bump is controlled by the weight $\lambda_r$, while its overall shape
is captured by the function $\phi$.
Of course one can choose different functions $\phi_r$ on each  ball,
for instance to capture local density if $B$ comes equipped with a measure.
Some examples of bump-shapes are:
\begin{itemize}
  \item \emph{Triangular:} The  positive part of $x \in \R$ is defined as
        $\ppart{x} = \max\{x, 0\}$, and
      $b\mapsto \lambda \cdot \ppart{1 - \frac{\mathbf{d}(b,\ell)}{\epsilon}}$
        is the associated triangular bump supported on $\overline{B_\epsilon(\ell)}$.
  \item \emph{Polynomial:} The polynomial bump with exponent $p>0$  is induced by the function $\phi(x) = \ppart{x}^p$.
         The triangular bump is recovered when $p =1$,  while $p=2$ yields the quadratic bump.
  \item \emph{Gaussian:} The Gaussian bump is induced by the funcion
        \[ \phi(x) = \left\{
                               \begin{array}{cl}
                                 e^{-1/x^2} & \hbox{if } x > 0 \\
                                 0 & \hbox{ else}
                               \end{array}
        \right.\]
  \item \emph{Logarithmic:} Is the one associated to $\phi(x) = \log(1 + \ppart{x})$
\end{itemize}
Figure \ref{fig:bump_functions} shows some of these bump functions, with weight $\lambda =1$, for $B_1(0)$.

\begin{figure}[!ht]
  \centering
  \includegraphics[width = 0.8\textwidth]{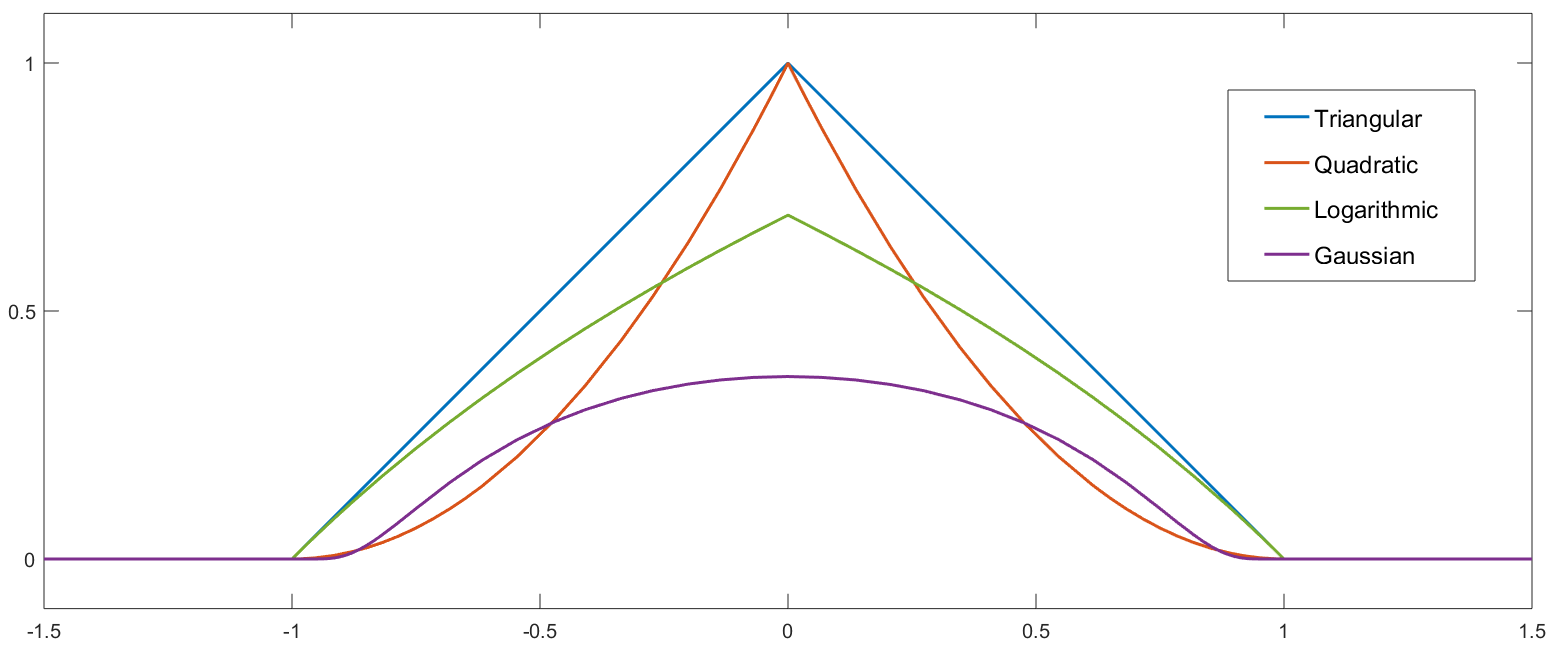}
  \caption{Examples of bump functions supported on $\overline{B_1(0)} \subset \R$. Please refer to an online version for colors}\label{fig:bump_functions}
\end{figure}

Choosing $\phi(x) = \ppart{x}^2$ and the weights as $\lambda_r = \epsilon_r^2$
simplifies Theorem \ref{thm:ClassifyingMapFormula1} to:

\begin{corollary}\label{coro:ClassifyingMapFormula1c}
Let $(B,\mathbf{d})$ be a metric space and
$\mathcal{B} = \{B_r = B_{\epsilon_r}(\ell_r)\}_{r=0}^n$
 a covering. If $\F $ is $ \R $ or $ \C$ and $\omega = \{\omega_{rt} : B_r \cap B_t \longrightarrow \F^\times\}$ are continuous maps satisfying the
cocycle condition, then $f_\omega : B \longrightarrow \FP^n$
given in homogeneous coordinates by
\[
f_\omega(b) =
\Big[
\omega_{0j}(b)\cdot
\ppart{\epsilon_0 - \mathbf{d}(b,\ell_0)}:
\cdots:
\omega_{nj}(b)\cdot
\ppart{\epsilon_n - \mathbf{d}(b,\ell_n)}
\Big] \;\; , \;\;\; b\in B_j
\]
is well-defined and classifies the $\F$-line bundle $\zeta_\omega$
induced by $(\mathcal{B}, \omega)$.
\end{corollary}

\subsection*{Geometric Interpretation}\label{subsection:Geometric Interpretation}
Let us clarify   (\ref{eq:classMapFormula_final}) for the
case of constant transition functions and $\F = \R$.
If $\mathcal{U} = \{U_0,\ldots, U_n\}$ is a cover of $B$, then
the \emph{nerve of} $\mathcal{U}$ --- denoted $\mathcal{N}(\mathcal{U})$ ---
is the abstract simplicial complex with one vertex for each open set
$U_r\in \mathcal{U}$, and a simplex $\{r_0,\ldots, r_k\}$ for
each collection $U_{r_0},\ldots, U_{r_k} \in \mathcal{U}$ such that
\[U_{r_0} \cap \cdots \cap U_{r_k} \neq \emptyset\]
Given a  geometric realization  $|\mathcal{N}(\mathcal{U})| \subset  \R^{2n+1}$,
 let $\vv_r \in |\mathcal{N}(\mathcal{U})|$ be the point corresponding
to the vertex $r\in \mathcal{N}(\mathcal{U})$.
Each  $\xx \in |\mathcal{N}(\mathcal{U})|$ is then
 uniquely determined by (and uniquely determines) its barycentric coordinates:
a sequence $\{x_r\}$ of real numbers between 0 and 1,
one for each open set $U_r\in \mathcal{U}$, so that
\[
\sum\limits_r x_r = 1
\;\;\;\;
\mbox{ and }
\;\;\;\;
\sum_{r} x_r\vv_r = \xx
\]
To see this, notice that given a non-vertex  $\xx \in |\mathcal{N}(\mathcal{U})|$
there exists a unique maximal geometric simplex $\sigma$ of $|\mathcal{N}(\mathcal{U})|$
so that $\xx$ is in the interior of $\sigma$.
If $\vv_{r_0}, \ldots, \vv_{r_k} \in |\mathcal{N}(\mathcal{U})|$
are the vertices of $\sigma$, then $\xx$ can be expressed uniquely
as a convex combination of $\vv_{r_0},\ldots, \vv_{r_k}$ which determines $x_{r_0}, \ldots, x_{r_k}$.
 If $r \neq r_0,\ldots, r_k$, then we let $x_r = 0$.

A   partition of unity  $\{\varphi_r\}$ dominated  by
$\mathcal{U}$  induces a continuous map
\begin{equation}\label{eq:phiMapNerveRealization}
\begin{array}{rccc}
\varphi: & B &\longrightarrow &|\mathcal{N}(\mathcal{U})| \\[.1cm]
& b & \mapsto & \sum\limits_r \varphi_r(b)\vv_r
\end{array}
\end{equation}
That is, $\varphi$ sends $b$ to the point $\varphi(b)$ with
barycentric coordinates $\varphi_r(b)$.
Moreover, if $\big\{\omega_{rt} : U_r \cap U_t  \longrightarrow \{-1,1\}  \big\}$
is a collection of \underline{constant} functions  satisfying the cocycle condition,
then the associated classifying map $f_\omega : B \longrightarrow \RP^n$ from Theorem \ref{thm:ClassifyingMapFormula1} can be decomposed as
\[
\begin{tikzcd}[column sep = scriptsize]
B \arrow{r}{\varphi} & \,|\mathcal{N}(\mathcal{U})|\, \arrow{r}{F_\omega} & \RP^n
\end{tikzcd}
\]
where $F_\omega : |\mathcal{N}(\mathcal{U})| \longrightarrow \RP^n$, in barycentric coordinates,
is given on the open star of a vertex $\vv_j \in |\mathcal{N}(\mathcal{U})|$ as
\[
F_\omega (x_0,\ldots, x_n) = \Big[\omega_{0j}\sqrt{x_0} : \cdots : \omega_{nj}\sqrt{x_n}\Big]
\]


\medskip
\noindent \underline{Example:} Let $ B = S^1$, the unit circle, and let
$\mathcal{U} = \{U_0, U_1, U_2\}$ be the open covering depicted in Figure \ref{fig:example_S1_coveringNerve}(left).

\begin{figure}[!ht]
\centering
\subfigure[$\mathcal{U} = \{U_0, U_1, U_2\}$]{
\includegraphics[scale = 0.2]{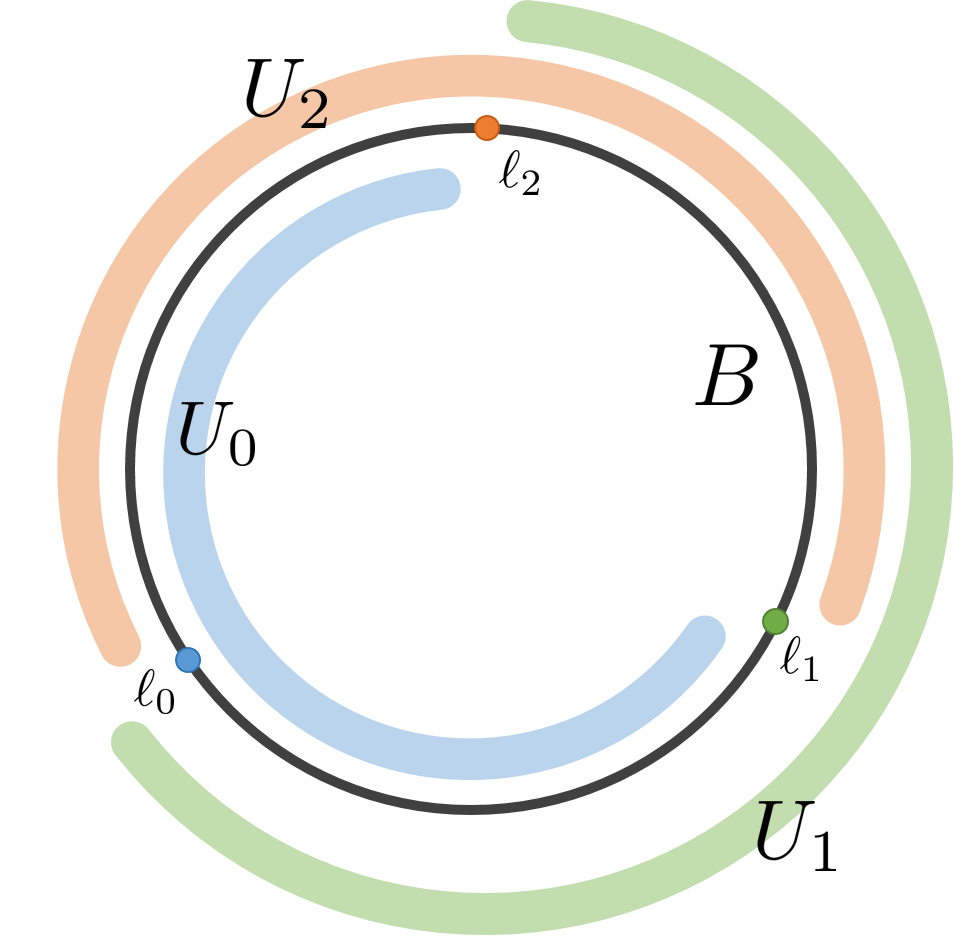}
}
\hspace{1.5cm}
\subfigure[Nerve complex $\mathcal{N}(\mathcal{U})$]{
\includegraphics[scale = 0.3]{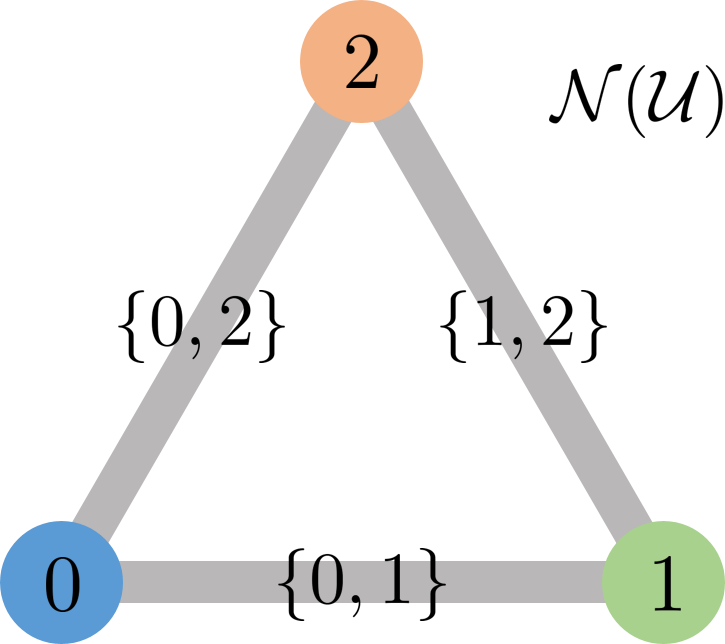}
}
\caption{An open covering of the circle and the resulting nerve complex.}
\label{fig:example_S1_coveringNerve}
\end{figure}

Define
$\omega_{rt} =1$ for  $r,t=0,1,2$; let $\omega'_{02} = \omega'_{20} = -1$
and let $\omega'_{rt} = 1$ for all $(r,t) \neq (0,2), (2,0)$.
Let $(x_0,x_1,x_2)$ denote the barycentric coordinates of a point
in $|\mathcal{N}(\mathcal{U})|$.
For instance, the vertex labeled as $0$ has coordinates $(1,0,0)$ and the midpoint of the edge $\{0,1\}$
 has coordinates $(1/2,1/2,0)$. Then
\[F_\omega(x_0,x_1,x_2) = \big[\sqrt{x_0} : \sqrt{x_1} : \sqrt{x_2}\big]\] and
for each $0 \leq x \leq 1$
\begin{eqnarray*}
F_{\omega'}(x, 1 - x,0) &=& \Big[\sqrt{x}: \sqrt{1 - x}: 0\Big] \\
F_{\omega'}(x,0,1-x) &=& \Big[\sqrt{x}: 0 : - \sqrt{1-x}\Big] \\
F_{\omega'}(0,x,1-x) &=& \Big[0:\sqrt{x}: \sqrt{1-x}\Big]
\end{eqnarray*}
We show in Figure \ref{fig:example_S1_RP2coordinates} how $F_\omega$ and $F_{\omega'}$ map $|\mathcal{N}(\mathcal{U})|$
to $\RP^2 = S^2 /(\xx \sim - \xx)$.

\begin{figure}[!ht]
\centering
\subfigure{
\includegraphics[scale = 0.3]{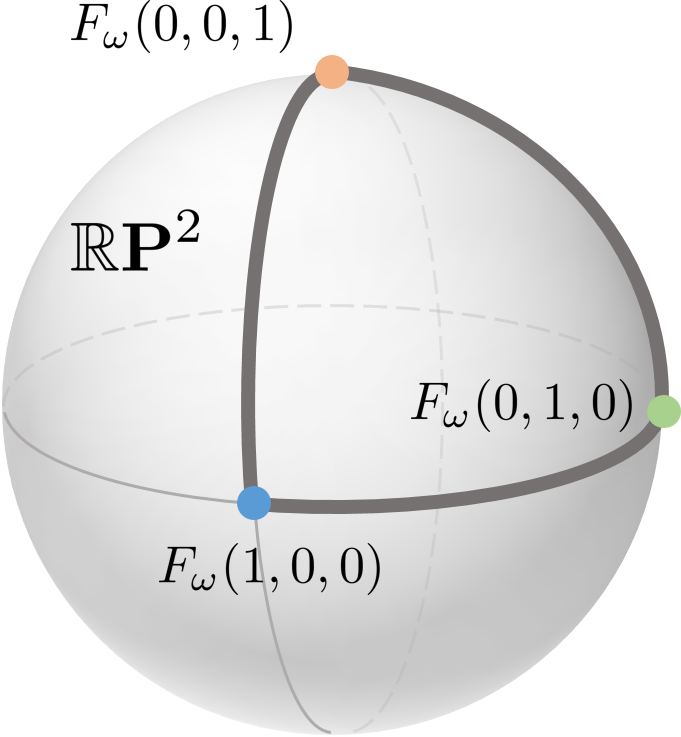}
}
\hspace{2.2cm}
\subfigure{
\includegraphics[scale = 0.3]{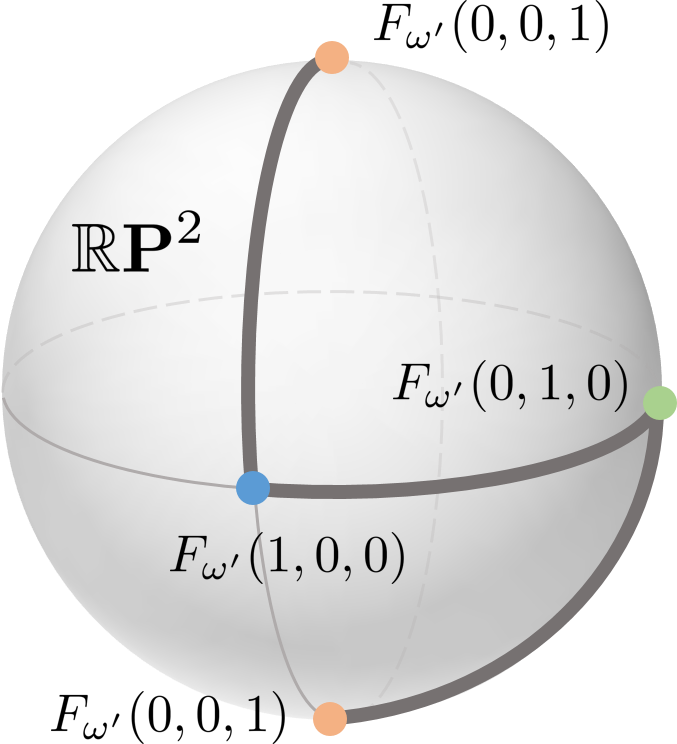}
}
\caption{Image of $F_\omega$ (left) and $F_{\omega'}$ (right) in $\RP^2 = S^2/(\uu \sim - \uu)$.}
\label{fig:example_S1_RP2coordinates}
\end{figure}

If $\mathbf{d}$ is the geodesic distance on $S^1$,
then each arc $U_r$ is an open ball $B_{\epsilon_r}(\ell_r)$
for some $(\epsilon_r,\ell_r)\in \R^{+} \times S^1$.
When   $\varphi: S^1 \longrightarrow |\mathcal{N}(\mathcal{U})| $ is
induced by the  partition of unity from the triangular bumps $\phi(x) = \ppart{x}$,
then $\varphi$
 maps each arc $\{\ell_r, \ell_t\}$ linearly onto the edge $|\{r,t\}|$.
It is not hard to see that the $\R$-line bundle induced
by $\{\omega_{rt}\}$ is trivial,
 while  the one induced by $\{\omega'_{rt}\}$
is the nontrivial bundle on $S^1$ having the Moebius band
as total space.
This is captured  by $f_\omega = F_\omega \circ\varphi : S^1 \longrightarrow \RP^2$
being null-homotopic and  $f_{\omega'} = F_{\omega'} \circ \varphi$
representing the nontrivial element in the fundamental group $\pi_{1}(\RP^2) \cong \Z/2$.

\section{Transition Functions from Simplicial Cohomology}\label{sec:TransitionFuncFromSimplicialCohomology}
Let $\mathcal{U} = \{U_r\}$ be a cover for  $B$.
We have shown thus far that given a
collection $\omega = \{\omega_{rt}\}$ of continuous maps
\[
\omega_{rt} : U_r \cap U_t \longrightarrow \mathsf{GL}_1(\F) = \F^\times \; ,\;\; U_r \cap U_t \neq \emptyset
\]
satisfying the cocycle condition,
one can explicitly write down (given a dominated partition of unity)
a classifying map $f_\omega: B \longrightarrow \FP^n$ for the associated $\F$-line bundle
$\zeta_\omega$.
What we will see next is that determining such transition functions can be reduced
to a computation in simplicial cohomology.

\subsection*{Formulation in Terms of  Sheaf Cohomology}
If $\mathscr{C}_\F^\times$ denotes the sheaf of continuous   $\F^\times$-valued functions on $B$,
then one has that $\omega\in \check{C}^1(\mathcal{U};\mathscr{C}_\F^\times)$.
Moreover, since $\omega$ satisfies the cocycle condition, then $\omega \in \check{Z}^1(\mathcal{U};\mathscr{C}_\F^\times)$,
and hence we can consider the cohomology class $[\omega] \in \check{H}^1(\mathcal{U};\mathscr{C}_\F^\times)$.
If $\mathcal{V}$ is another covering of $B$ we say
that $\mathcal{V}$ is a refinement of $\mathcal{U}$, denoted
$\mathcal{V} \prec \mathcal{U}$,
if for every $V \in \mathcal{V}$ there exists $U \in \mathcal{U}$
so that $V\subset U$.
A standard result (Exercise 6.2, \cite{bott2013differential}) is the following:

\begin{lemma} If $\omega, \omega' \in \check{Z}^1(\mathcal{U};\mathscr{C}_\F^\times)$ are cohomologous,
then $\zeta_\omega \cong \zeta_{\omega'}$.
Moreover, the function
\[
\begin{array}{ccc} \check{H}^1(\mathcal{U};\mathscr{C}_\F^\times) &\longrightarrow &\mathsf{Vec}^1_\F(B)\\[.1cm]
[\,\omega\,] &\mapsto& [\,\zeta_\omega\,]
\end{array}
\]
is an injective homomorphism,
and  natural with respect to refinements of $\mathcal{U}$.
\end{lemma}
Here natural means that if
$H^\mathcal{U}_\mathcal{V} :
\check{H}^1(\mathcal{U};\mathscr{C}_\F^\times)
\longrightarrow
\check{H}^1(\mathcal{V};\mathscr{C}_\F^\times)$
is the homomorphism induced by the refinement
$\mathcal{V} \prec \mathcal{U}$
(see \cite[Chapter IX, Lemma 3.10]{miranda1995algebraic}),
then the diagram
\[
\begin{tikzcd}[row sep = 0.2em]
\check{H}^1(\mathcal{U};\mathscr{C}_\F^\times) \arrow{dd}\arrow{rd}
\\
& \mathsf{Vect}_\F^1(B)
\\
\check{H}^1(\mathcal{V};\mathscr{C}_\F^\times) \arrow{ru}
\end{tikzcd}
\]
is commutative. Combining this with equation (\ref{eq:VectToMaps})
yields
\begin{corollary}  Let $\mathcal{U}$ be an open covering of $B$. Then the function
\[
\begin{array}{rccc}
\Gamma_{\mathcal{U}}: &\check{H}^1(\mathcal{U};\mathscr{C}_\F^\times)& \longrightarrow & \big[ B\, , \, \FP^\infty \big] \\[.1cm]
& [\omega] & \mapsto & [f_\omega]
\end{array}
\]
is injective, and natural with respect to refinements of $\mathcal{U}$.
\end{corollary}

The sheaf cohomology group $\check{H}^1(\mathcal{U};\mathscr{C}_\F^\times)$
can be replaced, under suitable conditions, by  simplicial cohomology groups:
$H^1(\mathcal{N}(\mathcal{U});\Z/2)$ when $\F = \R$, and $H^2(\mathcal{N}(\mathcal{U});\Z)$
when $\F = \C$.
We will not assume that $\mathcal{U}$ is a good cover (i.e. that each finite intersection $U_{r_0} \cap \ldots \cap U_{r_k}$ is either empty or contractible), but rather
will phrase the reduction theorems in terms of the relevant
connectivity conditions.
The construction is described next.

\subsection*{Reduction to Simplicial Cohomology}
If $\tau \in Z^1(\mathcal{N}(\mathcal{U});\Z/2)$,
then for each $U_r \cap U_t \neq \emptyset $ we have $\tau(\{r ,t\}) = \tau_{rt}\in \{0,1\} = \Z/2$.
Let $\phi^\tau = \{\phi^\tau_{rt}\}$ be the collection of constant functions
\[
\begin{array}{rccc}
\phi_{rt}^\tau : & U_r \cap U_t & \longrightarrow & \R^\times \\
& y & \mapsto & (-1)^{\tau_{rt}}
\end{array}
\]
Therefore each $\phi^\tau_{rt}$ is continuous, so
$\phi^\tau \in \check{C}^1(\mathcal{U};\mathscr{C}_\R^\times)$,
and since $\tau$ is a cocycle it follows that
$\phi^\tau \in \check{Z}^1(\mathcal{U}; \mathscr{C}_\R^\times)$.
Moreover, the association $\tau \mapsto \phi^\tau$  induces the homomorphism
\[
\begin{array}{rccc}
\Phi_\mathcal{U} : &H^1(\mathcal{N}(\mathcal{U}); \Z/2)& \longrightarrow &\check{H}^1(\mathcal{U};\mathscr{C}_\R^\times) \\
& [\tau] & \mapsto &\left[\phi^\tau\right]
\end{array}
\]
which is well-defined and satisfies:
\begin{proposition}\label{prop:Simplicial2CechReal}
\[\Phi_\mathcal{U}: H^1(\mathcal{N}(\mathcal{U});\Z/2) \longrightarrow \check{H}^1(\mathcal{U};\mathscr{C}^\times_\R)
\]
is natural with respect to refinements.
Moreover, if each $U_r$ is connected, then $\Phi_\mathcal{U}$ is injective.
\end{proposition}
\begin{proof}
Fix $\tau \in Z^1(\mathcal{N}(\mathcal{U});\Z/2)$
and assume that there is a collection of continuous maps
$\{f_r : U_r \longrightarrow \R^\times\}$ for which
$(-1)^{\tau_{rt}} = \frac{f_t}{f_r}$ on $U_r \cap U_t\neq \emptyset$.
If each element of $\mathcal{U}$ is connected, then
the $f_r$'s have constat sign (either $+1$ or $-1$) in their domains, and hence we can define $\nu \in C^0(\mathcal{N}(\mathcal{U});\Z/2)$ as
\[
\nu(\{r\}) = \frac{1 - \mathsf{sign}(f_r)}{2}
\]
Therefore
$
  \tau(\{r,t\}) = \delta^0(\nu)(\{r,t\})
$
and the result follows.
\end{proof}

Define $\mathbf{w}_1^\mathcal{U}$ as the composition
\[
\begin{tikzcd}[column sep = scriptsize]
\ww_1^\mathcal{U} :
H^1(\mathcal{N}(\mathcal{U});\Z/2)
\arrow{r}{\Phi_\mathcal{U}} &
\check{H}^1(\mathcal{U};\mathscr{C}_\R^\times)
\arrow{r}{\Gamma_\mathcal{U}} &
\left[B,\RP^\infty\right]
\end{tikzcd}
\]
\begin{corollary} If each $U_r$ is  connected,
then
\[
\ww_1^\mathcal{U}:H^1(\mathcal{N}(\mathcal{U});\Z/2) \longrightarrow [B\,,\, \RP^\infty]
\]
is injective and natural with respect to refinements.
\end{corollary}

Let us now address  the complex case. Given $\sigma \in Z^2(\mathcal{N}(\mathcal{U}); \Z)$
let $\psi^\sigma = \{\psi^\sigma_{rst}\}$ be the collection of constant functions
\[
\begin{array}{rccc}
\psi^\sigma_{rst} :& U_r \cap U_s \cap U_t &\longrightarrow& \Z \\
&y& \mapsto & \sigma_{rst}
\end{array}
\]
Each $\psi_{rst}^\sigma$ is  locally constant
and therefore $\psi^\sigma \in \check{Z}^2(\mathcal{U}; \underline{\Z})$.
It follows that the association $\sigma \mapsto \psi^\sigma$
induces a homomorphism
\[
\begin{array}{rccc}
\Psi_\mathcal{U} : &H^2 (\mathcal{N}(\mathcal{U}); \Z)& \longrightarrow & \check{H}^2(\mathcal{U};\underline{\Z}) \\
& [\sigma] & \mapsto & \left[\psi^\sigma\right]
\end{array}
\]
which is well-defined and satisfies:
\begin{proposition}
\[\Psi_\mathcal{U}: H^2(\mathcal{N}(\mathcal{U});\Z) \longrightarrow \check{H}^2(\mathcal{U};\underline{\Z})
\]  is natural with respect to refinements.
Moreover, if  each $U_r \cap U_t$ is either empty or connected, then  $\Psi_\mathcal{U}$ is injective.
\end{proposition}
\begin{proof}
The result is deduced from the following observation:
if  $U_r \cap U_t$ is   connected,
then any function $\mu_{rt} : U_r \cap U_t \longrightarrow \Z$
which is locally constant is in fact constant.
\end{proof}

We will now link $\check{H}^2(\mathcal{U};\underline{\Z})$
and $\check{H}^1(\mathcal{U};\mathscr{C}_\C^\times)$
using the exponential sequence
\[
\begin{tikzcd}[column sep = scriptsize]
0 \arrow{r} & \underline{\Z} \arrow{r} & \mathscr{C}_\C \arrow{r}{\exp} & \mathscr{C}_\C^\times \arrow{r} & 0
\end{tikzcd}
\]
which is given at the level of open sets $U\in \mathcal{U}$ by
\[
\begin{tikzcd}[column sep = scriptsize, row sep = .1em]
 \underline{\Z}(U) \arrow{r} &
\mathscr{C}_\C(U) \arrow{r}{\exp} & \mathscr{C}_\C^\times(U)  \\
 \eta \arrow[mapsto, shorten <= 1em, shorten >= 1em]{r}& \eta \\
  & f \arrow[mapsto , shorten <= .6em, shorten >= .6em]{r} & \exp\{2\pi if\}
\end{tikzcd}
\]
If $\mathfrak{Im}(\exp)$ denotes the image presheaf
\[
\mathfrak{Im}(\exp)\big(U\big) =
Img\left\{ \hspace{-.15cm}
\begin{tikzcd}[column sep = scriptsize]
\mathscr{C}_\C(U) \arrow{r}{\exp} & \mathscr{C}_\C^\times(U)
\end{tikzcd}
\hspace{-.15cm}
\right\}
\]
then
\[
\begin{tikzcd}[column sep = scriptsize]
0 \arrow{r} & \underline{\Z} \arrow{r} & \mathscr{C}_\C \arrow{r}{\exp} & \mathfrak{Im}(\exp) \arrow{r} & 0
\end{tikzcd}
\]
is a short exact sequence of presheaves (i.e. exact for every open set),
and hence we get a long exact sequence in \v{C}ech cohomology
 \cite[Section 24]{serre1955faisceaux}
\[
\begin{tikzcd}[column sep = 1.2em]
{ }\arrow{r} & \check{H}^k(\mathcal{U};\mathscr{C}_\C) \arrow{r} &
\check{H}^k(\mathcal{U};\mathfrak{Im}(\exp)) \arrow{r}{\Delta} &
\check{H}^{k+1}(\mathcal{U};\underline{\Z}) \arrow{r} &
\check{H}^{k+1}(\mathcal{U};\mathscr{C}_\C) \arrow{r} &  { }
\end{tikzcd}
\]
Since $\mathscr{C}_\C$ admits partitions of unity (i.e. it is a fine sheaf), then
\begin{lemma}$\check{H}^k(\mathcal{U};\mathscr{C}_\C) = 0$ for every $k\geq 1$.
\end{lemma}

Hence $\Delta: \check{H}^1(\mathcal{U};\mathfrak{Im}(\exp)) \longrightarrow
\check{H}^2(\mathcal{U};\underline{\Z})$ is an isomorphism.
Moreover,
\begin{proposition} Let $\{\varphi_t\}$ be a continuous partition of unity dominated by $\mathcal{U}$.
If $\eta = \{\eta_{rst}\} \in Z^2(\mathcal{N}(\mathcal{U});\Z)$, then
\[
\omega_{rs} = \exp
\Big (
2\pi i \sum_{t} \varphi_t\cdot \eta_{rst}
\Big ) \;\; , \;\;\; U_r \cap U_s \neq \emptyset
\]
defines an element $\omega = \{\omega_{rs}\} \in \check{C}^1(\mathcal{U}; \mathfrak{Im}(\exp))$.
Moreover, $\omega$ is a \v{C}ech cocycle, and
 the composition
\[
\begin{tikzcd}[column sep = scriptsize]
H^2(\mathcal{N}(\mathcal{U});\Z) \arrow{r}{\Psi_\mathcal{U}} &
\check{H}^2(\mathcal{U};\underline{\Z}) \arrow{r}{\Delta^{-1}} &
\check{H}^1(\mathcal{U};\mathfrak{Im}(\exp))
\end{tikzcd}
\]
satisfies  $\Delta^{-1}\circ\Psi_\mathcal{U} \big([\eta] \big)  = [\omega]$.
\end{proposition}
\begin{proof}
First we check that $\omega$ is a \v{C}ech cocycle:
\begin{eqnarray*}
  \omega_{rs}\cdot \omega_{st} &=& \exp\left(2\pi i \sum_{\ell} \varphi_\ell\cdot (\eta_{rs\ell} + \eta_{st \ell})\right) \\
   &=& \exp\left(2\pi i \sum_{\ell} \varphi_\ell\cdot(\eta_{rt\ell} + \eta_{rst})\right) \\
   &=& \exp\left(2\pi i \cdot \eta_{rst} +  \sum_{\ell} \varphi_\ell\cdot \eta_{rt\ell} \right) \\
&=& \omega_{rt}
\end{eqnarray*}
In order to see that $\Delta([\omega]) = \Psi_\mathcal{U}([\eta])$,
we use the definition of the connecting homomorphism $\Delta$.
First, we let $g = \{g_{rs}\} \in \check{C}^1(\mathcal{U};\mathscr{C}_\C)$
be the collection of functions
\[
g_{rs}(b) = \sum_\ell \varphi_\ell(b)\cdot \eta_{rs\ell} \;\;\;\;\;\;\; b\in U_r \cap U_s
\]
It follows that $\omega = \exp^\# (g)$ and therefore
$\Delta([\omega]) = [\delta^1(g)]$.
The coboundary $\delta^1(g)$ can be computed as
\begin{eqnarray*}
  \delta^1(g)_{rst} &=&  g_{rs} - g_{rt} + g_{st} \\
   &=& \sum_\ell \varphi_\ell \cdot (\eta_{rs\ell} - \eta_{rt\ell} + \eta_{st\ell}) \\
   &=& \sum_\ell \varphi_\ell\cdot  \eta_{rst} \\
&=& \eta_{rst}
\end{eqnarray*}
and therefore
$
\Delta([\omega])  =  \Psi_\mathcal{U}([\eta]).
$
\end{proof}

\begin{corollary}\label{coro:Simplicial2CechComplex}
If each $U_r \cap U_t$ is either empty or connected, then
\[
\begin{array}{cccc}
\Delta^{-1}\circ \Psi_\mathcal{U} : &H^2(\mathcal{N}(\mathcal{U});\Z)& \longrightarrow &
\check{H}^1(\mathcal{U};\mathfrak{Im}(\exp)) \\
& [\{\eta_{rst}\}] & \mapsto & [\{\omega_{rs}\}]
\end{array}
\]
is injective, where $\omega_{rs} = \exp \Big ( 2\pi i \sum\limits_{t} \varphi_t \cdot \eta_{rst}\Big )$.
\end{corollary}

When going from
$\check{H}^1(\mathcal{U};\mathfrak{Im}(\exp))$ to $\check{H}^1(\mathcal{U};\mathscr{C}_\C^\times)$
one considers the inclusion  of presheaves $\jmath: \mathfrak{Im}(\exp) \longrightarrow \mathscr{C}_\C^\times$
and its induced homomorphism in cohomology
\[
\jmath_\mathcal{U} : \check{H}^1(\mathcal{U};\mathfrak{Im}(\exp)) \longrightarrow \check{H}^1(\mathcal{U};\mathscr{C}_\C^\times)
\]
After taking direct limits over refinements of $\mathcal{U}$,
the resulting homomorphism
is
an isomorphism  \cite[Proposition 7, section 25]{serre1955faisceaux}.
That is, each element in $\ker(\jmath_\mathcal{U})$ is also in the kernel of
$\check{H}^1(\mathcal{U};\mathfrak{Im}(\exp)) \longrightarrow \check{H}^1(\mathcal{V};\mathfrak{Im}(\exp))$
for some refinement  $\mathcal{V} $ of $\mathcal{U}$;
and  for every element in $\check{H}^1(\mathcal{U};\mathscr{C}_\C^\times)$ there exists
a refinement $\mathcal{W}$ of $\mathcal{U}$ so that the  image
 of said element via $\check{H}^1(\mathcal{U};\mathscr{C}_\C^\times) \longrightarrow \check{H}^1(\mathcal{W};\mathscr{C}_\C^\times) $
is also in the image of $\check{H}^1(\mathcal{W};\mathfrak{Im}(\exp)) \longrightarrow \check{H}^1(\mathcal{W};\mathscr{C}_\C^\times)$.

The situation is sometimes simpler.
Recall that a topological space is said to be simply connected if
it is path-connected and its fundamental group is trivial.
In addition, it is said to be locally path-connected if each point
has a path-connected open neighborhood.

\begin{lemma} Let $\mathcal{U} = \{U_r\}$ be an open  covering of $B$
such that each $U_r$ is locally path-connected  and simply connected.
Then
\[
\jmath_\mathcal{U} : \check{H}^1(\mathcal{U};\mathfrak{Im}(\exp))
\longrightarrow
\check{H}^1(\mathcal{U};\mathscr{C}^\times_\C)
\]
is injective.
\end{lemma}
\begin{proof}
Let $[\{\omega_{rt}\}] \in \check{H}^1(\mathcal{U};\mathfrak{Im}(\exp))$
be an element in the kernel of $\jmath_\mathcal{U}$.
Then there exists a collection of continuous maps
\[
\nu_r : U_r \longrightarrow \C^\times
\]
so that $\omega_{rt} = \frac{\nu_t}{\nu_r}$ on $U_r \cap U_t \neq \emptyset$.
If we let
\[
\begin{array}{rccl}
p : &\R_{+} \times \R  &\longrightarrow & \C^\times \\
&(\rho,\theta)& \mapsto & \rho \cdot e^{2\pi i \cdot\theta}
\end{array}
\]
then it follows that $\left(\R_{+}\times \R\, ,\, p\right)$ is the universal cover
for $\C^\times$.
Moreover, since each $U_r$ is  locally path-connected
and simply connected, then each $\nu_r$
has a lift \cite[Proposition 1.33]{hatcher2002algebraic}
\[
\begin{array}{rccl}
\gor{\nu}_r : & U_r&  \longrightarrow & \R_{+} \times \R \\
&b& \mapsto & \big(\rho_r(b), \theta_r (b)\big)
\end{array}
\]
That is $p \circ \gor{\nu}_r(b )= \nu_r(b)$ for all $b\in U_r$.
Let $\phi_r : U_r \longrightarrow \C$ be defined as
\[
\phi_r(b) = \theta_r(b) - i\frac{\ln(\rho_r(b))}{2\pi}
\]
It follows that $\{\phi_r\} \in \check{C}^0(\mathcal{U};\mathscr{C}_\C)$
and that for all $b\in U_r$
\begin{eqnarray*}
\exp\big(2\pi i \cdot \phi_r(b)\big) &= & \exp\big(\ln(\rho_r(b)) + 2\pi i\cdot \theta_r(b)\big) \\
&=& \rho_r(b)\cdot e^{2\pi i \cdot \theta_r (b)} \\
&=& \nu_r (b)
\end{eqnarray*}
Therefore $\nu_r = \exp(2\pi i \cdot \phi_r)$ and
$\{\nu_r\} \in \check{C}^0(\mathcal{U};\mathfrak{Im}(\exp))$,
which implies   $[\{\omega_{rt}\}] = 0 $ in $\check{H}^1(\mathcal{U};\mathfrak{Im}(\exp))$ as claimed.
\end{proof}

In summary, given an open cover $\mathcal{U}$ of  $B$  we get the function
\[
\begin{tikzcd}[column sep = 1.6em]
\mathbf{c}_1^\mathcal{U}:H^2(\mathcal{N}(\mathcal{U});\Z) \arrow[shorten <= -1em, shorten >= -1.1em]{r}{\Delta^{-1}\circ\Psi_\mathcal{U}} \;\;\;
&
\;\;\;
\check{H}^1(\mathcal{U}; \mathfrak{Im}(\exp)) \arrow{r}{\jmath_\mathcal{U}}
&
\check{H}^1(\mathcal{U}; \mathscr{C}_\C^\times) \arrow{r}{\Gamma_\mathcal{U}}
&
\left[B,\CP^\infty\right]
\end{tikzcd}
\]
which is natural with respect to refinements and satisfies:
\begin{corollary}
Let $\mathcal{U} = \{U_j\}$ be an open  cover of $B$
such that each $U_r$ is locally path-connected  and simply connected, and each $U_r \cap U_t$ is either empty or connected.
Then
\[
\mathbf{c}_1^\mathcal{U}:H^2(\mathcal{N}(\mathcal{U});\Z)
\longrightarrow
\big[ B \, , \, \CP^\infty \big]
\]
is injective.
\end{corollary}

%
We summarize the results of this section in the following theorem:
\begin{theorem}\label{thm:ClassifyingMapFormula2} Let $\mathcal{U} = \{U_0 ,\ldots, U_n\}$ be an open cover of $B$, and let $\{\varphi_r\}$ be
a partition of unity dominated by $\mathcal{U}$. Then we have functions
\[
\begin{array}{rccc}
\hspace{-0.2cm}\ww^\mathcal{U}_1 : & H^1(\mathcal{N}(\mathcal{U});\Z/2) &\longrightarrow & \big[B\, , \, \RP^\infty \big] \\[.1cm]
&[\tau= \{\tau_{rt}\}] & \mapsto & [f_\tau]
\end{array}
\mbox{ , }
\begin{array}{rccc}
\mathbf{c}^\mathcal{U}_1 : & H^2(\mathcal{N}(\mathcal{U});\Z) & \longrightarrow & \big[B\, , \, \CP^\infty \big]\\[.1cm]
&[\eta= \{\eta_{rst}\}] & \mapsto & [f_\eta]
\end{array}
\]
natural w.r.t refinements of $\mathcal{U}$,
where $\tau_{rt} = \tau(\{r,t\})$, $\eta_{rst} = \eta(\{r,s,t\})$ and
\[
\begin{array}{rccl}
f_\tau : & B &\longrightarrow & \RP^n \\
& U_j \ni b & \mapsto & \left[(-1)^{\tau_{0j}}\sqrt{\varphi_0(b)}: \cdots :  (-1)^{\tau_{nj}}\sqrt{\varphi_n(b)} \right] \\ \\
f_\eta : & B &\longrightarrow & \CP^n \\
& U_j \ni b & \mapsto & \left[e^{2\pi i\sum\limits_t \varphi_t(b) \eta_{0jt}}\sqrt{\varphi_0(b)} \;: \cdots :\;
e^{2\pi i \sum\limits_t \varphi_t(b)\eta_{njt}}\sqrt{\varphi_n(b)} \right]
\end{array}
\]
are well-defined.
Moreover,
if each $U_r$ is  connected,
then
$
\ww_1^\mathcal{U}$
is injective;
if in addition each $U_r$ is locally path-connected  and simply connected, and each $U_r \cap U_t$ is either empty or connected, then
$
\mathbf{c}_1^\mathcal{U}
$
is injective.
\end{theorem}

\section{Dimensionality Reduction in $\FP^n$ via Principal Projective Coordinates}
\label{sec:ProjectiveCoordinates}

Let $V \subset \F^{n+1}$ be a linear subspace with $\dim(V) \geq 1$.
If $\sim$ is the equivalence relation on $\F^{n+1}\smallsetminus \{\mathbf{0}\}$
given by $\uu \sim \vv$ if and only if $\uu = \lambda \vv$ for some $\lambda \in \F$,
then  $\sim$ is also an equivalence relation on
$V\smallsetminus\{\mathbf{0}\}$ and hence we can define
\[
\FP^{\dim(V)-1}_V := \big(V \smallsetminus \{\mathbf{0}\}\big)\big/\sim
\]
In particular $\FP^{\dim(U)-1}_U = \FP^{\dim(V)-1}_V$ if and only if $U = V$,
and  $\FP_V^{\dim(V)-1}$ is a subset of $\FP^n$.
Recall that $\F$ is either $\R$ or $\C$.
For $\uu,\vv \in \F^{n+1}$ let
\[
\langle\uu,\vv\rangle = \sum_{r=0}^n u_r \cdot \bar{v}_r
\]
denote their inner product.
If $\mathbf{d}_g$ denotes the geodesic distance in $\FP^n$ induced by the Fubini-Study metric,
then one has that
\[
\mathbf{d}_g([\uu],[\vv]) = \arccos \left(\frac{|\langle\uu,\vv\rangle|}{\|\uu\| \cdot \|\vv\|}\right)
\]
and it can be checked that $\FP^{\dim(V)-1}_V$ is an isometric copy of $\FP^{\dim(V)-1}$ inside
$\FP^n$.

If $1\leq \dim(V)\leq n$  and
$V^\perp = \{\uu\in \F^{n+1} : \langle\uu,\vv\rangle= 0 \mbox{ for all }\vv\in V\}$,
then the orthogonal projection
$p_V : \F^{n+1} \longrightarrow V$ descends to a continuous
map
\[
\begin{array}{rrcl}
P_V : &\FP^{n} \smallsetminus
\FP_{V^\perp}^{n -\dim(V)} &\longrightarrow &
\FP_V^{\dim(V)-1} \\
& \, [\uu] & \mapsto& [\,p_V(\uu)]
\end{array}
\]
Recall that $p_V$ sends each $\uu \in \F^{n+1}$ to its closest point in
$V$ with respect to the distance induced by $\|\cdot\|$. A similar property
is inherited by $P_V$:
\begin{proposition}
If $[\ww] \in \FP^n\smallsetminus
\FP_{V^\perp}^{n-\dim(V)}$,
then $P_V([\ww])$ is the point in $\FP_V^{\dim(V)-1}$
which is closest to $[\ww]$
 with respect to $\mathbf{d}_g$.
\end{proposition}
\begin{proof}
Let $[\uu] \in \FP^{\dim(V)-1}_V$.
Since $[\ww] \notin \FP_{V^\perp}^{n - \dim(V)}$,
then $\ww - p_V(\ww)\in V^\perp$
with $p_V(\ww)\neq \mathbf{0}$. Therefore  $\langle\uu, \ww - p_U(\ww)\rangle = 0$
and by the Cauchy-Schwartz inequality
\[
\left| \left\langle\ww, \uu\right\rangle\right| = |\langle p_U(\ww),\uu\rangle| \leq \|p_U(\ww)\|\cdot \|\uu\|
\]
Hence
\[
\frac{|\langle\ww,\uu\rangle|}{\|\ww\|\cdot \|\uu\|}
\leq \frac{\|p_U(\ww)\|}{\|\ww\|}
= \frac{|\langle\ww,p_U(\ww)\rangle|}{\|\ww\|\cdot \|p_U(\ww)\|}
\]
and since  $\arccos(\alpha)$ is  decreasing, then
$\mathbf{d}_g([\ww],P_U([\ww])) \leq  \mathbf{d}_g([\ww],[\uu])$.
\end{proof}

Therefore, we can think of $P_V$ as the projection onto $\FP_{V}^{\dim(V)-1}$.
Moreover, let $\jmath_V : \FP_V^{\dim(V)-1} \hookrightarrow \FP^n \smallsetminus
\FP^{n - \dim(V)}_{V^\perp}$ be the inclusion map.

\begin{proposition} $\jmath_V \circ P_V $ is a deformation retraction.
\end{proposition}
\begin{proof}
Since $p_V$ is surjective and satisfies $p_V \circ p_V (\ww) = p_V(\ww)$
for all $\ww \in \F^{n+1}$, it follows that $P_V$ is a retraction.
Let $h : \F^{n+1} \times [0,1] \longrightarrow \F^{n+1}$ be given
by $h(\ww, t) = (1-t)\cdot \ww + t\cdot p_V(\ww) $.
Since
$h(\ww,t) = \mathbf{0}$ implies that $\ww \in V^\perp$,
then $h$ induces a continuous map
\[
\begin{array}{rcl}
 \left(\FP^n \smallsetminus \FP_{V^\perp}^{n - \dim(V)}\right)\times [0,1] & \longrightarrow & \FP^n \smallsetminus \FP_{V^\perp}^{n - \dim(V)} \\[.1cm]
 \big([\ww], t\big) & \mapsto & [h(\ww,t)]
\end{array}
\]
which is a homotopy between the identity of $\FP^n \smallsetminus \FP_{V^\perp}^{n - \dim(V)}$
and $\jmath_V \circ P_V$.
\end{proof}
Notice that $[\uu] = \mathsf{Span}(\uu)$ for $\uu \in \F^{n+1} \smallsetminus \{\mathbf{0}\}$.
The previous proposition yields

\begin{corollary} Let $f: B \longrightarrow \FP^n$ be a continuous map which is not surjective.
If $[\uu] \notin f(B)$, then $f$ is homotopic
to $P_{[\uu]^\perp} \circ f\; :\; B \longrightarrow \FP_{[\uu]^\perp}^{n-1}\subset \FP^n$.
\end{corollary}

In summary,  if $f: B \longrightarrow \FP^n$ is not surjective,
then it can be continuously deformed so that its image lies in
$\FP^{n-1}_{[\uu]^\perp} \subset \FP^n$,
for $[\uu]\notin f(B)$.
Moreover, the deformation is obtained by sending each $f(b) \in \FP^n$
to its closest point in $\FP^{n-1}_{[\uu]^\perp}$ with respect to
$\mathbf{d}_g$, along a shortest path in  $\FP^n$.
This analysis shows that the topological properties encoded
by $f$
are preserved by the dimensionality reduction step if
$f(B) \neq \FP^n$.

Given a  finite  set $\YY \subset \FP^n$, we will show next
that  $\uu$
can be chosen so that $\FP^{n-1}_{[\uu]^\perp}$ provides the best
$(n-1)$-dimensional approximation. Indeed,
given
\[\mathbf{Y} = \big\{ [\yy_1],\ldots, [\yy_N]\big\} \subset \FP^n\]
the goal is to find $\uu^* \in \F^{n+1}$  so that
\[
\uu^* = \argmin_{\substack{\uu \in \F^{n+1} \\[.1cm] \|\uu\| =1}}
\sum_{r=1}^N \mathbf{d}_g
\left([\yy_r], \FP^{n-1}_{[\uu]^\perp}\right)^2
\]
Since
\[
\mathbf{d}_g\left([\yy_r], \FP^{n-1}_{[\uu]^\perp}\right) =
\mathbf{d}_g\left([\yy_r], P_{[\uu]^\perp}([\yy_r])\right) =
\mathbf{d}_g\left([\yy_r], \big[\yy_r - \langle\yy_r,\uu\rangle\uu\big]\right)
\]
then
\begin{eqnarray}\label{eq:formulationLastPPC}
\uu^* &=& \argmin_{\substack{\uu \in \F^{n+1} \\[.1cm] \|\uu\| =1}}
\sum_{r=1}^N
\arccos\left(\frac{\left|\big\langle\yy_r, \yy_r - \langle\yy_r,\uu\rangle\uu\big\rangle \right|}
{\|\yy_r\|\cdot \|\yy_r - \langle\yy_r , \uu\rangle\uu\|}\right)^2 \nonumber \\ \nonumber \\
&=&
\argmin_{\substack{\uu \in \F^{n+1} \\[.1cm] \|\uu\| =1}}
\sum_{r=1}^N
\arccos\left(\frac{ \|\yy_r - \langle\yy_r , \uu\rangle\uu\|}
{\|\yy_r\|}\right)^2 \nonumber \\ \nonumber \\
&=&
\argmin_{\substack{\uu \in \F^{n+1} \\[.1cm] \|\uu\| =1}}
\sum_{r=1}^N
\left(\frac{\pi}{2} - \arccos
\left(\frac{\big|\langle\yy_r,\uu\rangle\big|}{\|\yy_r\|}\right)
\right)^2
\end{eqnarray}

This nonlinear least squares problem --- in a nonlinear domain --- can be solved  approximately using linearization;  the reduction, in turn, has a closed form solution.
Indeed,
the Taylor series expansion for $\arccos(\alpha)$ around 0 is
\[
\arccos(\alpha) = \frac{\pi}{2} - \left( \alpha +
\sum_{\ell=1}^\infty \frac{(2\ell)!}{4^\ell(\ell!)^2}\frac{\alpha^{2\ell+1}}{2\ell+1} \right) \;\; , \;\;\; |\alpha| < 1
\]
and therefore $\left|\frac{\pi}{2}-\arccos(\alpha)\right| \approx |\alpha|$
is a third order approximation.
Hence
\begin{equation}\label{eq:formulationLastPPC_Approx}
\uu^* \approx
\argmin_{\substack{\uu \in \F^{n+1} \\[.1cm] \|\uu\| =1}}
\sum_{r=1}^N \left|\frac{\langle\yy_r,\uu\rangle}{\|\yy_r\|}\right|^2
\end{equation}
which is a linear least squares problem, and a solution is
the  eigenvector of the
($n+1$)-by-($n+1$) uncentered covariance matrix
\[
\mathsf{Cov}\left(\frac{\yy_1}{\|\yy_1\|},\ldots, \frac{\yy_N}{\|\yy_N\|}\right) =
\Mat{| &  & | \\ \frac{\yy_1}{\|\yy_1\|} &\cdots & \frac{\yy_N}{\|\yy_N\|} \\ | &  & | }
\cdot
\Mat{ \mbox{\textemdash} & \frac{\bar{\yy}_1}{\|\bar{\yy}_1\|}& \mbox{\textemdash} \\
& \vdots &  \\
\mbox{\textemdash} & \frac{\bar{\yy}_N}{\|\bar{\yy}_N\|} & \mbox{\textemdash}}
\]
corresponding to the smallest eigenvalue.
Notice that if
$a_1 ,\ldots, a_N \in \F$ satisfy
$|a_r| = 1$ for each $r=1,\ldots, N$,
then
\[
\mathsf{Cov}\left(\frac{\yy_1}{\|\yy_1\|},\ldots, \frac{\yy_N}{\|\yy_N\|}\right) = \mathsf{Cov}\left(a_1\frac{\yy_1}{\|\yy_1\|},\ldots, a_N\frac{\yy_N}{\|\yy_N\|}\right)
\]
and hence we can write $\mathsf{Cov}(\mathbf{Y})$ for the unique uncentered covariance matrix
associated to $\mathbf{Y} = \big\{[\yy_1],\ldots, [\yy_N]\big\} \subset \FP^n$.
If $\uu\in \F^{n+1}$ is an eigenvector of $\mathsf{Cov}(\mathbf{Y})$ corresponding to the smallest eigenvalue, then
we use the notation
$\left[\uu \right] = {\tt LastProjComp}\left(\YY,\FP^n\right)$
with the understanding that $[\uu]$ is unique  only if the relevant
eigenspace has dimension one.
If not, the choice is arbitrary.

\subsection*{Principal Projective Coordinates}
First we define, inductively, the Principal Projective Components of $\YY$.
Starting with  $\left[\vv_n\right] =$\texttt{LastProjComp}$(\YY,\FP^n)$,
assume that for $1 \leq k \leq n-1$ the  components
$\left[\vv_{k+1}\right],\ldots, \left[\vv_{n}\right]\in \FP^{n}$
have been determined and let us define $[\vv_k]$.
To this end,
let $\{\uu_0,\ldots, \uu_k\}$ be an orthonormal basis for
$V^k = \mathsf{Span}\left(\vv_{k+1},\ldots, \vv_n\right)^\perp$,
let
\[
A_k = \Mat{|& & | \\ \uu_0 & \cdots & \uu_k \\ | & & |}
\]
and let $A_k^{\dag}$ be its conjugate transpose.
If $A_k^\dag \cdot \YY
= \left\{\left[A_k^\dag\yy_1 \right],\ldots, \left[A_k^\dag\yy_N \right]\right\}$, define
\[
[\vv_k]:= A_k \cdot {\tt LastProjComp}\left(A_k^\dag\cdot \YY \; , \; \FP^k\right)
\]
This is well-defined as the following proposition shows.
\begin{proposition} The class
$
[\vv_k] = A_k\cdot {\tt LastProjComp}\left(A_k^\dag \cdot \YY\;,\; \FP^k\right)
$ is independent of the choice of orthonormal basis $\{\uu_0,\ldots, \uu_k\}$.
\end{proposition}
\begin{proof} Let $\{\ww_0,\ldots, \ww_k\}$ be another orthonormal basis
for $V^k$ and let
\[
B_k = \Mat{| & & | \\ \ww_0 & \cdots & \ww_k \\ | & & |}
\]
It follows that
$B_k^\dag B_k = A_k^\dag A_k = I_{k+1}$, the
$(k+1)$-by-$(k+1)$ identity matrix,
and that $B_kB_k^\dag  = A_kA_k^\dag$ is the matrix
(with respect to the standard basis of $\F^{n+1}$)
of the orthogonal projection
$p_{V^k}: \F^{n+1} \longrightarrow V^k$.
Therefore
\begin{eqnarray*}
\left(
B^\dag_kA_k
\right)
\left(
A_k^\dag B_k
\right)
&=&
B^\dag_k \left(A_k  A_k^\dag\right) B_k \\
&=& B^\dag_k \left(B_k  B_k^\dag\right) B_k \\
&=& I_{r+1}
\end{eqnarray*}
which shows that $A^\dag_kB_k$
is an orthogonal matrix.
Since
\[
\left\|A_k^\dag \yy\right\|^2
=
\left\langle\yy, A_k A_k^\dag \yy\right\rangle
=
\left\|B_k^\dag \yy\right\|^2
\]
for every $\yy \in \F^{n+1}$,
then
\begin{eqnarray*}
\mathsf{Cov}\left(A^\dag_k \YY\right)
&=&
\mathsf{Cov}
\left(
\frac{A^\dag_k \yy_1}{\left\|A^\dag_k\yy_1\right\|}
, \ldots ,
\frac{A^\dag_k \yy_N}{\left\|A^\dag_k\yy_N\right\|}
\right) \\
&=&
\mathsf{Cov}
\left(A^\dag_kB_k
\frac{B^\dag_k \yy_1}{\left\|B^\dag_k\yy_1\right\|}
, \ldots ,
A^\dag_kB_k
\frac{B^\dag_k \yy_N}{\left\|B^\dag_k\yy_N\right\|}
\right) \\
&=&
A^\dag_kB_k \cdot
\mathsf{Cov}\left(B_k^\dag \YY\right)
\cdot
B^\dag_kA_k
\end{eqnarray*}
and  thus
$\mathsf{Cov}\left(A^\dag_k\YY\right)$
and $\mathsf{Cov}(B_k^\dag\YY)$ have the same spectrum.
Moreover,
$\uu$ is an eigenvector
of $\mathsf{Cov}\left(A^\dag_k\YY\right)$
corresponding to the smallest eigenvalue $\lambda$
if and only if
$\uu = A^\dag_kB_k\ww$
for a unique eigenvector $\ww$ of $\mathsf{Cov}\left(B^\dag_k\YY\right)$
with eigenvalue $\lambda$.
Since $B_k\ww \in V^k$
and $A_kA^\dag_k$ is the matrix of $p_{V^k}$, then
\[
A_k \uu = A_kA_k^\dag B_k \ww = B_k \ww
\]
which shows that
\[
A_k\cdot {\tt LastProjComp}\left(A_k^\dag \cdot\YY\;,\; \FP^k\right)
=
B_k\cdot {\tt LastProjComp}\left(B_k^\dag \cdot\YY\;,\; \FP^k\right)
\]
\end{proof}
This inductive procedure defines $[\vv_1],\ldots, [\vv_n] \in \FP^n$,
and we let $\vv_0 \in \F^{n+1}$ with $\|\vv_0\| = 1$ be
so that $\mathsf{Span}(\vv_0) = \mathsf{Span}(\vv_1,\ldots, \vv_n)^\perp$.
We will use the notation
\[
{\tt PrinProjComps}(\YY) = \big\{[\vv_0],\ldots, [\vv_n]\big\}
\]
for the principal projective components of $\YY$ computed in this fashion.
Each choice of unitary (i.e. having norm 1) representatives  $\uu_0 \in [\vv_0],\ldots, \uu_n \in [\vv_n]$
yields an orthonormal basis
$ \{\uu_0,\ldots, \uu_n\}$ for $\F^{n+1}$, and
each $\yy \in \F^{n+1}$ can be represented in terms of its vector of coefficients
\[
\mathsf{coeff}_U(\yy) =  \Mat{\langle\yy,\uu_0\rangle \\ \vdots \\ \langle\yy,\uu_n\rangle}
\]
If $\gor{U}$ is another set of unitary representatives
for ${\tt PrinProjComps}(\YY)$, then there exists a $(n+1)$-by-$(n+1)$
diagonal
matrix $\Lambda$, with entries in the unit circle in $\F$, and
so that
$\mathsf{coeff}_{\gor{U}}(\yy) = \Lambda\cdot \mathsf{coeff}_U(\yy)$.
That is, the resulting principal projective coordinates
$[\mathsf{coeff}_U(\yy)] \in \FP^n$ are unique
up to a diagonal  isometry.

\subsection*{Visualizing the Reduction}
Fix a set of unitary representatives $\vv_0,\ldots, \vv_n$
for ${\tt PrinProjComps}(\YY)$, and
let $V^k = \mathsf{Span}(\vv_0,\ldots, \vv_k)$
for $1\leq k \leq n$.
It is often useful to visualize
$P_{ V^k}(\mathbf{Y})\subset \FP_{V^k}^k$
for $k$ small,
specially in $\RP^1$, $\RP^2$, $\RP^3$ and $\CP^1$.
We do this using the principal projective coordinates of $\YY$.
For the real case (i.e. $\RP^1$, $\RP^2$ and $\RP^3$) we consider the set
\[
\left\{
\frac{\mathbf{x}_r}{\|\mathbf{x}_r\|} \in S^k \; \; \bigg| \; \;
\mathbf{x}_r \;=\; \mathsf{sign}\left(\langle\yy_r, \vv_0\rangle\right)\cdot\Mat{\langle\yy_r , \vv_0\rangle \\ \vdots \\ \langle\yy_r , \vv_{k}\rangle }
\;\; , \;\;\; r=1,\ldots, N
\right\}\;, \; k \leq 3
\]
and its image through the stereographic projection
$
S^k \smallsetminus\{-\mathbf{e}_1\} \longrightarrow D^k $
with respect to $-\mathbf{e}_1$, where $\mathbf{e}_1$ is
the first standard basis vector $\mathbf{e}_1 \in \R^{k+1}$.
That is, we visualize $P_{V^k}(\mathbf{Y})$ in the $k$-disk  $ D^k\subset \R^k$
with the understanding that antipodal points on the boundary are identified.
For the complex case (i.e. $\CP^1$) we consider   the set
\[
\left\{
\frac{\mathbf{z}_r}{\|\mathbf{z}_r\|} \in \C^2 \; \;:\; \;
\mathbf{z}_r = \Mat{\langle\yy_r , \vv_0\rangle \\ \langle\yy_r, \vv_1\rangle  }
\;\; , \;\;\; r=1,\ldots, N
\right\}
\]
and its image through the H\"{o}pf map
\[
\begin{array}{rccl}
H :& S^{3}\subset \C^2& \longrightarrow &S^2 \subset \C \times \R \\
&[z_1,z_2]&\mapsto&(2z_1\bar{z_2}, |z_1|^2 - |z_2|^2)
\end{array}
\]
which is exactly the composition of
$S^3 \subset \C^2 \longrightarrow \C_\infty = \C\cup \{\infty\}$,
sending $[z_1,z_2]$ to $z_1/z_2$,
and
 the isometry $\C_\infty \cong S^2\subset \C\times \R$ given by the inverse of the north-pole
stereographic projection.

\subsection*{Choosing the Target Dimension}
Given $1\leq k \leq n$, the cumulative variance recovered
by $P_{V^k}(\YY)\subset \FP_{V^k}^k$
is given by the expression
\begin{eqnarray}\label{eq:cumVar}
\mathsf{var}_\mathbf{Y}(k) &=&
\frac{1}{N}
\sum_{\ell=1}^{k} \sum_{r=1}^{N}
\mathbf{d}_g
\left(
P_{V^\ell}([\yy_r]) , \FP^{\ell -1}_{V^{\ell -1}}
\right)^2 \nonumber \\
&=&
\frac{1}{N}
\sum_{\ell=1}^k \sum_{r=1}^N
\left(
\frac{\pi}{2}  - \mathbf{d}_g\left(P_{V^{\ell-1}}([\yy_r]), \left[\vv_{\ell}\right]\right)
\right)^2
\end{eqnarray}
Define the percentage of cumulative variance as
\begin{equation}\label{eq:pert_of_cumVar}
\mathsf{p.var}_\mathbf{Y}(k) =
\frac{
\mathsf{var}_\mathbf{Y}(k)
}
{\mathsf{var}_\mathbf{Y}(n)}
\end{equation}

A common rule of thumb for choosing the target dimension
is  identifying the smallest value of $k$
so that $\mathsf{p.var}_\YY$ exhibits a prominent reduction in growth rate.
Visually, this creates an ``elbow'' in the graph of $\mathsf{p.var}_\YY$
at $k$
(see Figure \ref{fig:PCA and ISOMAP dim_VS_variance}(b), $k=5$).
The target dimension can also be chosen as the smallest $k\geq 1$ so that
$\mathsf{p.var}_\YY(k)$
is greater than a predetermined threshold, e.g. $0.7$
(see Figure \ref{fig:projdDim_vs_variance}, $k=2$).

\subsection*{Examples} Let us illustrate
the inner workings of the framework we have developed thus far.

\subsubsection*{The Projective Plane $\RP^2$:}
Let $\RP^2$ denote the quotient $S^2/(\uu \sim -\uu)$, endowed
with the geodesic distance  $\mathbf{d}_g(\uu,\vv) = \arccos(|\langle\uu ,\vv\rangle|)$.
 We begin by selecting six landmark points  $\ell_0,\ldots, \ell_5 \in \RP^2$  as shown in Figure \ref{fig:RP2_landmkars}(Left). If for each landmark $\ell_j$ we let
$r_j = \min\{\mathbf{d}_g(\ell_j,\ell_r)  : r\neq j\}$
 and let $\epsilon_j = 0.95*r_j$,
then $\mathcal{U} = \{B_{\epsilon_j}(\ell_j)\}$ is a covering for $\RP^2$
and the corresponding nerve complex $\mathcal{N}(\mathcal{U})$
is  shown in Figure \ref{fig:RP2_landmkars}(Right).
\begin{figure}[!ht]
  \centering
  \subfigure{
  \includegraphics[scale = 0.3]{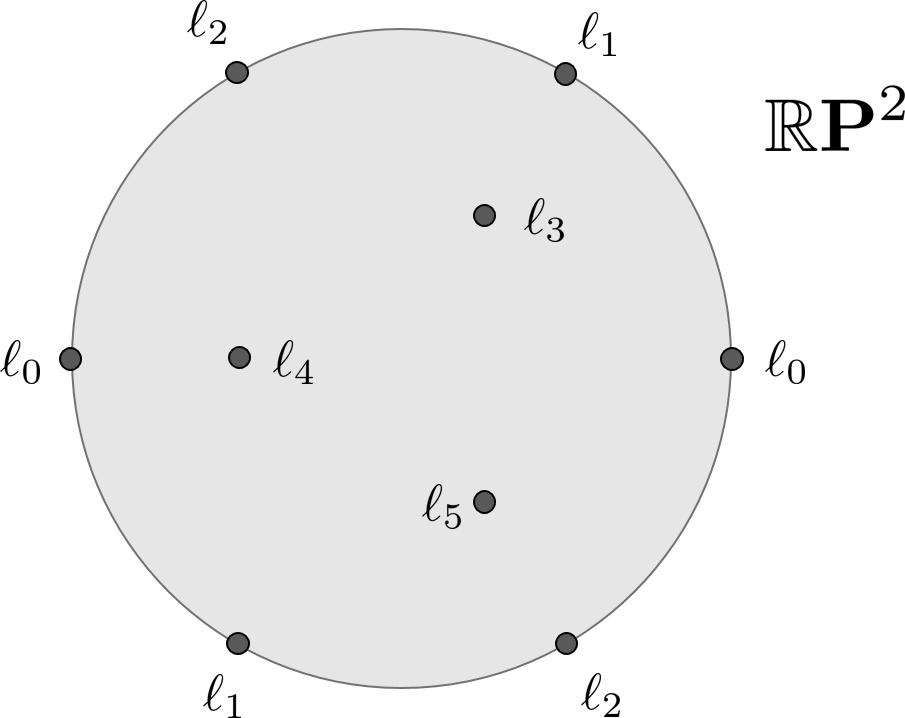}
  }
  \hspace{1cm}
  \subfigure{
  \includegraphics[scale = 0.3]{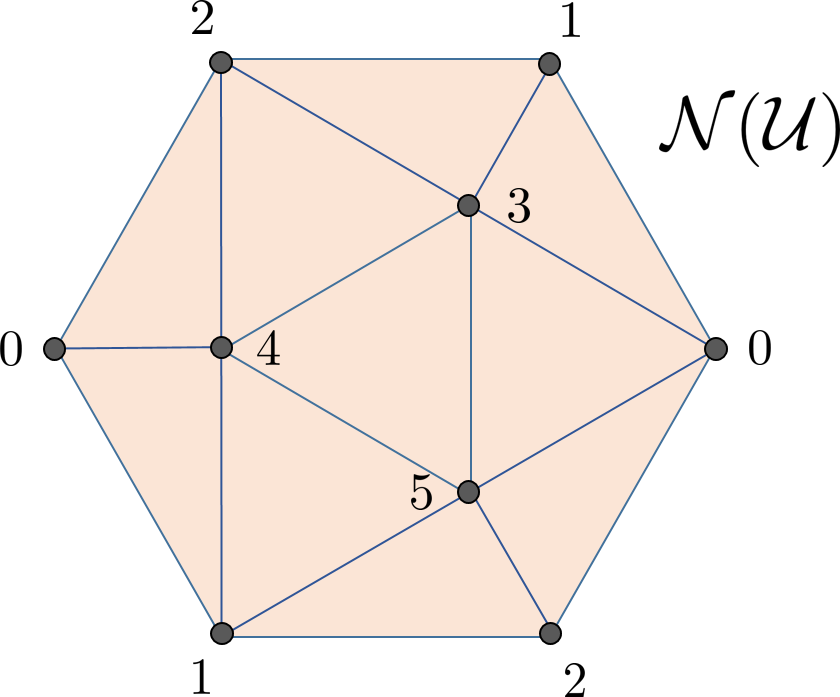}
  }
  \caption{\textbf{Left}: Landmark points on $\RP^2$, \textbf{Right}: Induced nerve complex $\mathcal{N}(\mathcal{U})$.}\label{fig:RP2_landmkars}
\end{figure}

Let
$\mathds{1}_{\{i,j\}} : \mathcal{N}(\mathcal{U})^{(1)} \longrightarrow \Z/2$
be the indicator function
$\mathds{1}_{\{i,j\}}(\{r,k\}) = 1 $ if $\{i,j\} = \{r,k\}$
and 0 otherwise.
Then
\[
\tau = \mathds{1}_{\{0,1\}} + \mathds{1}_{\{1,2\}} + \mathds{1}_{\{0,2\}}
+ \mathds{1}_{\{0,4\}} + \mathds{1}_{\{2,5\}} + \mathds{1}_{\{1,3\}}
\]
is  a 1-cocycle, and  its cohomology class $[\tau]$ is the non-zero
element in \[H^1(\mathcal{N}(\mathcal{U});\Z/2) \cong \Z/2\]
Using the formula from Theorem \ref{thm:ClassifyingMapFormula2},
the cocycle $\tau$ above, and  quadratic bumps
with weights $\lambda_j = \epsilon_j^2$,
we get the corresponding map $f_\tau : \RP^2 \longrightarrow \RP^5$.
For instance, if $b\in B_{\epsilon_0}(\ell_0)$, then
\begin{eqnarray*}
f_\tau(b) &=&
\left[
\ppart{\epsilon_0 - \mathbf{d}_g(b, \ell_0) } \;:\;
-\ppart{\epsilon_1 - \mathbf{d}_g(b, \ell_1) } \;:\;
-\ppart{\epsilon_2 - \mathbf{d}_g(b, \ell_2) } \;:\;\right. \\
&& \hspace{2cm }
\left.\ppart{\epsilon_3 - \mathbf{d}_g(b, \ell_3) } \;:\;
-\ppart{\epsilon_4 - \mathbf{d}_g(b, \ell_4) }\; :\;
\ppart{\epsilon_5 - \mathbf{d}_g(b, \ell_5) }
\right]
\end{eqnarray*}
Let $\XX \subset \RP^2$ be a uniform random sample with 10,000 points.
After computing the principal projective components of
$f_\tau(\XX)\subset \RP^5$ and the percentage of cumulative variance $\mathsf{p.var}_\YY(k)$
(see equation (\ref{eq:pert_of_cumVar}))
for $k=1,\ldots, 5$ we obtain the following:

\begin{figure}[!h]
  \centering
  \includegraphics[width =0.6\textwidth]{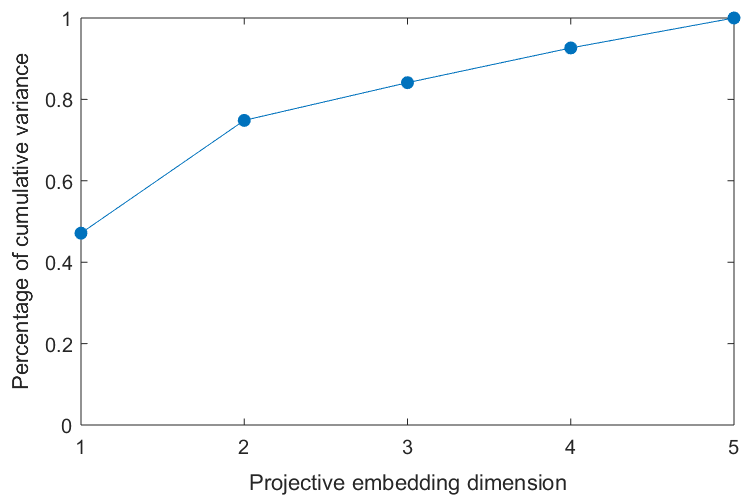}
  \caption{Percentage of cumulative variance}\label{fig:RP2_example_cumvar}
\end{figure}

This profile of cumulative variance suggests that dimension 2 is appropriate
for representing $f_\tau(\XX)$: both ``the elbow'' and the ``70\% of recovered
variance'' happen at around $k =2$.
Below in Figure \ref{fig:RP2_example_coordinates} we show the original sample $\XX \subset \RP^2$
as well as the point cloud $P_{V^2}(f_\tau(\XX))$ resulting from projecting  $f_\tau(\XX)$ onto $\RP^2_{V^2} \subset \RP^5$.
Recall that $P_{V^2}\big(f_\tau(\XX)\big)$ is visualized on the
unit disk $D^2 = \{\uu \in \R^2 \, : \, \|\uu\| \leq 1\}$,
with the understanding that points in the boundary $\partial D^2 = S^1$
are identified with their antipodes.
\begin{figure}[!ht]
  \centering
  \includegraphics[width=0.9\textwidth]{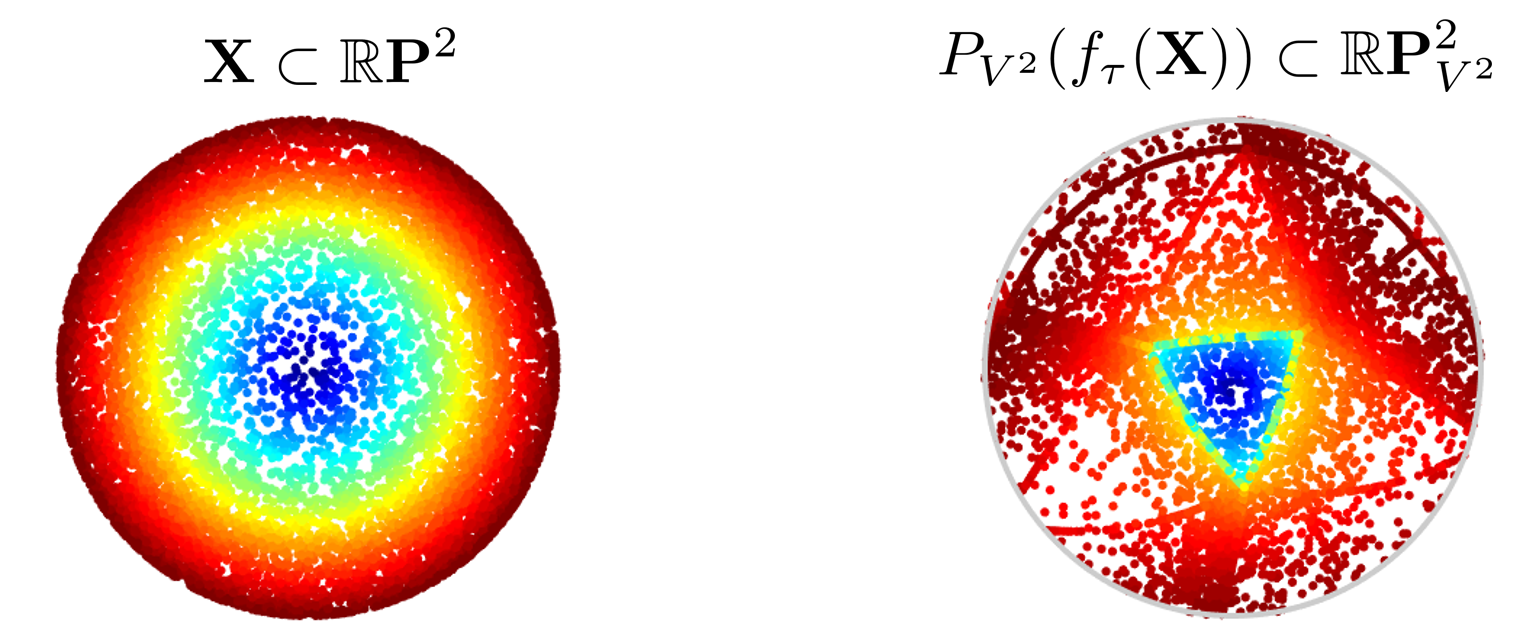}
  \caption{\textbf{Left:} Original sample $\XX \subset \RP^2$, \textbf{Right:} Visualization of resulting projective coordinates. Please refer to an electronic version for colors.}\label{fig:RP2_example_coordinates}
\end{figure}

These results are consistent with the fact that any
$f: \RP^2 \longrightarrow \RP^\infty$ which classifies the nontrivial bundle
over $\RP^2$, must be homotopic to the inclusion $\RP^2 \hookrightarrow \RP^\infty$.
So not only did we get the right homotopy-type, but also the global
geometry and the metric information were recovered to a large extent.

\subsubsection*{The Klein Bottle $K$:}
Let $K$ denote the quotient of the unit square
$[0,1]\times [0,1]$ by the relation $\sim$ given by
$(x,0) \sim (x,1)$ and $(0,y) \sim (1,1-y)$.
Let us endow $K$
 with the induced flat metric, which we denote by
$\mathbf{d}$, and let $\ell_0,\ldots, \ell_8 \in K$ be landmark points  selected
as shown in Figure \ref{fig:example_K_landmarks}(Left).
If for each landmark $\ell_j$ we let
\[\epsilon_j = \min\{\mathbf{d}(\ell_j, \ell_r) \; : \; j \neq r\}\]
then $\mathcal{U} = \{B_{\epsilon_j}(\ell_j)\}$ is a
covering for $K$ and the resulting nerve complex $\mathcal{N}(\mathcal{U})$ is shown in Figure \ref{fig:example_K_landmarks}(Right).
\begin{figure}[!ht]
  \centering
  \subfigure{
  \includegraphics[scale = 0.3]{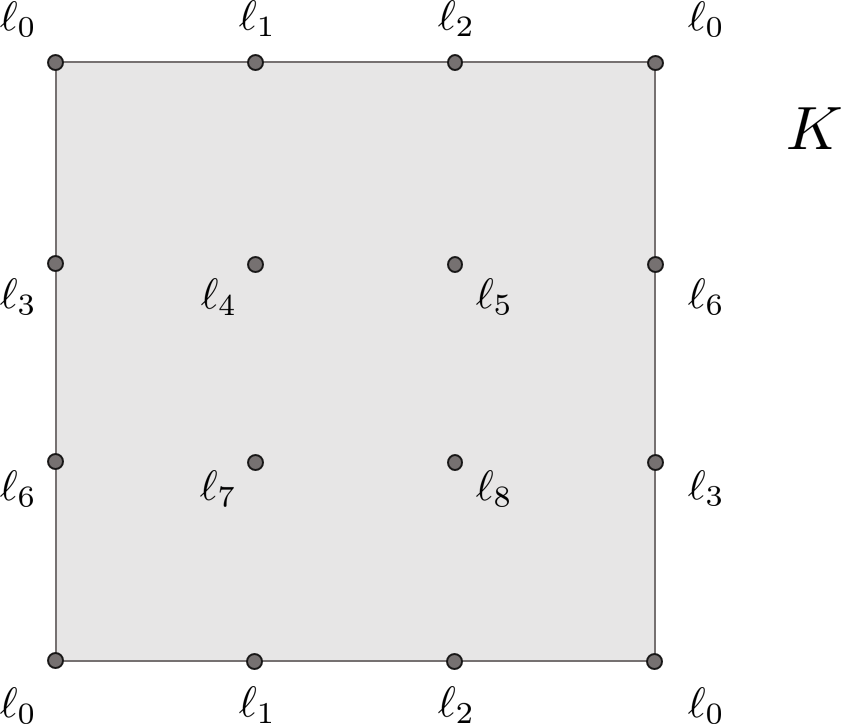}
  }
  \hspace{1.5cm}
  \subfigure{
  \includegraphics[scale = 0.3]{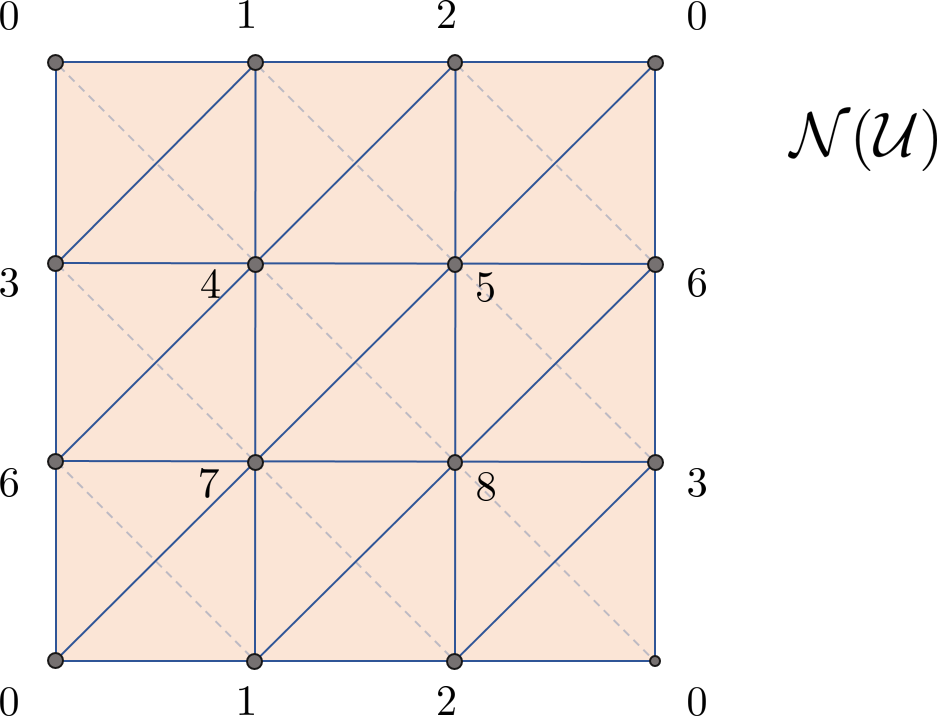}
  }
  \caption{\textbf{Lef:} Landmarks on the Klein bottle. \textbf{Right:} Induced nerve complex.}\label{fig:example_K_landmarks}
\end{figure}

It follows that the 1-skeleton of  $\mathcal{N}(\mathcal{U})$ is the
complete graph on nine vertices, and that there are
thirty-six 2-simplices and nine 3-simplices.
Let us  define the 1-chains
\[\tau_{\mathsf{diag}},\tau_{\mathsf{horz}},\tau_{\mathsf{vert}} \in C^1(\mathcal{N}(\mathcal{U});\Z/2)\] as follows:
$\tau_{\mathsf{diag}}$ will be the sum of indicator functions
on the diagonal edges,
$\tau_{\mathsf{horz}}$ is the sum of indicator functions on the
horizontal edges, and
$\tau_{\mathsf{vert}}$ will be the sum of indicator functios
on the vertical edges.
One can check that
\[
 \tau  =\tau_{\mathsf{diag}} + \tau_{\mathsf{horz}} \;\;\;\;\; \mbox{ and } \;\;\;\;\;
  \tau' = \tau_{\mathsf{diag}} + \tau_{\mathsf{vert}}
\]
are  coycles and that their cohomology classes
generate
\[H^1(\mathcal{N}(\mathcal{U});\Z/2) \cong \Z/2 \oplus \Z/2\]
Let $\XX \subset K$ be a random sample with 10,000 points.
The formula from Theorem \ref{thm:ClassifyingMapFormula2}
yields  classifying maps
\[
f_\tau, f_{\tau'}, f_{\tau + \tau'} : K \longrightarrow \RP^8
\]
and we obtain the point clouds
$f_\tau(\XX), f_{\tau'}(\XX), f_{\tau + \tau'}(\XX) \subset \RP^8$
of which we will compute their principal projective
coordinates.
Starting with $f_\tau(\XX)$ we get the profile of recovered
variance shown in Figure \ref{fig:example_K_cumvar}.

\begin{figure}[!ht]
  \centering
  \includegraphics[scale = 0.4]{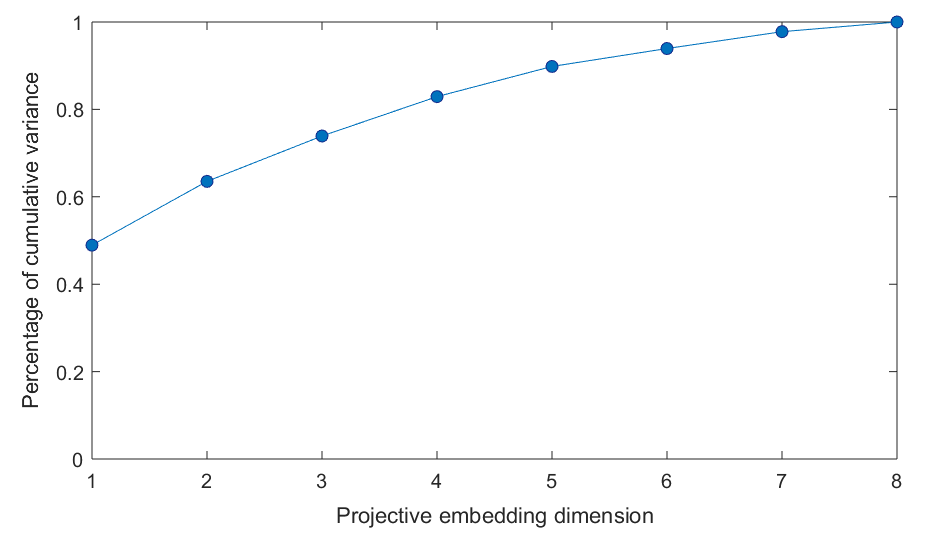}
  \caption{Percentage of cumulative variance for $f_\tau(\XX)$ }\label{fig:example_K_cumvar}
\end{figure}

The figure suggests that dimension 3 provides an appropriate
representation of $f_\tau(\XX) \subset \RP^8$.
As described above,
we visualize $P_{V^3}(f_\tau(\XX)) \subset \RP^3_{V^3} $
in the 3-dimensional unit disk
$D^3 = \{\uu \in \R^3 \, : \, \|\uu\| \leq 1\}$
with the understanding that points on the boundary
$\partial D^3 = S^2$ are identified with their antipodes.
The results are summarized in Figure \ref{fig:example_K_coordinates}.

\begin{figure}[!ht]
  \centering
  \includegraphics[width=0.9\textwidth]{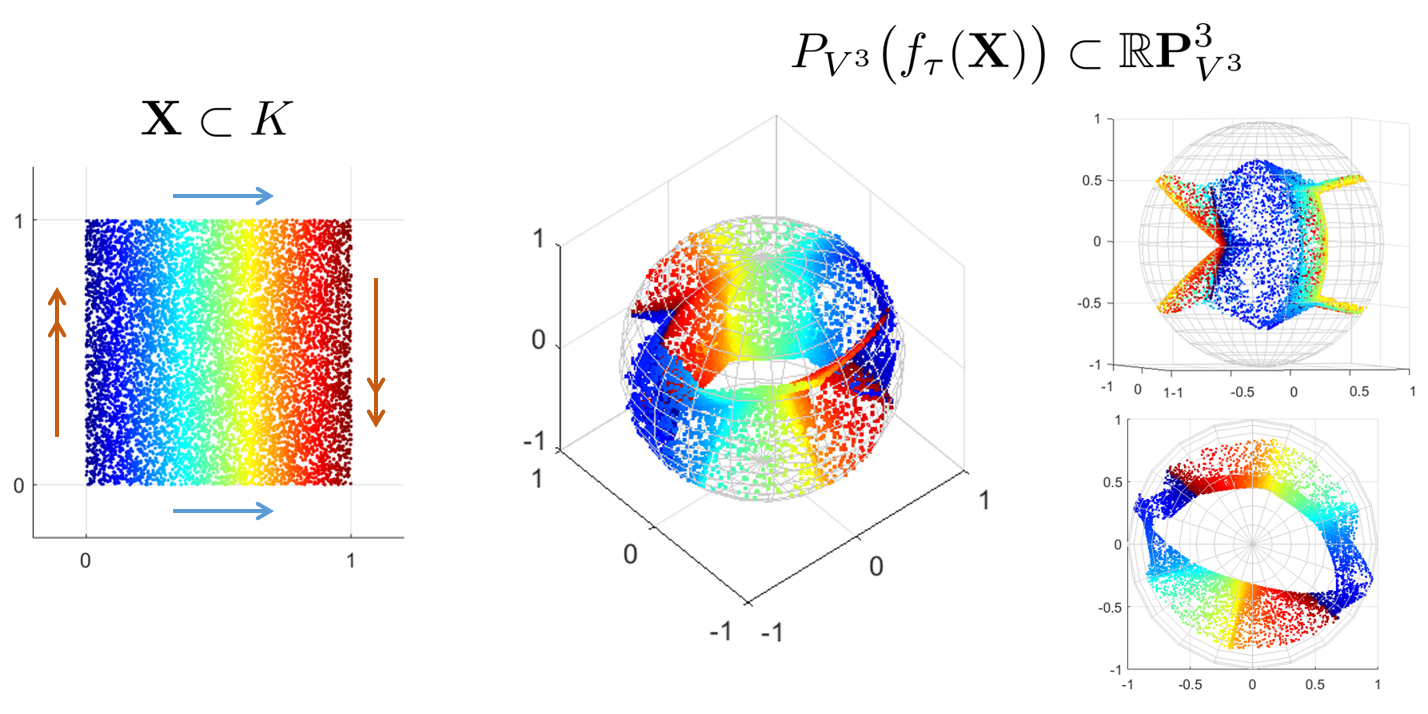}
  \caption{
\textbf{Left:} Original sample $\XX \subset K$; \textbf{Center:}
Visualization of $P_{V^3}\big(f_\tau(\XX)\big) \subset \RP^3_{V^3}$ in $D^3/\! \sim \;  = \RP^3$; \textbf{Right:} Lateral (upper right) and
 aerial view of the representation (lower right). Please refer to an electronic version for colors.}
  \label{fig:example_K_coordinates}
\end{figure}

This example highlights the following point:
when representing data sampled from complicated
spaces, e.g. the Klein bottle, it is advantageous to use
target spaces with similar properties.
In particular, the representation for $\XX \subset K$ we recover here
is much simpler than those obtained with traditional dimensionality
reduction methods.
We now transition to the 2-dimensional reduction $P_{V^2}(f_\tau(\XX)) \subset \RP^2_{V^2}$.
As before we visualize the representation in the 2-dimensional unit disk
$D^2$ with the understanding that points on the boundary
$\partial D^2 = S^1$ are identified with their antipodes.
\begin{figure}[!ht]
  \centering
  \includegraphics[width=0.8\textwidth]{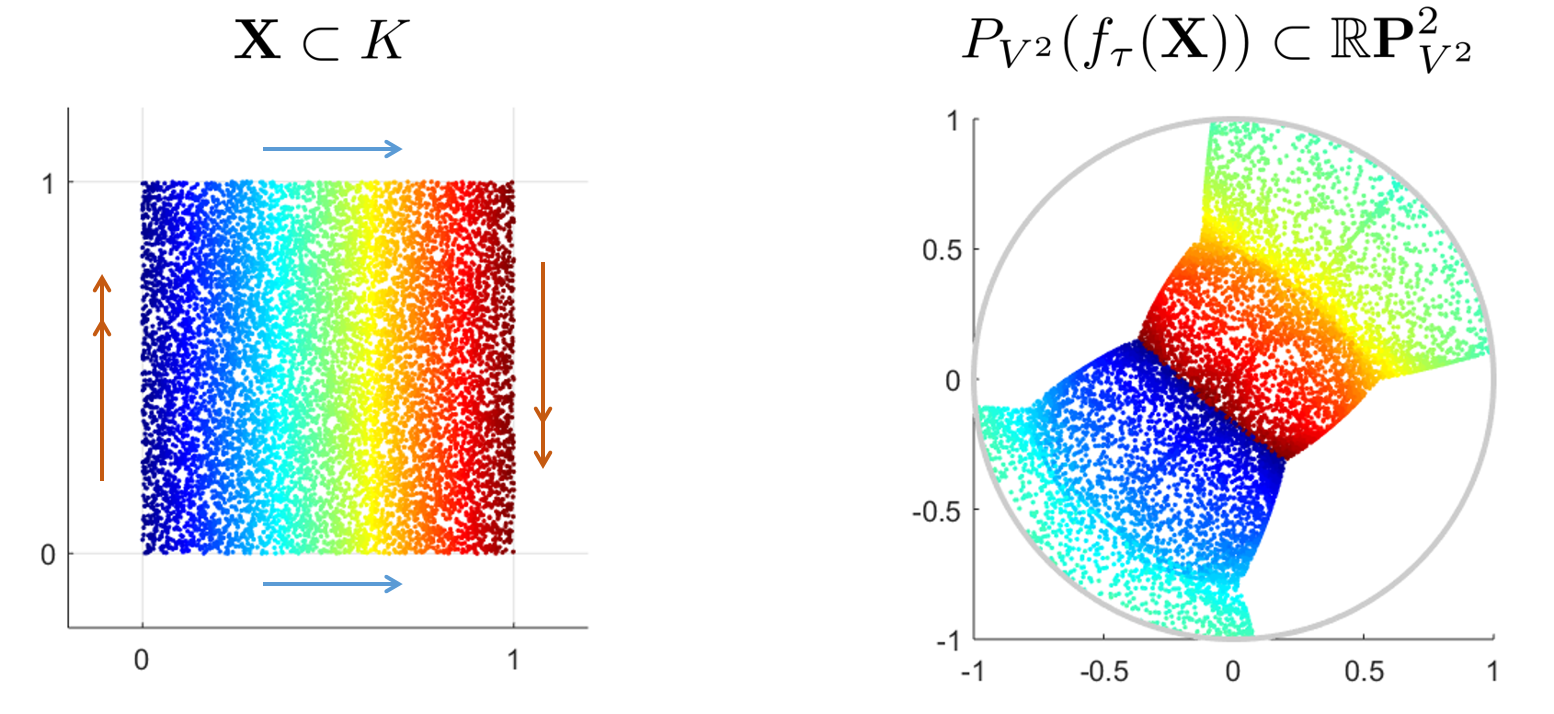}
  \caption{$\RP^2$-coordinates for $\mathbf{X}
  \subset K$ corresponding to the cocycle $\tau$. }\label{fig:example_K_coordinates_RP3a}
\end{figure}

We conclude this example by examining  the $\RP^2$ coordinates induced by
 $\tau'$ and $\tau + \tau'$ (Figure \ref{fig:exampel_K_coordinates_RP2a}).
For completeness we include the one for $\tau$ and also add
figures with coloring by the vertical direction in $K$.

\begin{figure}[!ht]
  \centering
  \includegraphics[width=\textwidth]{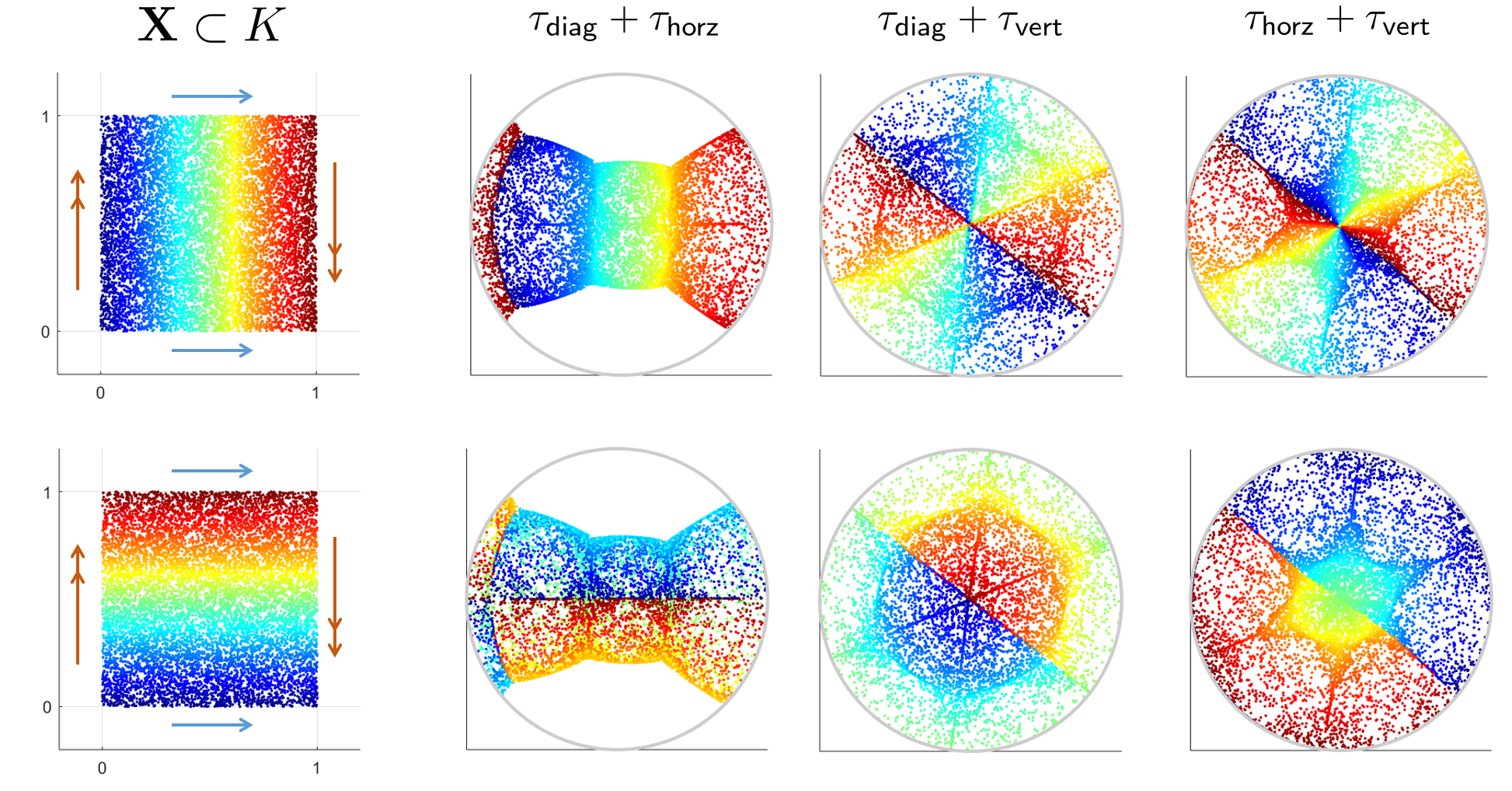}
  \caption{\textbf{Columns:} $\RP^2$ coordinates for $\XX\subset K$ induced by the cocycles $\tau = \tau_\mathsf{diag} + \tau_\mathsf{horz}$, $\tau' = \tau_\mathsf{diag} + \tau_\mathsf{vert}$ and $\tau + \tau' = \tau_\mathsf{horz} + \tau_\mathsf{vert}$, respectively.
\textbf{Rows:} Color schemes of the computed coordinates according
to the horizontal and vertical directions in $K$.
Please refer to an electronic version for colors.}\label{fig:exampel_K_coordinates_RP2a}
\end{figure}

\section{Choosing Cocycle Representatives}
\label{sec:ChoosingCocycles}
We now describe how  $f_\tau: B \longrightarrow \RP^n$
and $f_\eta : B \longrightarrow \CP^n$ depend on the choice of
representatives $\tau = \{\tau_{rt}\} \in Z^1(\mathcal{N}(\mathcal{U});\Z/2)$
and $\eta = \{ \eta_{rst}  \} \in Z^2(\mathcal{N}(\mathcal{U});\Z)$, respectively.
We know that any two such choices yield homotopic maps (Theorem \ref{thm:ClassifyingMapFormula2}),
but intricate geometries can negatively impact the dimensionality reduction step.
Given $\XX = \{\xx_1,\ldots, \xx_N\} \subset B$,
the goal is to elucidate the affects on
the principal projective coordinates of  $f_\tau(\XX) \subset \RP^n$
and $f_\eta(\XX) \subset \CP^n$.
The results are:
the real case is essentially independent of the cocycle representative;
while the complex case requires the harmonic cocycle.

We begin with a simple observation. Let
$\YY = \big\{[\yy_1],\ldots, [\yy_N]\big\}\subset \FP^n$, let
$A$ be a $(n+1)\times (n+1)$ orthogonal matrix with entries in $\F$,
that is $A\in \mathsf{O}(n+1,\F)$,
and let $A\cdot \YY$ denote the set $\big\{[A\cdot \yy_1], \ldots, [A\cdot\yy_N]\big\}$.

\begin{proposition}
Let $A\in \mathsf{O}(n+1,\F)$ and let
 $\YY\subset \FP^n$  be finite. Then
\[
{\tt PrinProjComps}(A\cdot \YY) = A\cdot {\tt PrinProjComps}(\YY)
\]
\end{proposition}
\begin{proof}
Since $A$ is an orthogonal matrix, then
\[
\mathsf{Cov}(A\cdot\YY )
=
A \cdot \mathsf{Cov} (\YY) \cdot A^\dag
\]
Hence, if $\uu$ is an eigenvector of
$\mathsf{Cov}(\YY)$,
then $A\cdot \uu$ is an eigenvector of
$\mathsf{Cov}(A\cdot \YY)$
with the same eigenvalue
and therefore, if
${\tt LastProjComp}\left(\YY,\FP^n\right) = [\vv_n]$,
then
\[
{\tt LastProjComp}\left(A\cdot \YY,\FP^n\right) = [A\cdot\vv_n]
\]
Since the remaining principal projective components are
computed in the same fashion, after the appropriate orthogonal projections,
the result follows.
\end{proof}

\subsection*{The Real Case is Independent of the Cocycle Representative}
Let $\alpha = \{\alpha_r\} \in C^0(\mathcal{N}(\mathcal{U});\Z/2)$ and let
$\gor{\tau} = \tau + \delta^0(\alpha)$.
It follows that for $b\in U_j$
\begin{eqnarray*}
f_{\gor{\tau}}(b) &=&
\left[
(-1)^{\gor{\tau}_{0j}}\sqrt{\varphi_0(b)}
: \cdots :
(-1)^{\gor{\tau}_{nj}}\sqrt{\varphi_n(b)}
\right] \\ \\
&=&
\left[
(-1)^{\tau_{0j} + \alpha_0}\sqrt{\varphi_0(b)}
: \cdots :
(-1)^{\tau_{nj} + \alpha_n}\sqrt{\varphi_n(b)}
\right]
\end{eqnarray*}
Hence, if $f_\tau (\XX) = \big\{[\yy_1],\ldots, [\yy_N]\big\}$ and
\[
A_\alpha = \Mat{(-1)^{\alpha_0} & & \mbox{\Large0} \\  & \ddots& \\ \mbox{\Large0}& & (-1)^{\alpha_n}}
\]
then
$f_{\gor{\tau}}(\XX) = \big\{[A_\alpha\cdot \yy_1] ,\ldots, [A_\alpha\cdot \yy_N]\big\}
= A_\alpha \cdot f_\tau(\XX)$.
This shows that
\[
{\tt PrinProjComps}(f_{\gor{\tau}}(\XX)) =
A_\alpha\cdot {\tt PrinProjComps}(f_\tau(\XX))
\]
which implies, in particular, that the profiles of cumulative variance
for $f_{\gor{\tau}}(\XX)$ and $f_\tau(\XX)$ are identical.
Moreover, the resulting projective coordinates for both point-clouds  differ by the
 isometry of $\FP^n$ induced by $A_\alpha$.
\subsection*{The  Harmonic Representative is Required for the Complex Case}
Just as we did in Section \ref{subsection:Geometric Interpretation}(Geometric Interpretation),
 given $\eta = \{\eta_{rst}\}\in Z^2(\mathcal{N}(\mathcal{U});\Z)$
we can express
$f_\eta : B \longrightarrow \CP^n$ as
\[
\begin{tikzcd}
f_\eta:B\arrow{r}{\varphi} & \left| \mathcal{N}(\mathcal{U})\right| \arrow{r}{F_\eta}& \CP^n
\end{tikzcd}
\]
where $\varphi$ is defined in equation (\ref{eq:phiMapNerveRealization})
and $F_\eta: |\mathcal{N}(\mathcal{U})|\longrightarrow \CP^n$
is given (in barycentric coordinates)
on the open star of a vertex $\vv_j$ by
\[
F_\eta(x_0,\ldots, x_n) =
\left[
\sqrt{x_0}\cdot
e^{2\pi i \sum\limits_t (x_t\cdot \eta_{0jt})}
: \cdots :\sqrt{x_n}\cdot
e^{2\pi i \sum\limits_t (x_t \cdot \eta_{njt})}
\right]
\]
Let us describe the local behavior of $F_\eta$
when restricted to the 2-skeleton of $|\mathcal{N}(\mathcal{U})|$.
To this end, let $\sigma$ be the 2-simplex of $|\mathcal{N}(\mathcal{U})|$
spanned by the vertices $\vv_r, \vv_s, \vv_t$, with $0\leq r < s < t\leq n$.
It follows that $F_\eta : \sigma \longrightarrow \CP^n$  can be written as
\begin{equation}\label{eq:valueOn2Simplex}
F_\eta (x_r,x_s,x_t) =
\Big[
\sqrt{x_r} :
\sqrt{x_s} \cdot e^{-2\pi i\cdot \eta_{rst}\cdot x_t} :
\sqrt{x_t} \cdot e^{2\pi i \cdot \eta_{rst}\cdot x_s}
\Big]
\end{equation}
with the understanding that only   the potentially-nonzero entries appear.
Furthermore, if we fix $0 < c < 1$ and consider the straight  line in
$\sigma$ given by
\[
L_c = \Big\{(1-c,x, c-x) \; : \; 0 \leq x \leq c \Big\}
\]
then $F_\eta:L_c\longrightarrow \CP^n $ can be written as
\begin{eqnarray*}
F_\eta(x) &=&
\Big[
\sqrt{1-c}\; :\; \sqrt{x} \cdot e^{2\pi i (x-c)\eta_{rst}}\; :\; \sqrt{c-x}\cdot e^{2\pi i \cdot x \cdot \eta_{rst}}
\Big] \\
&=&
\Big[
\sqrt{1-c}\cdot e^{-2\pi i \cdot x \cdot \eta_{rst}} \;:\; \sqrt{x}\cdot e^{-2\pi i \cdot c\cdot \eta_{rst}}\; :\;
\sqrt{c - x}
\Big]
\end{eqnarray*}
which parametrizes a spiral with radius $\sqrt{1-c}$
and winding number
$\left\lfloor c\cdot |\eta_{rst}| \right\rfloor$.
Hence, as each $|\eta_{rst}|$ gets larger,
 $F_\eta$ becomes increasingly highly-nonlinear
on the 2-simplices of $|\mathcal{N}(\mathcal{U})|$.
As a consequence, the dimensionality reduction scheme
furnished by principal projective components
is less likely to work as it relies
on a (global) linear approximation.
Let us illustrate this phenomenon via an example.
\\[.2cm]
\noindent\underline{Example:} Let $B = S^2$ be the unit sphere in $\R^3$,
and for $r\in \{0,1,2,3\}$ let $U_r\subset S^2$ be the geodesic open ball of
radius $ \arccos(-1/3)$ centered at
\[
\ell_0 = \Mat{0 \\ 0 \\ 1}  \;\; , \;\;
\ell_1 = \frac{1}{3}\Mat{2\sqrt{2} \\ 0 \\ -1} \;\; , \;\;
\ell_2 = \frac{1}{3}\Mat{ -\sqrt{2} \\ \sqrt{6} \\ -1}  \;\; , \;\;
\ell_3 = \frac{1}{3}\Mat{-\sqrt{2} \\ -\sqrt{6} \\ -1}
\]
respectively.
It follows that $\mathcal{U}= \{U_0,U_1,U_2,U_3\}$ is an open cover of of $S^2$,
and that $\mathcal{N}(\mathcal{U})$ is the boundary of
the 3-simplex.
Therefore,
if
\[
\sigma_0 = \{0,1,2\} \;\; , \;\;
\sigma_1 = \{0,2,3\} \;\; , \;\;
\sigma_2 = \{0,1,3\} \;\; , \;\;
\sigma_3 = \{1,2,3\}
\]
denote the 2-simplices of $\mathcal{N}(\mathcal{U})$ and
$\{\eta^0 ,\eta^1,\eta^2, \eta^3\}$ is the  basis for $C^2(\mathcal{N}(\mathcal{U});\Z)$ of indicator functions,
then each $\eta^r$ is a cocycle whose cohomology class generates
\[H^2(\mathcal{N}(\mathcal{U});\Z) \cong \Z\]
Moreover, $\big\{\eta^0 - \eta^1\;,\; \eta^0 + \eta^2\;,\; \eta^0 + \eta^3 \big\}$
is a basis for
$B^2(\mathcal{N}(\mathcal{U});\Z)$ and therefore
\[
[\eta^0] = [\eta^1] = - [\eta^2] = -[\eta^3]
\]
Consider the map $f_\eta : S^2 \longrightarrow \CP^3$
associated to $\eta = \eta^0$; the results
will be similar for the other $\eta^r$'s.
We show in Figure \ref{fig:example_S2_coordinates_CP1} the computed $\CP^1$ coordinates of
$f_\eta(\XX)$  for a random sample $\XX \subset S^2$
with 10,000 points. As one can see, the homotopy type of the resulting
map is correct, but the distances are completely distorted.

\begin{figure}[!ht]
  \centering
  \includegraphics[width=0.7\textwidth]{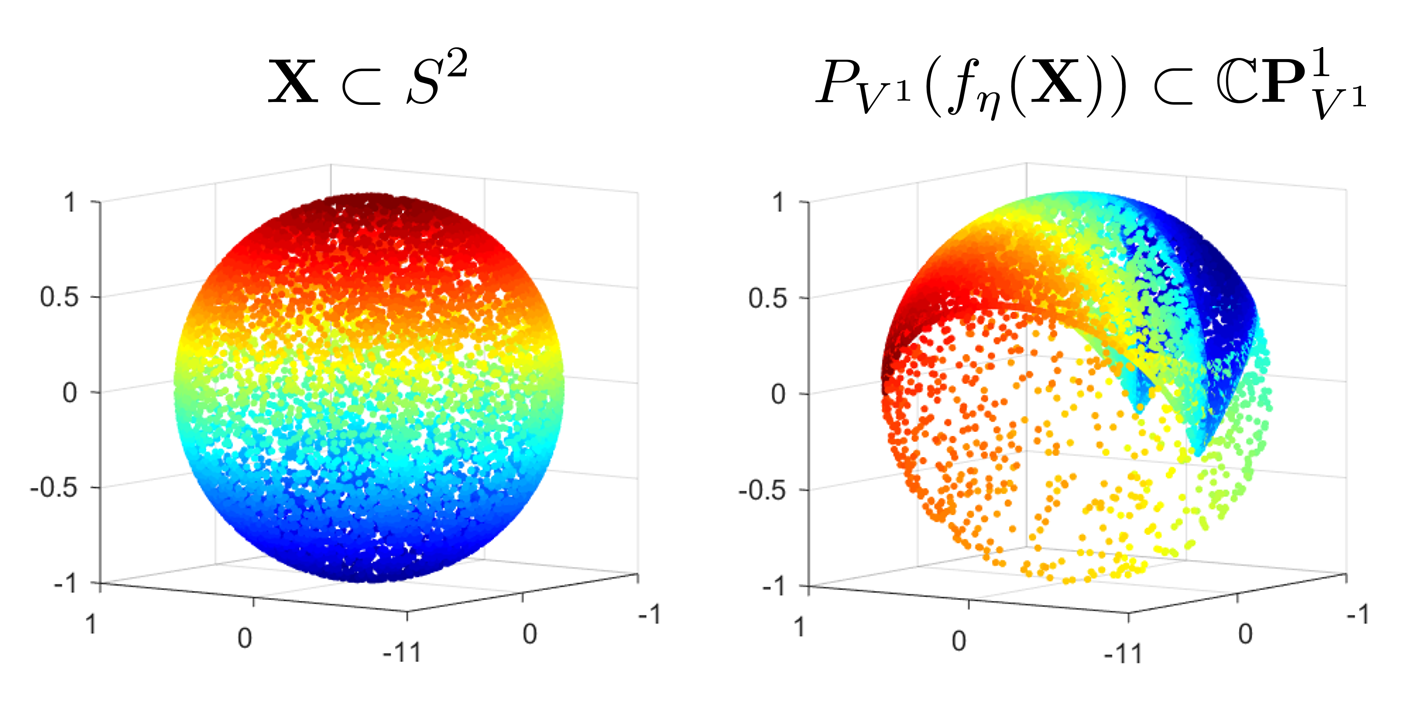}
  \caption{$\CP^1$ coordinates for points on the 2-sphere
  using the map $f_\eta$  associated to the integer cocycle $\eta = \eta^0$.
This shows the inadequacy of the integer cocycle; see Figure \ref{fig:example_S2_coordinates_CP1b} for comparison.
Please refer to an electronic version for colors.}
  \label{fig:example_S2_coordinates_CP1}
\end{figure}

The main difference between the real and  complex cases is that the
former is locally linear, while the latter has  local nonlinearities
arising from the terms
\[\exp\big(2\pi i \cdot
\eta_{rst} \cdot x_t \big).\]
The tempting conclusion would be then to choose the cocycle representative
$\eta = \{\eta_{rst}\} \in Z^2(\mathcal{N}(\mathcal{U});\Z)$
that makes $F_\eta$  as locally linear as possible.
This can be achieved by making each $|\eta_{rst}|$  small.
The problem --- as in the sphere example --- is that since $\eta_{rst} \in \Z$,
then even this choice
is inadequate.
What we will see now is that the integer constraint can be relaxed via
Hodge theory (see, for instance, Section 2 of \cite{lim2015hodge}).

\subsubsection*{Harmonic Smoothing}
Let $K$ be a finite simplicial complex.
Then for each $n\geq 0$ the group of $n$-cochains
$C^n(K;\R)$  is a
finite dimensional vector space over $\R$,
and hence can be endowed with an inner product.
A common choice is
\[
\langle\beta_1,\beta_2\rangle_n = \sum_{\sigma} \beta_1(\sigma)\beta_2(\sigma)
\]
where $\beta_1,\beta_2 \in C^n(K;\R)$ and the sum ranges over all
 $n$-simplices
$\sigma$ of $K$.
In particular we have the induced norm
\[
\|\beta\|^2 = \sum_{\sigma} |\beta(\sigma)|^2
\]
Each boundary map $\delta^n : C^n(K;\R) \longrightarrow C^{n+1}(K;\R)$
is therefore a linear transformation between inner-product spaces, and hence
has an associated dual map
$d_{n+1} : C^{n+1}(K;\R) \longrightarrow C^n(K;\R)$
uniquely determined by the identity
\[
\left\langle\,\delta^n(\nu)\;,\; \beta \,\right\rangle_{n+1}
=
\left\langle\,\nu\;,\; d_{n+1}(\beta)\,\right\rangle_n
\]
for all $\nu \in C^n(K;\R)$ and all $\beta \in C^{n+1}(K;\R)$.
The Hodge Laplacian $\Delta_n$ is the endomorphism of $C^n(K;\R)$
defined by the formula
\[
\Delta_n = d_{n+1}\circ \delta^n \;+\; \delta^{n-1}\circ d_n
\]
and a cochain $\theta \in C^n(K;\R)$ is said to be harmonic
if $\Delta_n(\theta) = 0$.
A simple linear algebra argument shows that
Harmonic cochains can be characterized as follows:
\begin{proposition} $\theta \in C^n(K;\R)$ is harmonic if and only if
\[d_n(\theta) = 0\;\;\; \mbox{ and }\;\;\; \delta^n(\theta)= 0.\]
\end{proposition}
That is, harmonic cochains are in particular  cocycles.
Moreover
\begin{theorem} Every  class $[\beta] \in H^n(K;\R)$
is represented by a unique harmonic cocycle $\theta \in Z^n(K;\R)$
satisfying $\theta = \beta - \delta^{n-1}(\nu^*)$, where
\[
\nu^* = \argmin\big\{ \left\| \beta - \delta^{n-1}(\nu)\right\| : \nu \in C^{n-1}(K;\R) \big\}
\]
\end{theorem}
In other words, given $\beta \in Z^n(K;\R)$, $\theta$ is obtained by
projecting $\beta$ orthogonally onto the orthogonal complement of $B^n(K;\R)$ in $Z^n(K;\R)$.
\\[.2cm]
\noindent Let us now go back to our original set up:
A covering $\mathcal{U} = \{U_r\}$ for a space $B$,
a  partition of unity $\{\varphi_r\}$ dominated by $\mathcal{U}$
and a class $[\eta] \in H^2(\mathcal{N}(\mathcal{U});\Z)$.
The inclusion $\jmath: \Z \hookrightarrow \R$ induces a homomorphism
\[
\jmath^*: H^2(\mathcal{N}(\mathcal{U});\Z) \longrightarrow H^2(\mathcal{N}(\mathcal{U});\R)
\]
and if $\beta\in \jmath^*([\eta])$, then there exists $\nu \in C^1(\mathcal{N}(\mathcal{U});\R)$
so that $\jmath^\#(\eta) = \beta + \delta^1(\nu)$.
\begin{lemma} Let $\omega = \{\omega_{rs}\}$ and
$\gor{\omega} = \{\gor{\omega}_{rs}\}$ be the sets of functions
\[
\begin{array}{cccl}
\omega_{rs} : & U_r \cap U_s & \longrightarrow &\C^\times \\
 & b & \mapsto & \exp\left\{2\pi i \sum\limits _t \varphi_t(b)\eta_{rst}\right\} \\ \\
\gor{\omega}_{rs} : & U_r \cap U_s & \longrightarrow & \C^\times \\
& b & \mapsto & \exp\left\{2\pi i\left(\nu_{rs} + \sum\limits_t \varphi_t(b)\beta_{rst}\right)\right\}
\end{array}
\]
then $\omega,\gor{\omega} \in \check{C}^1(\mathcal{U};\mathscr{C}_\C^\times)$
are cohomologous \v{C}ech cocycles.
\end{lemma}
\begin{proof}
Since $\omega$ is a cocycle, it is enough that check that $\omega$ and $\gor{\omega}$ are cohomologous.
To this end let $\mu = \{\mu_r\}\in \check{C}^0(\mathcal{U};\mathscr{C}_\C^\times)$, where
\[
\mu_r(b) = \exp\left\{2\pi i \sum_t \varphi_t(b)\cdot \nu_{rt}\right\} \;\;\;\; b\in U_r
\]
Since for every $U_r \cap U_t \neq \emptyset$
\begin{eqnarray*}
  \sum_t \varphi_t \cdot n_{rst}&=& \sum_t \varphi_t \cdot \left(\beta_{rst} + \delta^1(\nu)_{rst}\right) \\
   &=& \sum_t \varphi_t \cdot (\beta_{rst} + \nu_{rs} - \nu_{rt} + \nu_{st}) \\
   &=& \nu_{rs} + \sum_t \varphi_t \cdot \beta_{rst} + \sum_t \varphi_t \cdot \nu_{st} - \sum_t \varphi_t \cdot \nu_{rt}
\end{eqnarray*}
then
\[\omega_{rs} = \gor{\omega}_{rs} \cdot \frac{\mu_s}{\mu_r} =  \gor{\omega}_{rs} \cdot \delta^0(\mu)_{rs}\]
and the result follows.
\end{proof}
Let $\theta \in Z^2(\mathcal{N}(\mathcal{U});\R)$
be the harmonic cocycle representing
the class
\[\jmath^*([\eta]) \in H^2(\mathcal{N}(\mathcal{U});\R)\]
let $ \nu \in C^1(\mathcal{N}(\mathcal{U});\R)$ be
so that
$
\jmath^\#(\eta) - \theta = \delta^1(\nu)
$ and let $f_{\theta ,\nu} : B \longrightarrow \CP^n$ be given
on $b\in U_j$
\begin{equation}\label{eq:classMapFormulaHarmonic}
f_{\theta,\nu}(b) =
\left[
\sqrt{\varphi_0(b)}\cdot e^{2\pi i \big(\nu_{0j} + \sum\limits_t \varphi_t(b)\theta_{0jt}\big)}
: \cdots :
\sqrt{\varphi_n(b)}\cdot e^{2\pi i \big(\nu_{nj} + \sum\limits_t \varphi_t(b)\theta_{njt}\big)}
\right]
\end{equation}
It follows that $f_\eta $ and $ f_{\theta,\nu}$ are homotopic,
$f_{\theta,\nu}$ is as locally linear as possible, and
for different choices of $\nu$ the resulting principal projective coordinates of
$f_{\theta,\nu}(\XX)$ differ by a linear (diagonal) isometry.

We now revisit the 2-sphere example.
One can check that
\[
\theta = \jmath^\#\left(\eta^0 + \eta^1 - \eta^2 - \eta^3\right)/4
\]
is the harmonic cocycle representing the cohomology class
\[\jmath^* \left(
\left[ \eta^0 \right]\right) \in H^2(\mathcal{N}(\mathcal{U}); \R)\]
Let $\nu \in C^1(\mathcal{N}(\mathcal{U});\R)$
be so that  $\theta = \jmath^\#(\eta^0) - \delta^1(\nu)$
and let $f_{\theta,\nu} : S^2 \longrightarrow \CP^3$
be as in equation (\ref{eq:classMapFormulaHarmonic}).
We show in Figure \ref{fig:example_S2_coordinates_CP1b} the computed $\CP^1$ coordinates of
$f_{\theta,\nu}(\XX)$  for the finite random sample $\XX \subset S^2$.

\begin{figure}[!ht]
  \centering
  \includegraphics[width= 0.7\textwidth]{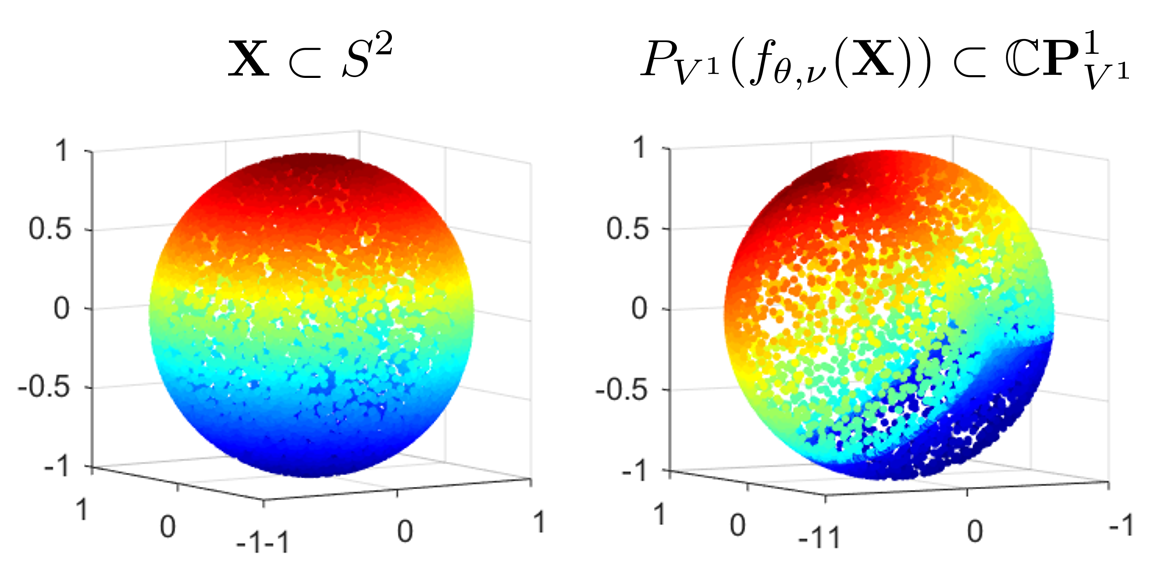}
  \caption{$\CP^1$ coordinates for $\XX \subset S^2$ using the harmonic cocycle $\theta$;
see  Figure \ref{fig:example_S2_coordinates_CP1} for comparison. Please refer to an electronic version for colors.
}\label{fig:example_S2_coordinates_CP1b}
\end{figure}

\section{Multiscale Projective Coordinates
via Persistent Cohomology of Sparse Filtrations}
\label{sec:TransitionFuncFromPersistentCohomology}
The goal of this section is to show how one can use persistent cohomology
to construct multiscale projective coordinates.

\subsection*{Greedy Permutations}
Let $\underline{k} = \{0,\ldots,k\}$ for  $k\in \Z_{\geq0}$,
let $(\mathbb{M},\mathbf{d})$ be a metric space and let $X\subset \mathbb{M}$ be
a finite subset with $n+1$ elements.
A greedy permutation on $X$ is a bijection
$\sigma_g: \underline{n} \longrightarrow X$
which satisfies
\[
\sigma_g(s+1) = \argmax\limits_{x\in X} \mathbf{d} \big(x, \sigma_g\left(\underline{s}\right)\big)
\;\;,\;\;\; s =0,\ldots, n-1
\]
\subsection*{Sparse Filtrations}
Given a greedy permutation $\sigma_g : \underline{n} \longrightarrow X $,
let $x_s = \sigma_g(s)$
and $X_s = \sigma_g\left(\underline{s}\right)$.
The \emph{insertion radius} of $x_s$,
denoted  $\lambda_s$, is defined as
\[
\lambda_s = \left\{
\begin{array}{ccc}
\infty & \mbox{ if } & s = 0 \\[.2cm]
\mathbf{d}\big(x_s , X_{s-1}\big) & \mbox{ if } & s>0
\end{array} \right.
\]
If follows that
$\infty =\lambda_0 > \lambda_1 \geq  \cdots \geq \lambda_n $ .  Fix $0 < \epsilon < 1$ and for $\alpha \geq 0$ define
\begin{equation}\label{eq:defR_s}
r_s(\alpha) =
\left\{
  \begin{array}{ccc}
    \alpha & \hbox{ if } & \alpha < \lambda_s (1+\epsilon)/\epsilon\\[.2cm]
    \lambda_s(1 + \epsilon)/\epsilon & \hbox{ if }& \lambda_s(1+\epsilon)/\epsilon \leq \alpha \leq \lambda_s(1+\epsilon)^2/\epsilon \\[.2cm]
    0 & \hbox{ if }& \alpha >\lambda_s(1+\epsilon)^2/\epsilon
  \end{array}
\right.
\end{equation}
In particular $r_0(\alpha) = \alpha$ for all $\alpha \geq 0$, and
for $s\geq 1$ the graph of $r_s$ is shown below.
\begin{figure}[!ht]
  \centering
  \includegraphics[scale =.4]{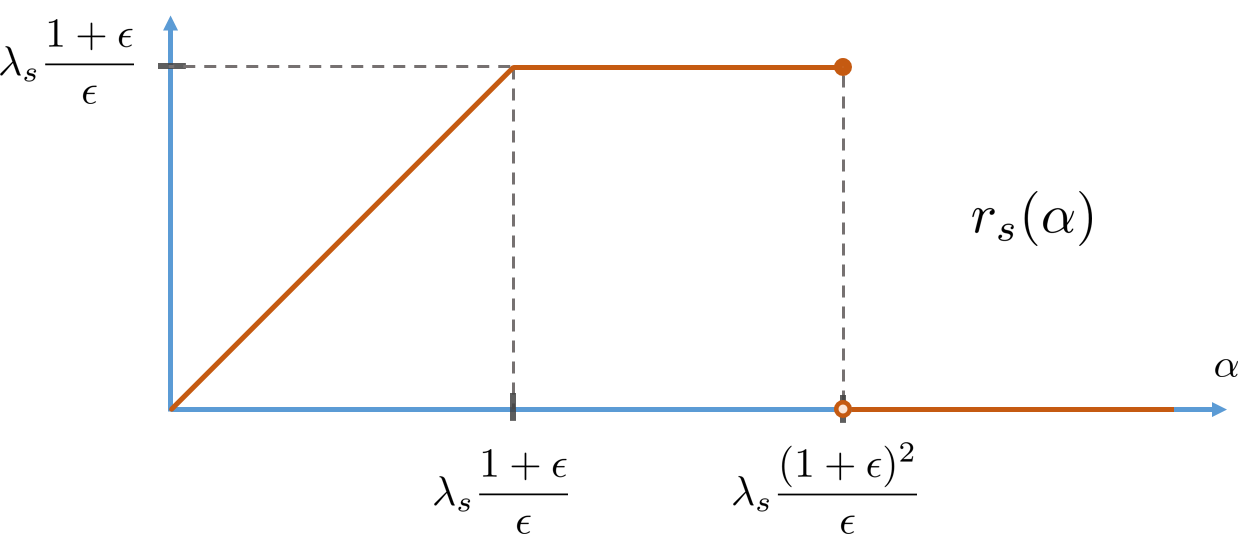}
  \caption{Function $r_s$}
\end{figure}

\begin{definition}
For $\alpha \geq0$ let
\[
B_s^\alpha = \{b\in \mathbb{M} \; : \; \mathbf{d}(b,x_s) < r_s(\alpha)\}
\]
\end{definition}
It follows thats
$S_\alpha = \left\{ s\in \underline{n} \; : \; \lambda_s \geq \epsilon\alpha/(1+\epsilon)^2\right\}$
is the collection of indices $s$ for which $B_s^\alpha \neq \emptyset$.
Moreover,
\[\mathcal{B}^\alpha = \{B_s^\alpha\; : \; s \in S_\alpha\}\]
satisfies $X \subset \bigcup \mathcal{B}^\alpha$
for each $\alpha >  0$, and
it is a sparse covering in the sense that as $\alpha$ increases there are  fewer  balls in
$\mathcal{B}^\alpha$,
but of larger radii.
Moreover, it is a $(1+\epsilon)$-approximation of the $\alpha$-offset
$X^\alpha = \{b\in \mathbb{M} \; : \; \mathbf{d}(b,X) < \alpha\}$:
\begin{proposition}[\cite{cavanna2015geometric}, Corollary 2] If $\beta \geq  (1+\epsilon)\alpha$, then
\[
\bigcup \mathcal{B}^\alpha \subset X^\alpha \subset \bigcup \mathcal{B}^\beta
\]
\end{proposition}

Let
\[
U^\alpha_s
=
\bigcup_{0 \leq \lambda \leq \alpha} \big( B_s^\lambda \times \{\lambda\} \big)
\;\;\;
\mbox{ and }
\;\;\;
\mathcal{U}^\alpha = \left\{ U_s^\alpha \; : \; s \in \underline{n}\right\}
\]
It follows that $\mathcal{N}(\mathcal{U}^\alpha)\subset \mathcal{N}(\mathcal{U}^\beta)$
whenever $\alpha \leq \beta$.

\begin{definition}
The sparse \v{C}ech filtration with sparsity parameter $0<\epsilon<1$, induced
by the greedy permutation $\sigma_g : \underline{n} \longrightarrow X$, is the filtered simplicial complex
\[
\check{\mathcal{C}}(\sigma_g, \epsilon)
 =
\left\{
\check{C}_\alpha(\sigma_g, \epsilon)\;: \;\alpha \geq  0
\right\}
\]
where
\[
\check{\mathcal{C}}_\alpha(\sigma_g, \epsilon) =
\mathcal{N}\left(\mathcal{U}^\alpha\right)
\]
\end{definition}

\subsection*{Multiscale Projective Coordinates}
We will show now how the persistent cohomology of  $\check{\mathcal{C}}(\sigma_g , \epsilon)$
can be used to compute multiscale compatible
classifying maps.
The first thing to notice is that projection onto the first coordinate
\[
\begin{array}{ccc}
\bigcup \mathcal{U}^\alpha &\longrightarrow &\bigcup\mathcal{B}^{\alpha} \\[.1cm]
(b,\lambda) & \mapsto & b
\end{array}
\]
is a deformation retraction if one regards $B_s^\alpha$ as a subset
of $U_s^{\alpha}$ via the inclusion
\begin{equation}\label{eq:Inlcusion}
\begin{array}{ccc}
B_s^\alpha & \hookrightarrow & U^\alpha_s \\
b &\mapsto & (b,\alpha)
\end{array}
\end{equation}

\begin{theorem}\label{thm:ClassifyingFormulaSparse}
Let $(\mathbb{M}, \mathbf{d})$ be a metric space and  let $X\subset \mathbb{M}$ be
a subset with $n+1$ points.
Given a greedy permutation  $\sigma_g : \underline{n} \longrightarrow X$ and
a sparsity parameter $0 < \epsilon < 1$,
let $r_s(\alpha)$ for $\alpha \geq 0$ be as in equation (\ref{eq:defR_s}).

If $\check{\mathcal{C}}(\sigma_g, \epsilon)$ is the resulting sparse \v{C}ech filtration and
\[
\varphi^\alpha_s(b) =
\frac
{\ppart{r_s(\alpha) - \mathbf{d}(b,x_s)}^2}
{\sum\limits_{t\in \underline{n}} \ppart{r_t(\alpha) - \mathbf{d}(b,x_t)}^2}
\;\;\;\;\;\; \mbox{ for } \;\;\;\;\;\;\;  s\in \underline{n}\;\; , \;\; b\in \bigcup\mathcal{B}^\alpha
\]
then we have well-defined maps
\[
\begin{array}{rccl}
\ww^\alpha_1 : &
H^1\left(\check{\mathcal{C}}_\alpha(\sigma_g, \epsilon);\Z/2\right) &
\longrightarrow &
 \left[\bigcup \mathcal{B}^\alpha\, , \, \RP^\infty \right] \\[.2cm]
&[\tau= \{\tau_{rt}\}] &  \mapsto & [f^\alpha_\tau] \\ \\
f^\alpha_\tau : &\hspace{-1.5cm} \bigcup \mathcal{B}^\alpha & \hspace{-2.5cm}\longrightarrow &\hspace{-1cm} \RP^n \\
&\hspace{-1.5cm} B^\alpha_j \ni b & \hspace{-2.5cm} \mapsto &\hspace{-1cm} \left[(-1)^{\tau_{0j}}\sqrt{\varphi^\alpha_0(b)}: \cdots :  (-1)^{\tau_{nj}}\sqrt{\varphi^\alpha_n(b)} \right]
\end{array}
\]
and
\[
\begin{array}{rccl}
\mathbf{c}^\alpha_1 : &
H^2\left(\check{\mathcal{C}}_\alpha(\sigma_g,\epsilon);\Z\right) &
\longrightarrow & \left[\bigcup \mathcal{B}^\alpha\, , \, \CP^\infty \right]\\[.2cm]
&\left[\eta= \theta + \delta^1(\nu)\right] &\mapsto & [f^\alpha_{\theta,\nu}] \\ \\
f^\alpha_{\theta,\nu} : & \hspace{-1.9cm}\bigcup \mathcal{B}^\alpha & \hspace{-3.8cm}\longrightarrow & \hspace{-1.8cm}\CP^n \\
& \hspace{-2.4cm}B^\alpha_j \ni b & \hspace{-4.1 cm} \mapsto & \hspace{-2.25cm} \left[e^{2\pi i
( \nu_{0j} + \sum\limits_t \varphi^\alpha_t(b) \theta_{0jt} )}
\sqrt{\varphi^\alpha_0(b)}
 : \cdots :
e^{2\pi i
( \nu_{nj} + \sum\limits_t \varphi^\alpha_t(b)\theta_{njt} )
}
\sqrt{\varphi^\alpha_n(b)} \right]
\end{array}
\]
where $\theta = \{\theta_{rst}\}
\in Z^2\left(\check{C}_\alpha(\sigma_g,\epsilon); \R\right)$ is the
harmonic cocycle representing
$\jmath^*([\eta]) \in H^2\left(\check{C}_\alpha(\sigma_g,\epsilon);\R\right)$
and
$\nu = \{\nu_{rt}\} \in
C^1\left(\check{C}_\alpha(\sigma_g,\epsilon);\R\right)
$
is so that $\theta = \jmath^\#(\eta) - \delta^1(\nu)$.
Moreover,
if each $B^\alpha_s  \in \mathcal{B}^\alpha$ is  connected,
then
$
\ww_1^\alpha$
is injective;
if in addition each $B^\alpha_s$ is locally path-connected  and simply connected, and each $B^\alpha_r \cap B^\alpha_t$ is either empty or connected, then
$
\mathbf{c}_1^\alpha
$
is injective.
\end{theorem}
\begin{proof}
The first thing to notice is that the collection of continuous maps
\[
\begin{array}{rccl}
 \varphi_s& : \bigcup \mathcal{U}^\alpha& \longrightarrow &  \R \\
  &(b,\lambda) & \mapsto & \varphi_s^\lambda(b)
\end{array} \;\;\;\;\; , \;\;\;\; s\in \underline{n}
\]
is a partition of unity dominated by $\mathcal{U}^\alpha$.
The Theorem follows from combining Theorem \ref{thm:ClassifyingMapFormula2} and the following two facts:
the inclusion  $\bigcup \mathcal{B}^\alpha \hookrightarrow \bigcup
\mathcal{U}^\alpha$ from Equation (\ref{eq:Inlcusion})
induces a bijection
\[\left[\bigcup
\mathcal{U}^\alpha, \FP^\infty\right] \longrightarrow
\left[\bigcup
\mathcal{B}^\alpha, \FP^\infty\right]\]
and  the necessary connectedness conditions are
satisfied by $\mathcal{U}^\alpha$ if they are satisfied
by $\mathcal{B}^\alpha$.
\end{proof}

\begin{remark} As $\alpha$ increases, the number of
potentially nontrivial dimensions in the images of
$f_\tau^\alpha$ and $f_{\theta,\nu}^\alpha$ decrease.
Indeed,  since $\varphi_s^\alpha$ is identically zero if and only if $B_s^\alpha = \emptyset$,
it follows that for any $b\in \bigcup\mathcal{B}^\alpha$
 the only potentially non-zero entries in either  $f^\alpha_\tau(b)$ or $f^\alpha_{\theta,\nu}(b)$
correspond to the indices in
\[
S_\alpha = \{s \in \underline{n} \; : \; \lambda_s \geq \epsilon \alpha/ (1+\epsilon)^2\}\]
The observation follows from the fact that
the sequence $\{\lambda_s\}_{s\in \underline{n}}$ is non-increasing,
and monotonically decreasing for generic $X$.
\end{remark}

\begin{proposition}\label{prop:Naturality}
Let $ \alpha \leq \beta $, then the diagrams
\[
\begin{tikzcd}[column sep = 1.5em]
H^1\left(\check{\mathcal{C}}_\alpha(\sigma_g,\epsilon);\Z/2\right) \arrow{r}{\ww_1^\alpha} &
\left[\bigcup\mathcal{B}^\alpha, \RP^\infty\right] \\
H^1\left(\check{\mathcal{C}}_\beta(\sigma_g,\epsilon);\Z/2\right) \arrow[swap]{r}{\ww_1^\beta}\arrow{u} &
\left[\bigcup\mathcal{B}^\beta, \RP^\infty\right]\arrow{u}
\end{tikzcd}
\begin{tikzcd}[column sep = 1.4em]
H^2\left(\check{\mathcal{C}}_\alpha(\sigma_g,\epsilon);\Z\right) \arrow{r}{\mathbf{c}_1^\alpha} &
\left[\bigcup\mathcal{B}^\alpha, \CP^\infty\right] \\
H^2\left(\check{\mathcal{C}}_\beta(\sigma_g,\epsilon);\Z\right) \arrow[swap]{r}{\mathbf{c}_1^\beta}\arrow{u} &
\left[\bigcup\mathcal{B}^\beta , \CP^\infty\right]\arrow{u}
\end{tikzcd}
\]
are commutative.
\end{proposition}

As a consequence, if for $\alpha = \alpha_0 < \cdots < \alpha_{\ell-1} < \alpha_\ell = \beta$
one has classes
\[
\begin{tikzcd}[column sep = scriptsize, row sep = .1em]
H^2\left(\check{\mathcal{C}}_{\alpha_\ell}(\sigma_g,\epsilon);\Z\right)
\arrow{r}&
 H^2\left(\check{\mathcal{C}}_{\alpha_{\ell-1}}(\sigma_g,\epsilon);\Z\right)
\arrow{r}&
\cdots
\arrow{r}& \hspace{-.05cm}
H^2\left(\check{\mathcal{C}}_{\alpha_0}(\sigma_g,\epsilon);\Z\right) \\
{ }[\eta_\ell]
\arrow[mapsto, shorten <= 3em, shorten >= 3em]{r}&
{ }[\eta_{\ell-1}]
\arrow[mapsto, shorten <= 3em]{r}&
\cdots
\arrow[mapsto, shorten >= 3em]{r}&
{ }[\eta_0]
\end{tikzcd}
\]
then the diagram
\[
\begin{tikzcd}
\bigcup \mathcal{B}^{\alpha_0}
\arrow{rrd}[below]{f^{\alpha_0}_{\eta_0}}
\rar[hook]&
\bigcup \mathcal{B}^{\alpha_1}
\arrow{rd}[near end]{f^{\alpha_1}_{\eta_1}}
\rar[hook]&
\cdots
\rar[hook]&
\bigcup \mathcal{B}^{\alpha_\ell}
\arrow{ld}{f^{\alpha_\ell}_{\eta_\ell}} \\
{}&{}& \CP^n
\end{tikzcd}
\]
commutes up  to a homotopy which perhaps takes place in a higher
dimensional projective space.
The same is true in dimension one with $\Z/2$ coefficients.
The persistent cohomology
of the sparse \v{C}ech filtration
$\check{\mathcal{C}}(\sigma_g, \epsilon)$ now
becomes relevant:
over $\Z/2$,
a 1-dimensional cohomology class with nonzero persistence yields
a multiscale system of compatible (up to homotopy)
$\RP^n$ coordinates.
Constructing multiscale $\CP^n$ coordinates
from a persistent cohomology computation for $k=2$ requires a bit more work, as the
barcode decomposition is not valid for
integer coefficients.
Let $p$ be a prime and consider the short exact sequence
of abelian groups
\[
\begin{tikzcd}[column sep = scriptsize]
0 \arrow{r} &\Z \arrow{r}{\times p} &\Z \arrow{r}& \Z/p \arrow{r} & 0
\end{tikzcd}
\]
The induced homomorphism
\[
H^2\left(\mathcal{\check{C}}_\alpha(\sigma_g, \epsilon);\Z\right) \longrightarrow H^2\left(\mathcal{\check{C}}_\alpha(\sigma_g, \epsilon);\Z/p\right)
\]
 will be an epimorphism whenever
$H^3\left(\mathcal{\check{C}}_\alpha(\sigma_g, \epsilon);\Z\right)$ has no $p$-torsion.
The universal coefficient theorem implies the following
\begin{proposition}
Let  $p$ be a prime
not dividing the order of the torsion subgroup of $H_2(\mathcal{\check{C}}_\alpha(\sigma_g, \epsilon);\Z)$.
Then the homomorphism
\[
H^2(\mathcal{\check{C}}_\alpha(\sigma_g, \epsilon);\Z) \longrightarrow H^2(\mathcal{\check{C}}_\alpha(\sigma_g, \epsilon);\Z/p)
\]
is surjective.
\end{proposition}

Now one can follow the strategy in
\cite[Sections 2.4 and 2.5]{de2011persistent}
for choosing $p$, lifting to integer coefficients
and constructing the harmonic representative.
The solution to the harmonic representative problem
is plugged into equation (\ref{eq:classMapFormulaHarmonic}).

\begin{definition}
The sparse Rips filtration, with sparsity parameter $0<\epsilon<1$, induced
by the greedy permutation $\sigma_g : \underline{n} \longrightarrow X$ is the filtered simplicial complex
\[
\mathcal{R}(\sigma_g, \epsilon)
 =
\big\{
R_\alpha(\sigma_g, \epsilon)\;: \;\alpha \geq  0
\big\}
\]
where
\[
R_\alpha(\sigma_g,\epsilon) =
\big\{
\{s_0,\ldots, s_k\} \subset \underline{n} \;:\;
U_{s_r}^\alpha \cap U_{s_t}^\alpha \neq \emptyset \mbox{ for all }
0\leq r,t \leq k
\big\}
\]
\end{definition}

\begin{remark}
It follows  that
for  all $\beta \geq 0$
\[
\check{C}_\beta(\sigma_g , \epsilon)
\subset
R_\beta(\sigma_g, \epsilon)
\]
and if $\alpha$ is small enough (as the cones $U_s^\alpha$ stop growing)  we also get the inclusion
$R_\alpha(\sigma_g, \epsilon)
\subset
\check{C}_{2\alpha}(\sigma_g,\epsilon).$
Then for each abelian group $G$ and integer $k\geq 0$ we get a commutative diagram
\[
\begin{tikzcd}
H^k\left(R_{\beta}(\sigma_g , \epsilon); G\right)
\arrow{r}
\arrow{d}&
H^k\left(R_{\alpha/4}(\sigma_g, \epsilon); G\right)  \\
H^k\left(\check{C}_{\beta}(\sigma_g , \epsilon); G\right)
\arrow{r}&
H^k\left(\check{C}_{\alpha/2}(\sigma_g, \epsilon); G\right)
\arrow{u}
\end{tikzcd}
\]
which shows that cohomology classes in the sparse Rips filtration
with long enough persistence, and small enough death time, yield nontrivial persistent cohomology classes in
the sparse \v{C}ech filtration.
This is useful because the persistent cohomology of the sparse Rips filtration is easier to compute in practice.
\end{remark}

\subsection*{Example:} Let $\XX$ be a uniform random sample with
$2,500$ points from the 2-dimensional torus
$S^1\times S^1 \subset \C^2$, endowed with the metric $\mathbf{d}$
given by
\[
\mathbf{d}\big((z_1,w_1) ,(z_2, w_2)\big)
=
\sqrt{\big|\arccos(\langle z_1,z_2\rangle)\big|^2 + \big|\arccos(\langle w_1,w_2\rangle)\big|^2 }
\]
There are two things we would like to illustrate with this example:
First, that one does not need the entire data set $\XX$ to compute
appropriate classifying maps $f^\alpha_\tau$, in fact a small subsample suffices;
and second, that one  can use the sparse Rips filtration instead of
the \v{C}ech filtration, which simplifies computations.
Indeed, let $n= 34$ and let \[X = \{x_0,\ldots, x_n\}\subset \XX\] be obtained through \verb"maxmin" sampling.
Notice that $X$ is  $1.4\%$ of the total size of $\XX$
and that  $\sigma_g: \underline{n} \longrightarrow X$ given by
$\sigma_g(s) = x_s$ is a greedy permutation on $X$.
We let $\epsilon = 0.01$ since the sample is already sparse.
Computing the 1-dimensional persistent cohomology with coefficients in
$\Z/2$ for the sparse Rips filtration $\mathcal{R}(\sigma_g, \epsilon)$,  yields the barcode shown in Figure \ref{fig:example_T_sparse}(\textbf{left}).

\begin{figure}[!ht]
  \centering
  \includegraphics[width=\textwidth]{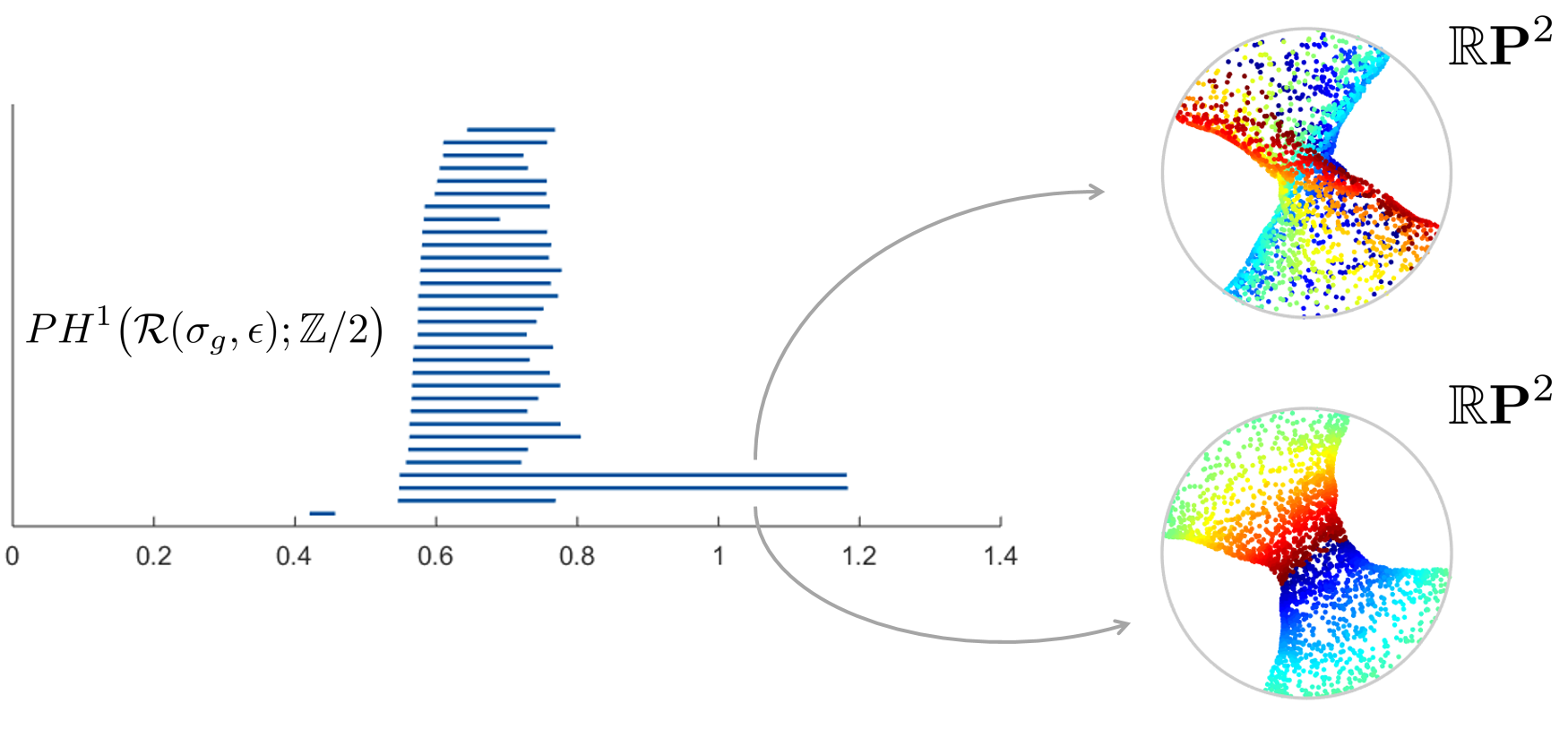}
  \caption{$\RP^2$ coordinates for $\XX\subset S^1 \times S^1$,
from the 1-dimensional $\Z/2$-persistent cohomology of a sparse rips filtration.
\textbf{Left:} Computed barcode, \textbf{Right:} resulting $\RP^2$ coordinates induced by classes with large persistence. In both cases, the $\RP^2$ coordinates of a point $(z,w) \in \XX$
are colored according to $\arg(z)\in [0,2\pi)$.
Please refer to an electronic version for colors.
}\label{fig:example_T_sparse}
\end{figure}

For this calculation we first determine the birth-times of the edges
as in \cite[Algorithm 3]{cavanna2015geometric}, and input them as a distance matrix into \verb"Dionysus"'
persistent cohomology algorithm \cite{dionysus2012}.
After selecting the two classes with the longest persistence,
\verb"Dionysus" outputs cocycle representatives $\mu_1$ and $\mu_2$
at cohomological birth $\alpha_1, \alpha_2 \approx 1.18$.
Now, using the fact that
$  \check{C}_\alpha(\sigma_g, \epsilon) \subset R_\alpha(\sigma_g, \epsilon)$
for all $\alpha \geq 0$,
we have that the induced homomorphism
\[C^1 (R_\alpha(\sigma_g, \epsilon) ; \Z/2) \longrightarrow
C^1(\check{C}_\alpha(\sigma_g, \epsilon);\Z/2)
\]
sends $\mu_1$ and $\mu_2$ to $\tau_1$ and $\tau_2$, respectively.
Moreover, since $R_\alpha(\sigma_g, \epsilon)$ is connected
at $\alpha = \min\{\alpha_1,\alpha_2\}$ it follows that
$\XX \subset \bigcup \mathcal{B}^{\alpha}$,
and using the formula from Theorem \ref{thm:ClassifyingFormulaSparse}
we get the point clouds
$f_{\tau_1}^{\alpha}(\XX), f_{\tau_2}^{\alpha}(\XX) \subset \RP^{34}$.
The result of computing their $\RP^2$ coordinates via principal
projective components is shown in Figure \ref{fig:example_T_sparse}(\textbf{right}).

\section{Discussion}
We have shown in this paper how 1-dimensional (resp. 2-dimensional)
persistent cohomology classes with $\Z/2$ coefficients
(resp. $\Z/p$ coefficients for appropriate primes $p$)
can be used to produce multiscale projective coordinates for data.
The main ingredients were: interpreting a given cohomology class
as the characteristic class corresponding to a unique isomorphism
type of line bundle, and constructing explicit classifying maps
from \v{C}ech cocycle representatives.
In addition, we develop a dimensionality reduction step in projective
space in order to lower the target dimension of the original classifying map.

Some questions/directions suggested by the current approach are the following: The case $H^3(B;\Z)$ has a similar flavor to the bundle
perspective presented here, and can perhaps be addressed using gerbes
\cite{hitchin2003communications}.
On the other hand, since Principal Projective Components   is essentially a global fitting procedure, it would be valuable to investigate what local nonlinear dimensionality reduction techniques can be adapted to projective space.

\subsection*{Acknowledgements}
The author  would like to thank Nils Baas, Ulrich Bauer, John Harer, Dmitriy Morozov and Don Sheehy for extremely helpful conversations regarding the contents of this paper.
The detailed comments of the anonymous reviewers were invaluable in fixing multiple imprecisions found in the original draft;  thank you for the high-quality feedback.

\bibliography{references}

\end{document}